\documentclass[oneside,leqno]{amsart}

\usepackage[top=2.5cm,bottom=2cm,left=2cm,right=3cm]{geometry}
\usepackage[utf8]{inputenc}
\usepackage[T1]{fontenc}
\usepackage{lmodern}
\usepackage[english]{babel}
\usepackage{microtype}

\usepackage{amsmath}
\usepackage{amssymb}
\usepackage{mathtools}
\usepackage{amsthm}
\usepackage{stmaryrd}
\usepackage{chngcntr}

\usepackage{enumitem}
\usepackage{multirow}
\usepackage{booktabs}
\usepackage{nicefrac}

\usepackage{graphicx}
\usepackage[font=small,labelfont=bf]{caption}
\usepackage{subcaption}

\usepackage{clrscode3e}
\usepackage{etoolbox}
\usepackage{marginnote}
\usepackage{xcolor}
\usepackage[textsize=tiny]{todonotes}
\usepackage{csquotes}

\usepackage{scalerel,stackengine}

\usepackage[
  colorlinks=true,
  linkcolor=blue,
  citecolor=blue,
  urlcolor=blue,
  backref=page
]{hyperref}

\usepackage{cleveref}

\newcommand*{\mailto}[1]{\href{mailto:#1}{\nolinkurl{#1}}}

\newcommand{\Dx}{{\Delta x}}
\newcommand{\Dt}{{\Delta t}}
\newcommand{\R}{\mathbb{R}}
\newcommand{\E}{\mathbb{E}}
\renewcommand{\P}{\mathbb{P}}

\newcommand{\norm}[1]{\left\|#1\right\|}

\newcommand{\F}{\mathcal{F}}

\newcommand{\T}{\mathcal{T}}

\apptocmd{\lim}{\limits}{}{}

\newcommand{\eps}{\varepsilon}

\newcommand{\Div}{\operatorname{div}}

\newcommand*{\ldblbrace}{\{ \mskip-5mu \{}
\newcommand*{\rdblbrace}{\} \mskip-5mu \}}
\newcommand*{\avg}[1]{\ldblbrace #1 \rdblbrace}

\DeclareFontFamily{U}{mathx}{}
\DeclareFontShape{U}{mathx}{m}{n}{<-> mathx10}{}
\DeclareSymbolFont{mathx}{U}{mathx}{m}{n}
\DeclareMathAccent{\widecheck}{0}{mathx}{"71}

\let\svwidehat\widehat
\renewcommand\widehat[1]{%
  \ifx\relax#1\relax
    \svwidehat{\phantom{x}}%
  \else
    \svwidehat{#1}%
  \fi
}

\newcommand\pig[1]{%
  \scalerel*[5pt]{\big#1}{%
    \ensurestackMath{\addstackgap[1.5pt]{\big#1}}%
  }%
}
\newcommand\pigl[1]{\mathopen{\pig{#1}}}
\newcommand\pigr[1]{\mathclose{\pig{#1}}}

\def\div{\operatorname{div}}

\newlist{thmlist}{enumerate}{1}
\setlist[thmlist]{
  label=(\roman{thmlisti}),
  noitemsep
}

\numberwithin{equation}{section}

\theoremstyle{plain}
\newtheorem{definition}{Definition}[section]
\newtheorem{remark}[definition]{Remark}

\newtheorem{example}[definition]{Example}

\newtheorem{assumption}[definition]{Assumption}
\newtheorem{theorem}[definition]{Theorem}
\newtheorem{corollary}[definition]{Corollary}
\newtheorem{lemma}[definition]{Lemma}

\begin{document}

\title[Structure-preserving LDG methods]
{Structure-preserving LDG methods
for linear and nonlinear transport equations
with gradient noise}

\author[T. Christiansen]{T. Christiansen}
\author[K. H. Karlsen]{K. H. Karlsen}

\address[Thomas Christiansen and Kenneth H. Karlsen]
{\newline Department of Mathematics
\newline University of Oslo
\newline P.O. Box 1053,  Blindern
\newline N–0316 Oslo, Norway}
\email[]{thomchr@math.uio.no, kennethk@math.uio.no}

%

\subjclass[2020]{
  Primary: 60H15, 35L65;
  Secondary: 65M60, 65C30
}

\keywords{
  Stochastic conservation laws,
  stochastic fluxes,
  gradient noise,
  Stratonovich--It\^o correction,
  local discontinuous Galerkin methods,
  energy stability,
  numerical experiments
}

\thanks{
  This work was funded by the Research Council of Norway
  under project 351123 (NASTRAN)
}

\raggedbottom
\allowdisplaybreaks

\date{\today}

\begin{abstract}
We develop local discontinuous Galerkin (LDG)
methods for conservation laws with
heterogeneous stochastic fluxes,
where the Stratonovich-driven transport terms may be linear
or nonlinear. Such equations arise, for example,
in simplified turbulence models, mean field games, and
fluctuating hydrodynamics.
Starting from the It\^{o} formulation, we construct
semi-discretizations that build the cancellation
mechanism of transport noise into the numerical method.
At the discrete energy level, the second-order
Stratonovich--It\^{o} correction is balanced by the
quadratic variation, up to numerical flux terms, so that
the hyperbolic stability structure is retained.
Suitable numerical fluxes yield discrete energy
conservation or energy dissipation, valid
either pathwise or in expectation. The resulting
high-order schemes are proved well posed through
stability estimates combined with a
Khasminskii-type argument, without imposing
linear growth assumptions. Numerical
experiments confirm stability
and high-order accuracy.
\end{abstract}

\maketitle

\setcounter{tocdepth}{1}


{\small 
\begingroup
\makeatletter
\let\oldcontentsline\contentsline
\renewcommand{\contentsline}[4]{%
  \oldcontentsline{#1}{#2}{}{#4}%
}
\tableofcontents
\makeatother
\endgroup}

\section{Introduction}
	
Stochastic perturbations of partial 
differential equations (PDEs) 
arise in several meaningful ways. At a modeling level, 
they capture fluctuations generated by unresolved scales, 
random environmental forcing, or intrinsic microscopic effects. 
Classical examples involve It\^o-type forcing terms, 
which model direct additive or 
multiplicative disturbances 
(see, e.g., \cite{Breit:2018aa,Chow:2015aa,DaPrato:2014aa}). 
In many physical applications—particularly in fluid 
dynamics and transport—the noise instead enters in a 
more structured manner through spatial 
gradients, leading to so-called gradient (or transport) noise.
In particular, gradient noise 
perturbations have been promoted as a natural description 
of the impact of small-scale randomness 
on macroscopic transport phenomena, 
a viewpoint emphasized in the works 
\cite{Holm:2015tc,Mikulevicius:2004aa} 
on stochastic fluid dyanamics. 
The inclusion of such transport noise is 
further motivated by its connection to 
turbulence and enhanced diffusion and dissipation in 
deterministic counterparts, see the 
recent book \cite{FlandoliLuongo2023} 
(and references therein). 

In this paper we investigate both linear 
and nonlinear stochastic partial differential 
equations (SPDEs) with gradient noise. Our goal is to 
develop high-order numerical methods that 
preserve the underlying (stochastic) hyperbolic 
structure of these equations. 
The SPDEs under consideration are given by
\begin{equation}\label{eq:SBL}
	\partial_t u+\sum_{\ell\in L}
	\div_x\bigl(\sigma_\ell(x)g_\ell(u)\bigr)
	\circ \partial_t W^\ell(t)=0,
\end{equation}
for $(t,x)\in (0,T)\times \R^d$, with $T>0$ 
and $d\ge 1$. Here $u=u(t,x)$ 
is the unknown random field (solution). 
The stochastic perturbation enters through
Stratonovich transport (gradient) noise,
where each $\sigma_\ell:\R^d\to\R^d$ is a 
prescribed spatial vector field
and $g_\ell:\R\to\R$ is a 
scalar (linear or nonlinear) flux function.

In dimensions $d\ge 2$,  we often assume 
that the vector fields $\sigma_\ell$ 
are either divergence-free (as in \cite{FlandoliLuongo2023} 
and many related works) or that 
each $\sigma_\ell$ has sufficiently 
many bounded derivatives to close 
the stability estimates. 

The driving processes $\{W^\ell\}_{\ell\in L}$ 
are independent real-valued Wiener processes 
defined on a complete filtered probability space 
$(\Omega,\mathcal{F},\{\mathcal{F}_t\}_{t\ge 0},P)$, 
and the stochastic integral in \eqref{eq:SBL} is 
understood in the Stratonovich sense. 
The ``noise-frequency'' index set $L$ is finite in our setting. 
In applications such as turbulence modeling, however, it 
may naturally be infinite, in which case 
numerical approximations are carried out using 
a truncated finite set of modes 
(see Subsection \ref{subsec:Kraichnan})..

Let $f:\R^d\times \R\to \R^d$ be a given deterministic flux vector,
and let $a(x,u):\R^d\times \R\to \R^{d\times d}$ be a deterministic,
symmetric, nonnegative matrix-valued function. The numerical schemes
developed herein extend naturally to mixed hyperbolic--parabolic SPDEs
of the form
\begin{equation}\label{eq:SBL-general}
	\partial_t u
	+\Bigl(\div_x f(x,u)
	-\Div_x\bigl(a(x,u)\nabla_x u\bigr)\Bigr)
	+\sum_{\ell\in L}
	\div_x\bigl(\sigma_\ell(x)g_\ell(u)\bigr)
	\circ \partial_t W^\ell(t)=0.
\end{equation}
Since LDG discretizations of the deterministic terms in
\eqref{eq:SBL-general} are well established
(see, e.g., \cite{Cockburn:1999ud, Cockburn:1998ai}), we restrict, for simplicity of
presentation, to the case $f\equiv 0$ and $a\equiv 0$, and thus focus
on \eqref{eq:SBL}.

The class of SPDEs described by \eqref{eq:SBL}
encompasses linear and nonlinear hyperbolic equations
driven by gradient noise or stochastic fluxes,
and thus provides a flexible testbed 
for the development of numerical schemes capable of 
capturing both smooth solutions
and discontinuous shock waves. 
A linear example of relevance is the equation
\begin{equation}\label{eq:turb}
	\partial_t u+\sum_{\ell\in L}
	\Div_x\bigl(\sigma_\ell(x)u\bigr)
	\circ \partial_t W^\ell(t)=\eps \Delta u,
\end{equation}
where $\eps \ge 0$ is a parameter.
The mathematical analysis 
of equations like \eqref{eq:turb} 
often relies on stochastic calculus, which requires
expressing them in their It\^{o} form:
\begin{equation}\label{eq:turb-ito}
	\partial_t u+\sum_{\ell\in L}
	\Div_x\bigl(\sigma_\ell(x)u\bigr)
	\partial_t W^{\ell}=\eps \Delta u
	+\frac{1}{2}\sum_{\ell\in L}
	\Div_x\bigl(\sigma_\ell(x) 
	\Div_x(\sigma_\ell(x)u)\bigr),
\end{equation}
where the last term on the right-hand 
side is the Stratonovich--It\^{o} correction, 
a second order differential operator. 
Equations of this type arise in the modeling of passive
scalar transport in turbulent flows, with close connections
to the Kraichnan model of turbulent 
advection \cite{Kraichnan1967,Kraichnan1994} 
(see also the recent accounts 
\cite{FlandoliLuongo2023,LototskyRozovsky2017}).
More precisely, assuming that 
each noise vector $\sigma_\ell$ is 
divergence-free, the SPDE \eqref{eq:turb} 
can be viewed as modeling heat diffusion under 
a turbulent velocity field, where the noise enters 
in the transport form $v(t,x)\cdot\nabla u$ with 
$v(t,x)=\sum_{\ell\in L}\sigma_\ell(x)
\partial_t W^\ell(t)$ a Gaussian random field. 
Its It\^{o} formulation then makes explicit how 
such random advection enhances 
dissipation of the mean value of $u$ 
(as observed in real turbulent flows) 
\cite{FlandoliLuongo2023}.

Define the $(t,x)$-noise field 
$\mathcal{W}(t,x):=\sum_{\ell\in L}
\sigma_\ell(x)W^\ell(t)$, 
which is correlated in space and white in time. 
Then the transport term can be written as
$v\cdot\nabla u=\nabla u 
\circ \partial_t \mathcal{W}$.
The associated spatial 
covariance function is
$$
\mathcal{C}(x,y):=
\E\left[\mathcal{W}(t,x)\otimes 
\mathcal{W}(t,y)\right]
=\sum_{\ell\in L}\sigma_\ell(x)
\otimes\sigma_\ell(y).
$$
Define $C(x):=\mathcal{C}(x,x)$, 
which is a nonnegative matrix. Assuming that each 
$\sigma_\ell$ is divergence-free, one verifies that 
\eqref{eq:turb-ito} (with $\eps=0$) can be written as
\begin{equation*}
	\partial_t u+\sum_{\ell\in L}
	\bigl(\sigma_\ell(x) \cdot \nabla u\bigr)
	\, \partial_t W^{\ell}
	=\frac{1}{2}\nabla
	\cdot\bigl(C(x)\nabla u\bigr)
	= \frac{1}{2}\sum_{i,j=1}^d
	\partial_{x_i}\bigl(C_{ij}(x)\partial_{x_j}u\bigr).
\end{equation*}
Despite the presence of a second-order 
parabolic operator, the equation remains 
hyperbolic in nature.  Random velocity fields 
in Kraichnan’s turbulence model corresponds 
to a homogeneous, divergence-free field 
with covariance depending only on the difference $x-y$, 
so that $\mathcal{C}(x,y)=\mathcal{C}(x-y)$ 
and $C=(C_{ij})=\mathcal{C}(0)$.

Beyond linear stochastic transport, nonlinear conservation 
laws with stochastic fluxes arise in several other settings. 
In mean field games \cite{Lasry:2007fk,Lions:2013aa}, stochastic 
perturbations of continuity equations appear in 
the description of interacting particle 
systems and their mean field limits. 
In non-equilibrium statistical mechanics, fluctuating 
hydrodynamics leads to conservative 
SPDEs with multiplicative gradient noise, 
such as the Dean-Kawasaki equation, and to related models 
connected with macroscopic fluctuation theory and large 
deviations of interacting particle systems 
(see, for example, \cite{fehrman2023nonequilibrium} 
and the references therein)

There are several works addressing 
the well-posedness of \eqref{eq:SBL}
under different structural assumptions 
on the drift and noise terms, 
as well as within various 
functional-analytic frameworks
and notions of solution. In the purely 
deterministic setting with nonlinear drift 
and no stochastic perturbation, well-posedness 
follows from the classical Kru{\v z}kov 
entropy theory \cite{Kruzkov:1970kx} 
or its kinetic formulation \cite{Perthame:2002qy}. 
Linear deterministic transport equations  
are treated via renormalization techniques
\cite{Ambrosio:2004aa, DiPerna:1989aa}. 
For linear stochastic transport (continuity) 
equations, including cases with rough 
deterministic velocities $V$ (i.e., \eqref{eq:SBL-general} 
with $f=V(x)u$ and $a\equiv 0$), a rich theory 
of weak solutions has been developed
\cite{Attanasio:2011fj,Beck:2019aa,Flandoli:2010yq,Mau2011},
highlighting in particular regularization by noise phenomena.
Further developments in this direction are surveyed in
\cite{FlandoliLuongo2023,LototskyRozovsky2017}. 
Stochastic conservation laws with nonlinear
gradient-type noise have been 
studied both by pathwise methods, 
avoiding stochastic calculus 
\cite{Gess:2017aa, Lions:2013aa},
and within a stochastic calculus framework
in \cite{Gess:2018aa} for linear noise and nonlinear 
deterministic transport.
Equations with nonlinear noise, arising for 
example in models of fluctuating hydrodynamics, 
are analyzed in \cite{Fehrman:2024aa,FehrmanGess2025}.
Further contributions addressing 
other classes of nonlinear SPDEs with gradient-type noise,
beyond the settings discussed above, can be found in
\cite{Albeverio:2021uf,Alonso-Oran:2021wt,
Breit:2022aa,Crisan:2019aa,
Flandoli:2020aa,Galimberti:2024aa}.
We do not attempt a comprehensive overview here.

Classical solutions to the 
stochastic conservation law \eqref{eq:SBL}
satisfy an entropy (energy) balance 
that reveals the hyperbolic character 
of transport noise. To make this 
structure explicit, we pass 
to the It\^o formulation of
\eqref{eq:SBL}, which reads
\begin{equation}\label{eq:ItoSBL}
	\partial_t u
	+\sum_{\ell\in L}
	\div_x\bigl(\sigma_\ell(x)g_\ell(u)\bigr)
	\,\partial_t W^\ell
	=\frac12\sum_{\ell\in L}
	\div_x\Bigl(\sigma_\ell(x)g_\ell'(u)
	\div_x\bigl(\sigma_\ell(x)g_\ell(u)\bigr)\Bigr).
\end{equation}
Let $S\in C^2(\R)$ be a convex entropy. 
Define the entropy fluxes
\begin{align*}
	 & 
	G_\ell^{(S)}(u)
	:=\int_0^u S'(\lambda)
	g_\ell'(\lambda)\,d\lambda, 
	\quad
	D_\ell^{(S)}(u):= \int_0^u S'(\lambda)
	(g_\ell'(\lambda))^2\,d\lambda,
	\quad \ell\in L,
\end{align*}
and the functions
\begin{align*}
	H^{(S)}_{\ell}(u)
	:= \int_0^u S''(\lambda)
	\,g_\ell(\lambda)g_\ell'(\lambda)\,d\lambda,
	\quad \ell\in L.	
\end{align*}
For the quadratic entropy $S(u)=\frac12 u^2$, 
which is of special interest to us,
one has $H^{(S)}_\ell(u)=\tfrac12 g_\ell(u)^2$.
We also need the accompanying functions
\begin{align*}
	\mathcal F^{(S)}(x,u)
	&:=
	\frac12\sum_{\ell\in L}
	\Div_x\bigl(\sigma_\ell(x)
	\sigma_\ell(x)^{\top}\bigr)
	D_\ell^{(S)}(u)
	-\frac12\sum_{\ell\in L}
	\sigma_\ell(x)(\div_x\sigma_\ell(x))
	H^{(S)}_{\ell}(u), 
	\\
	\mathcal Z_\ell^{(S)}(x,u)
	&:=
	\div_x\bigl(\sigma_\ell(x)G_\ell^{(S)}(u)\bigr)
	+(\div_x\sigma_\ell(x))
	\bigl(S'(u)g_\ell(u)-G_\ell^{(S)}(u)\bigr), 
	\quad \ell\in L,
\end{align*} 
and 
\begin{align*}
	R_S(x,u)
	& :=
	\frac12\sum_{\ell\in L}
	\Biggl[
	S''(u)g_\ell(u)^2\,(\div_x\sigma_\ell(x))^2
	\notag \\ & \qquad \qquad\qquad\quad
	-H^{(S)}_{\ell}(u)
	\Bigl(
	\sigma_\ell(x)\cdot\nabla_x(\div_x\sigma_\ell(x))
	+(\div_x\sigma_\ell(x))^2
	\Bigr)
	\Biggr].
\end{align*}
Any classical solution $u$ of \eqref{eq:SBL} 
satisfies the following entropy balance in divergence form 
with source terms, understood in the weak sense in time:
\begin{align}\label{eq:entropy-eq}
	\partial_t S(u)
	& +\div_x \mathcal F^{(S)}(x,u)
	+\sum_{\ell\in L}\mathcal Z_\ell^{(S)}(x,u)
	\,\partial_t W^\ell
	=\frac12\sum_{\ell\in L}\Div_x\Div_x\Bigl(
	\sigma_\ell(x)\sigma_\ell(x)^{\top}D_\ell^{(S)}(u)\Bigr)
	+R_S(x,u).
\end{align}
The derivation of \eqref{eq:entropy-eq} 
is somewhat lengthy but follows standard 
entropy computations for
mixed hyperbolic-parabolic equations and relies 
on It\^o's formula together with repeated 
use of the chain and Leibniz rules.
At first sight, one encounters terms of the form
$S''(u)|\sigma_\ell\cdot\nabla_x g_\ell(u)|^2$,
which resemble genuine parabolic dissipation.
In a deterministic parabolic setting such terms 
would survive and generate entropy production.
Here, however, they cancel exactly with contributions
generated by the Stratonovich--It\^o correction 
linked to the gradient noise part of \eqref{eq:SBL}. 
The remaining gradient terms regroup into the expression 
$\Div_x\Div_x(\sigma_\ell\sigma_\ell^{\top}D_\ell^{(S)}(u))$, 
so all derivatives of $u$ appear only in divergence 
form and the remainder $R_S$ is a zero-order source term. 
This exact cancellation of the 
$S''(u)|\sigma_\ell\cdot\nabla_x g_\ell(u)|^2$ terms 
is a characteristic feature of transport noise. 
It reflects a genuinely hyperbolic mechanism that 
is delicate to preserve in naive numerical discretizations, 
especially since \eqref{eq:ItoSBL} contains both 
hyperbolic and parabolic differential operators.

If the $\sigma_\ell$'s are divergence-free then the entropy balance 
\eqref{eq:entropy-eq} simplifies. 
In this case the source terms 
generated by $\div_x\sigma_\ell$ vanish: 
$\mathcal F^{(S)}\equiv0$ 
and $R_S\equiv0$. Moreover,
$\mathcal Z_\ell^{(S)}(x,u)
=\div_x(\sigma_\ell G_\ell^{(S)}(u))$,
and all stochastic contributions appear 
purely in conservative form.
Thus the entropy equation 
\eqref{eq:entropy-eq} 
consists only of conservative 
divergence and double-divergence terms.

For non-smooth solutions, the 
identities \eqref{eq:entropy-eq} 
(one for each $S$) are replaced 
by inequalities, and the resulting entropy 
inequalities distinguish between non-unique 
weak solutions. In this setting, 
\eqref{eq:entropy-eq} 
is understood in the weak (distributional) sense. 
From a numerical perspective we focus on the 
quadratic entropy (energy) $S(u)=\tfrac{1}{2}u^2$, 
for which \eqref{eq:entropy-eq} 
yields (for $d\ge 2$, assuming 
divergence-free vector fields)
\begin{equation}\label{eq:energy-est}
	\norm{u(t)}_{L^2(\R^d)}^2
	\le\norm{u_0}_{L^2(\R^d)}^2, 
	\quad \text{almost surely}, 
\end{equation}
for any $t>0$, where $u_0\in L^2(\R^d)$ 
is the given initial function. 
A weaker version of \eqref{eq:energy-est} 
replaces the almost sure bound with an estimate 
in expectation. Some of our schemes satisfy the strong 
pathwise form \eqref{eq:energy-est}, while 
others only the weaker one. 
From \eqref{eq:entropy-eq}, assuming 
divergence-free vector fields, it also follows that 
if $u_0$ takes values in an 
interval $[a,b]$, then the solution $u$ 
remains in $[a,b]$ for almost every $(\omega,t,x)$. 
This property will not be preserved 
by the proposed schemes. 
However, we address the non-trivial issue of developing 
invariant-preserving finite difference schemes 
in a forthcoming paper.

\medskip

For developing numerical schemes for 
the general SPDEs \eqref{eq:SBL}, it is natural 
to adopt the Itô formulation rather than the 
Stratonovich form, since this aligns with 
stochastic calculus and facilitates stability 
analysis. At the same time, it poses challenges 
for constructing structure-preserving methods, 
in particular for ensuring that the discrete 
analogue of \eqref{eq:entropy-eq} exhibits 
appropriate gradient cancellations. 
We will use \eqref{eq:turb-ito} (with $\eps=0$) 
as a model problem for constructing discretizations 
of stochastic transport that preserve 
its hyperbolic character. The resulting design 
principles will be extended 
to the fully nonlinear equation \eqref{eq:SBL}.

Numerical analysis of SPDEs 
\eqref{eq:SBL} with gradient noise 
is still scarce, with only a few works available.
The recent article \cite{Fjordholm:2023aa} 
considers linear equations and identifies the 
key challenge of enforcing gradient cancellations in 
\eqref{eq:entropy-eq} at the discrete level in order 
to recover mean energy stability. 
The analysis of the tailored first-order difference 
scheme in \cite{Fjordholm:2023aa} 
also exploits a regularization-by-noise effect 
to treat rough deterministic velocities.

Other contributions on equations with gradient noise 
include \cite{Hoel:2018aa}, \cite[Sec.~8]{LL2022}, 
and \cite{CCDPS2019}. For SPDEs with lower-order 
stochastic forcing instead of transport noise, there 
is a broad literature 
\cite{Bauzet:2016aa,Bauzet:2016ab,Banas:2014aa,
Dotti:2020aa,Funaki:2018aa,Jentzen:2011aa,
Kroker:2012fk,Li:2020aa,Lord:2014aa,
Majee:2018aa,Ondrejat:2022aa}, 
though this list is far from complete. 
For a general introduction and overview 
of numerical methods for SPDEs, we 
refer to \cite{Zhang:2017aa}

\medskip

In the present paper we develop first- and 
high-order numerical schemes within the 
local discontinuous Galerkin (LDG) framework. 
These schemes preserve the gradient-cancellation mechanism 
underlying transport noise and thereby retain its 
hyperbolic stability property, both pathwise and 
in expectation. In contrast, \cite{Fjordholm:2023aa} analyzes 
a first-order difference scheme that is 
stable in the mean energy sense.

The local discontinuous Galerkin (LDG) method was 
introduced as an extension of the Runge–Kutta DG method 
to handle PDEs with higher-order derivatives 
\cite{Cockburn:1999ud,Cockburn:1998ai}. 
The key idea is to rewrite the PDE as a first-order 
system by introducing auxiliary variables, which 
are then discretized by DG techniques with carefully 
chosen numerical fluxes; the flux choice affects accuracy, 
stability, and stencil size 
\cite{LDGErrorEstimates,Di-Pietro:2012aa}. 
While this increases the number of degrees of freedom and 
may impose restrictive time-step conditions, LDG schemes 
remain highly local, parallelizable, and 
well-suited to hp-adaptivity, making them 
attractive for advection-dominated problems. 
LDG has been widely used in deterministic contexts, 
including elliptic problems 
\cite{DGEllipticProblems,UnifiedAnalysis,
PerformanceDGElliptic}  
and nonlinear wave-like equations 
\cite{InvariantLDG,LDGCamassaHolm,DGMethodHS1,
DGMethodHS2,Xu:2011aa,LDGmuCHDP}.

Its extension to SPDEs is more recent: 
\cite{Li:2021aa} established well-posedness 
and mean $L^2$ stability 
for nonlinear stochastic parabolic equations 
with multiplicative noise, with quasi-optimal 
$O(\Delta x^{k+1})$ error estimates in 
the semilinear case. Related DG approaches have 
been proposed for stochastic scalar conservation 
laws \cite{Li:2020aa}, and the symmetric 
interior penalty method, together with an 
oscillation-free variant of it, was recently extended 
to convection-diffusion equations 
with multiplicative noise \cite{StochasticOSCFree}. 
A systematic development of LDG methods for SPDEs with 
gradient noise has not yet been carried out; 
this is the focus of the present paper.

While \cite{Fjordholm:2023aa}
treated linear equations with 
a first-order difference scheme, 
our approach provides a framework that 
naturally accommodates high-order polynomial 
approximations (for linear and nonlinear equations). 
We note that the first-order versions of our schemes do not 
coincide with the discretization 
proposed in \cite{Fjordholm:2023aa} 
(see Section \ref{sec:FD}). 
Preserving the gradient cancellations in \eqref{eq:entropy-eq} 
at the discrete level (at least up to an inequality) guides 
our scheme design and motivates the introduction of 
suitably chosen auxiliary variables.

For the model equation \eqref{eq:turb-ito} 
(with $\eps=0$) we proceed by introducing the auxiliary variables  
$q_\ell = \nabla\cdot(\sigma_\ell u)$ for $\ell\in L$, 
so that the SPDE \eqref{eq:turb-ito} 
may be recast as the first-order system
\begin{align*}
	du &= \frac12 \sum_{\ell\in L}
	\Bigl((\nabla\cdot\sigma_\ell)\,
	\nabla\cdot(\sigma_\ell u) 
	+\sigma_\ell\cdot\nabla q_\ell\Bigr)\,dt
	-\sum_{\ell\in L} q_\ell\, dW^\ell, 
	\qquad
	q_\ell= \nabla\cdot(\sigma_\ell u),
	\quad \ell\in L.
\end{align*}
This formulation makes the role of the auxiliary 
variables explicit: they appear 
both in the stochastic term and in 
(one part of) the It\^{o} correction, ensuring that 
the discrete scheme can replicate the 
continuous cancellation property that leads 
to \eqref{eq:entropy-eq}. 
Alternative first-order reformulations are 
possible; we will comment on these as needed and 
explain why the present choice is particularly well suited. 
Similar considerations will apply to the 
nonlinear SPDEs. 

With this first-order system structure in place, we
develop LDG schemes of arbitrary polynomial order,
based on carefully chosen numerical fluxes.
We derive \emph{in-cell} energy (in)equalities, 
which reveal at the discrete level the gradient 
cancellations present in \eqref{eq:entropy-eq} 
for $S(u)=\frac{1}{2}u^2$. For certain flux choices, these inequalities hold 
in the strong pathwise sense, while for others they hold 
only in expectation, showing that the fluxes must be 
chosen carefully: the “right’’ pairing reproduces the 
continuous variation-dissipation balance, whereas 
other choices guarantee only stability in the mean. 
By interpolating between central and upwind fluxes,
we introduce a single tunable parameter that controls
the balance between conservation and numerical damping,
thereby allowing the scheme either to preserve or to
dissipate the energy.
The resulting LDG schemes are shown to be well 
posed via stability estimates combined with a 
Khasminskii-type argument, without imposing 
linear growth on the nonlinear fluxes.

\medskip

The remainder of the paper is organized as follows. 
Section \ref{sec:prelim} presents some 
stochastic preliminaries and finite element notation. 
Section \ref{sec:LDGModel} studies 
a one-dimensional linear model equation and 
establishes global well-posedness and stability of 
the associated LDG semi-discretizations. 
Section \ref{sec:multiTransport} extends the 
LDG schemes to nonlinear multidimensional transport 
noise and proves global well-posedness and stability using 
a Khasminskii-type argument. 
Section \ref{sec:FD} derives the corresponding 
finite difference schemes arising from piecewise constant 
LDG approximations and make comparisons with 
the scheme in \cite{Fjordholm:2023aa}. Finally, 
Section \ref{sec:NumericalExamples} provides 
numerical experiments demonstrating high-order accuracy 
and the influence of flux choices on stability 
and shock resolution.

\section{Preliminaries}\label{sec:prelim}

We study finite element approximations 
of linear and nonlinear hyperbolic SPDEs \eqref{eq:SBL} 
driven by gradient noise. These equations are naturally 
posed in Stratonovich form and
subsequently rewritten in It\^o form. 

In this paper, we rely only 
on standard results from stochastic analysis,
including It\^o and Stratonovich integration,
the conversion formula between them,
and well-posedness theory for SDEs, as well as 
basic concepts and notation 
from finite element methods.
For convenience, we summarize the 
necessary background
in the following two subsections.

\subsection{Stochastic background}

In what follows, we briefly recall a few 
results from stochastic analysis. 
Detailed accounts can be found 
in \cite{Kunita:1990aa,Protter:2005aa}. 
General background on SDEs is available 
in \cite{Khasminskii:2012aa,Mao:2008aa}. 
For SPDEs, including equations with Stratonovich-driven 
transport terms, we refer to 
\cite{Chow:2015aa,FlandoliLuongo2023,LototskyRozovsky2017}. 
For numerical aspects of SDEs and SPDEs,
we refer to \cite{Kloeden:1992aa,Zhang:2017aa}.

Let $(\Omega, \F, \P)$ 
be a complete probability space and
$\{\F_t\}_{t \in [0,T]}$ a complete, 
right-continuous filtration. 
The quadruplet $(\Omega, \F,\{\F_t\}_{t \in [0,T]},\P)$
is referred to as the stochastic basis.
We denote by $\E[\cdot ]$ the expectation operator
associated with $\P$.

All processes inside stochastic integrals 
are assumed to be adapted and, when required, 
predictable or progressively measurable with 
respect to $\{\F_t\}_{t \in [0,T]}$. Predictable processes are 
progressively measurable, and progressively measurable 
processes are adapted. For precise definitions, 
see the above-mentioned books.

A martingale is an adapted integrable process
$M=\{M_t\}_{t\in[0,T]}$ such that
$\mathbb{E}\left[M_t\mid\F_s\right]=M_s$
for all times $t$ such that $0\le s\le t\le T$.
A local martingale is a process that 
becomes a martingale after stopping at 
an increasing sequence of stopping times
converging almost surely to $T$.

A (continuous) semimartingale is a process 
that can be written
as the sum of a continuous local 
martingale and a continuous
finite variation process.
This class includes Wiener processes and solutions of
stochastic differential equations and forms the natural
framework for It\^o integration.

If $X=\{X_t\}_{t\ge 0}$ and $Y=\{Y_t\}_{t\ge 0}$
are continuous semimartingales, then
\begin{equation}\label{eq:semimartingaleID}
	X_t Y_t
	=
	X_0 Y_0
	+
	\int_0^t X_s \, dY_s
	+
	\int_0^t Y_s \, dX_s
	+
	\langle X, Y \rangle_t,
\end{equation}
where $\langle X,Y\rangle_t$ 
denotes the quadratic covariation.
If either $X$ or $Y$ has 
finite variation, then
\begin{equation}\label{eq:vanishingCovar}
	\langle X,Y\rangle_t = 0.
\end{equation}
If $X$ is a 
continuous semimartingale and
$f\in C^2(\mathbb{R})$, then 
the It\^o formula reads
\begin{equation}\label{eq:ItoFormula}
	f(X_t)
	=
	f(X_0)
	+
	\int_0^t f'(X_s)\, dX_s
	+
	\frac12
	\int_0^t f''(X_s)\, d\langle X, X\rangle_s.
\end{equation}

The Fisk–Stratonovich integral can be defined 
via the It\^o integral by
\begin{align}\label{eq:StratIntegral}
	\int_0^t X_s \circ dY_s
	=
	\int_0^t X_s \, dY_s
	+
	\frac12 \langle X,Y\rangle_t,
\end{align}
which gives the Stratonovich–It\^o 
conversion formula. One may also define 
the Stratonovich integral directly
as the limit of midpoint Riemann sums, 
in contrast to the left endpoint 
construction for the It\^o integral.

For It\^o integrals with respect 
to a Wiener process, it suffices that the 
integrand be progressively measurable
and square-integrable.
In the general semimartingale setting,
predictability of the integrand is required.

\begin{lemma}[Martingale property of It\^o integrals]
\label{lem:Martingale}
Let $W=\{W_t\}_{0\le t\le T}$ be 
a one-dimensional Wiener process adapted to
$\{\mathcal{F}_t\}_{t\in[0,T]}$, and 
let $\{\mathcal{H}_t\}_{0\le t\le T}$ 
be a predictable process. If
$\E\left[\int_0^T \mathcal{H}_s^2\,ds
\right]<\infty$, then the It\^o integral
$M_t := \int_0^t \mathcal{H}_s\,dW_s$, 
$0\le t\le T$, is a square-integrable 
martingale with 
respect to $\{\mathcal{F}_t\}$ and 
satisfies the It\^o isometry 
$\E\bigl[|M_t|^2\bigr]
=\E\left[\int_0^t\mathcal{H}_s^2\,ds\right]$.  
If instead $\E\left[\Bigl(\int_0^T \mathcal{H}_s^2
\,ds\Bigr)^{1/2}\right] < \infty$, then $M$ is a 
martingale (but not necessarily square-integrable).
\end{lemma}

Next, we recall a standard well-posedness result for
matrix-valued stochastic differential equations.
Let $\|\cdot\|$ denote the Euclidean 
norm on $\R^d$, $\|x\|=\sqrt{x\cdot x}$.
For a matrix $\boldsymbol{w}
=(w_l^j)\in\mathbb{R}^{m\times k}$,
its Frobenius norm is defined by
\begin{equation*}
	\|\bold{w}\|_F^2
	:=
	\sum_{l=1}^m\sum_{j=1}^k (w_l^j)^2
	=
	\sum_{j=1}^k \|\bold{w}^j\|^2,
\end{equation*}
where $\bold{w}^j$ denotes 
the $j$th column. For $R>0$ we sometimes 
write $B_R:=\{x\in\R^{m\times k}:\|x\|_F\le R\}$.

\begin{theorem}[Local well-posedness of SDE systems]
\label{thm:SDEStandard}
Consider the matrix-valued SDE
\begin{align}\label{eq:SDEPrelim}
	dX_t
	=
	b(t,X_t)\,dt
	+
	\sum_{\ell\in L}
	\sigma_{\ell}(t,X_t)\,dW_t^{\ell},
\end{align}
where $b:[0,T]\times\mathbb{R}^{m\times k}\to
\mathbb{R}^{m\times k}$ and
$\sigma_{\ell}:[0,T]\times\mathbb{R}^{m\times k}
\to\mathbb{R}^{m\times k}$.
Assume that $b$ and $\sigma$ satisfy the 
following local Lipschitz conditions:
for every $t\in[0,T]$ and $n\in\mathbb{N}$
there exists $K_n>0$ such that
for all $x,z\in\mathbb{R}^{m\times k}$
with $\|x\|_F\vee\|z\|_F\le n$,
\begin{align*}
	&\|b(t,x)-b(t,z)\|_F
	\le K_n\|x-z\|_F, \nonumber\\
	&\|\sigma_{\ell}(t,x)-\sigma_{\ell}(t,z)\|_F
	\le K_n\|x-z\|_F,
	\quad \text{for all $\ell\in L$}.
\end{align*}
Assume moreover that $X_0$ is
$\F_0$–measurable and
$\|X_0\|_F<\infty$ almost surely.
Then \eqref{eq:SDEPrelim} 
admits a unique local strong solution defined 
on $[0,\tau_{\mathrm{max}})$,
for a possibly random explosion time
$\tau_{\mathrm{max}}\in(0,T]$.

If, in addition, 
$\E\left[\|X_0\|_F^2\right]<\infty$,
and for every $R>0$ there exists a constant $C_R>0$
such that for all $(t,x)\in[0,T]\times B_R$,
\begin{align}\label{prelim:locGrowth}
	&\|b(t,x)\|_F
	\le C_R(1+\|x\|_F), \nonumber\\
	&\|\sigma_{\ell}(t,x)\|_F
	\le C_R(1+\|x\|_F),
	\quad \text{for all $\ell\in L$},
\end{align}
then, for any $t^*<\tau_{\mathrm{max}}$ 
and $p\ge 1$,
\begin{equation*}
	\E\left[\,
	\sup_{t\in[0,t^*]}
	\|X_t\|_F^p
	\right]
	<\infty.
\end{equation*}
If the linear growth bounds \eqref{prelim:locGrowth}
hold globally on $[0,T]\times\mathbb{R}^{m\times k}$
with a constant independent of $R$,
then $\tau_{\mathrm{max}}=T$ almost surely,
that is, the solution is global and 
extends to $[0,T]$.
\end{theorem}

This theorem combines results from
\cite[Sec.~3]{Khasminskii:2012aa} and
\cite[Sec.~2]{Mao:2008aa}. More refined 
global existence criteria under
nonlinear growth conditions are available,
notably via Lyapunov function techniques as in
\cite[Sec.~3]{Khasminskii:2012aa}.
We will return to this approach in
Section \ref{sec:multiTransport},
where global well-posedness of the LDG
semi-discretization is established in a setting
where global linear growth is too restrictive.

\subsection{Finite element notation} \label{sec2:FEMNotation}

We next fix the finite element notation 
used throughout the paper. For general 
background on discontinuous 
finite element methods, we refer to
\cite{Cockburn:1999ud,Dolejsi:2015aa,Hesthaven:2008aa}.

Let $D \subset \mathbb{R}^d$ be 
a bounded polygonal domain.
The polygonal assumption avoids 
geometric approximation
errors, since the mesh resolves 
the boundary exactly.
Let $\T_h=\{K\}$ be a 
conforming partition of $D$
into non-overlapping $d$-dimensional 
polyhedra such that
$$
\overline{D}
=\bigcup_{K\in\T_h} \overline{K}.
$$
We assume that for any two elements
$K_-,K_+\in\mathcal{T}_h$,
the intersection $\overline{K}_-\cap\overline{K}_+$
is either a common $(d-1)$-dimensional face $e$
or has Hausdorff dimension strictly less than $d-1$.
In particular, for $d=1$ the elements are intervals and
their interfaces are points,
for $d=2$ the elements are polygons
(triangles or quadrilaterals) with edges as interfaces,
and for $d=3$ the elements are polyhedra
(tetrahedra or hexahedra) with polygonal faces.
Let $\mathcal{E}_h$ denote 
the set of all (unique) element faces.

For each $K\in\mathcal{T}_h$,
let $h_K:=\mathrm{diam}(K)$ and define
the mesh size
$$
h := \max_{K\in\mathcal{T}_h} h_K.
$$
We assume that the 
family $\{\mathcal{T}_h\}_{h>0}$
is shape-regular, i.e., there exists $c>0$
independent of $h$ such that
$$
\frac{h_K}{\rho_K}\le c
\quad\text{for all $K\in\mathcal{T}_h$},
$$
where $\rho_K$ denotes the 
diameter of the largest
ball contained in $K$.

Given a mesh $\T_h=\{K\}$, we define the 
discontinuous finite element space
of piecewise polynomials 
of degree at most $k\ge0$ by
\begin{equation*}
	\mathcal{V}^k
	:=\bigl\{
	v\in L^2(D)
	:v|_K \in \mathcal{P}^k(K)
	\text{ for all } K\in\mathcal{T}_h
	\bigr\}.
\end{equation*}

For $e\in\mathcal{E}_h$ with
$e=\partial K_+\cap\partial K_-$,
let $n^+$ and $n^-$ denote the 
outward unit normals on $\partial K_+$ 
and $\partial K_-$, respectively,
so that $n^-=-n^+$ (see Figure \ref{fig:normals}). 
When needed, we emphasize the dependence 
on the interface by writing $n_e^\pm$.
If the orientation or the distinction 
between the two sides
is clear from the context or not relevant,
we simply write $n$ or $n_e$.

For a function $u:\R^d\to \R$, 
we define the average and jump across 
an interface $e$ by
\begin{equation}\label{eq:avgNotation}
	\avg{u}=\frac12
	\bigl(u|_{K_+}+u|_{K_-}\bigr)
\end{equation}
and
\begin{align}\label{eq:jumpNotation}
	\llbracket u \rrbracket
	=
	n^+ u|_{K_+}
	+
	n^- u|_{K_-}
	=
	n^+\bigl(
	u|_{K_+}-u|_{K_-}
	\bigr).
\end{align}
These definitions are independent of the ordering
of $K_+$ and $K_-$.

\begin{figure}
	\includegraphics[width=0.9\linewidth]
	{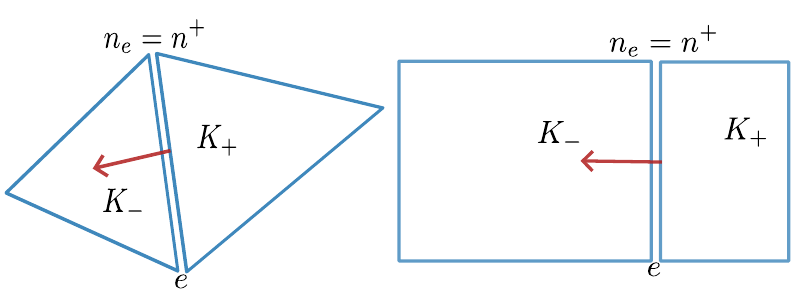}
	\vspace{-0.5cm}
	\captionsetup{width=.975\linewidth}
	\caption{Two adjacent elements $K_-,K_+\in\mathcal{T}_h$
	(left: triangulation, right: quadrangulation).
	The outward unit normal $n_e$ 
	on their common interface $e$ is indicated.}
	\label{fig:normals}
\end{figure}

In one spatial dimension the 
mesh admits a natural ordering. 
Let $\{x_{j+\frac12}\}_{j\in\mathbb{Z}}$ be 
the cell interfaces of a possibly nonuniform partition, 
with cells $I_j=[x_{j-\frac12},x_{j+\frac12}]$, 
cell centers $x_j=\frac12\big(x_{j-\frac12}+x_{j+\frac12}\big)$, 
and cell sizes
\begin{equation*}
	\Delta x_j
	:=x_{j+\frac12}-x_{j-\frac12}.
\end{equation*}
For a function $u$, we write
\begin{equation*}
	\llbracket u \rrbracket_{j+\frac12}
	=u_{j+\frac12}^+
	-u_{j+\frac12}^-,
	\quad \text{where} \quad
	u_{j+\frac12}^{\pm}
	:=\lim_{\varepsilon\downarrow0}
	u(x_{j+\frac12}\pm\varepsilon).
\end{equation*}
In one dimension, shape-regularity 
reduces to the requirement
that the ratio between maximal 
and minimal cell sizes remains 
uniformly bounded under 
grid refinements.

Finally, we introduce the 
standard upwind and downwind
fluxes with respect to 
a scalar velocity $b$:
\begin{align}\label{eq:downwindFlux}
	\widehat{\F}(u; b) := \begin{cases}
	u^-, & b \geq 0, \\
	u^+, & b < 0,
	\end{cases} \hspace{0.55cm} \text{and} \hspace{0.55cm}
	\widecheck{\F}(u; b) := \begin{cases}
	u^+, & b \geq 0, \\
	u^-, & b < 0. \end{cases}
\end{align}

\section{A one-dimensional model problem}\label{sec:LDGModel}

To highlight the main structural 
issues arising in the construction of numerical 
approximations for gradient noise and in the associated 
stability analysis, we begin with a simplified 
one--dimensional linear model problem. 
Specifically, we consider the stochastic continuity equation
\begin{equation}\label{eq:prob1Strat}
	\partial_t u + \partial_x(\sigma u) \circ \dot{W} = 0. 
\end{equation}
In this subsection, the noise 
coefficient $\sigma$ is assumed 
to have the following regularity
\begin{equation}\label{eq:regularitySigma}
	\sigma \in C^2(\R)
	\hspace{0.45cm}
	\text{and} 
	\hspace{0.45cm}
	\sigma'\!, \sigma''\!, 
	(\sigma^2)'' 
	\in L^{\infty}(\R),
\end{equation}
and such a function is globally Lipschitz 
and of at most linear growth. 
Furthermore, $W = \{W_t \mid 0 \leq t \leq T \}$ 
is a one-dimensional Brownian motion with 
respect to the specified stochastic basis, and 
$\circ \, \dot{dW}$ denotes the 
Stratonovich differential.

The regularity assumptions 
\eqref{eq:regularitySigma} are imposed only for 
simplicity. In the more typical 
multidimensional setting, where the noise 
amplitudes are divergence-free---as is 
common in turbulence 
modeling \cite{FlandoliLuongo2023}---much weaker 
regularity is sufficient. 
In particular, it is enough to assume 
$\sigma\in L^2$ with well-defined normal traces, as discussed 
in later sections.

To exploit the martingale property of It\^{o} 
integrals in both the construction of the 
numerical scheme and the stability analysis, 
we apply the It\^{o}--Stratonovich 
conversion formula \eqref{eq:StratIntegral} 
to rewrite \eqref{eq:prob1Strat} in It\^{o} form as
\begin{align}\label{eq:problem1}
	du-\frac{1}{2}\partial_x
	\bigl(\sigma \partial_x (\sigma u) \bigr)\, dt 
	+\partial_x\bigl(\sigma u\bigr)\, dW_t  &= 0.
\end{align}
This formulation is the basis for numerically 
approximating weak solutions, as 
defined in Section \ref{sec:prelim}.

In contrast to \cite{Fjordholm:2023aa}, 
which proposed a first-order 
difference scheme for linear equations with possibly 
nonsmooth deterministic velocities, 
we develop a class of schemes that extends 
to higher order and preserves 
the hyperbolic energy structure 
induced by gradient noise via 
cancellations between quadratic 
variation and the It\^{o}--Stratonovich 
correction (parabolic dissipation). 
While our framework applies to multidimensional 
as well as nonlinear 
equations (as discussed in a later section), 
the fundamental mechanisms and challenges 
already manifest at the level of the 
model problem \eqref{eq:prob1Strat}.

\subsection{The LDG-formulation}\label{sec:LDGFormulation}
To appropriately discretize the It\^{o}--Stratonovich 
term $\partial_x \big(\sigma \partial_x (\sigma u) \big)$ 
in \eqref{eq:problem1}, 
which is a variable-coefficient second-order 
parabolic operator, we 
introduce the auxiliary variable $q$ via 
\begin{equation*}
	q - \partial_x(\sigma u) = 0,
\end{equation*}
and rewrite
\begin{equation}\label{eq:rewriteCorrection}
	\partial_x \big(\sigma \partial_x(\sigma u )\big) 
	=\sigma' \partial_x(\sigma u)+\sigma \partial_x q,
\end{equation}
where $\sigma' = \partial_x\sigma$. 
Consequently, \eqref{eq:problem1} may be 
written as the following 
system of first-order equations
\begin{subequations}
	\begin{align}
		du &= \frac{1}{2}\pigl(\sigma' \partial_x(\sigma u) 
		+\sigma \partial_x q\pigr)dt - q 
		\,  dW_t, \label{eq:rewrite1}\\
		q &=\partial_x(\sigma u). \label{eq:rewrite2}
	\end{align}
\end{subequations}

To derive the LDG scheme, we discretize 
\eqref{eq:rewrite1}--\eqref{eq:rewrite2} 
elementwise using a DG method (see, 
e.g., \cite{Cockburn:1999ud,Di-Pietro:2012aa} 
for comprehensive introductions). 
Specifically, we multiply the two equations 
by smooth test functions $\varphi$ and $\psi$ 
and integrate by parts over each cell $I_j$ 
to shift spatial derivatives 
onto the test functions. This yields
\begin{subequations}
	\begin{align}
		d\int_{I_j} \varphi  u(t)\, dx & = 
		- \frac{1}{2} \int_{I_j} \partial_x\bigl(\sigma'\varphi\bigr) 
		\sigma u(t)\, dx\, dt \nonumber
		\\ & \quad+ \frac{1}{4}\Big(\pigl( \bigl(\sigma^2\bigr)'u(t) 
		\varphi^-\pigr)_{j+\frac{1}{2}} - \pigl( \bigl(\sigma^2\bigr)'u(t) 
		\varphi^+\pigr)_{j-\frac{1}{2}}\Big)\, dt \nonumber
		\\ & \quad - \frac{1}{2}\int_{I_j}
		\partial_x\bigl(\sigma \varphi\bigr)q(t)\, dx\, dt 
		+ \frac{1}{2} \Big( \pigl(\sigma q(t)
		\varphi^-\pigr)_{j+\frac{1}{2}}
		-\pigl(\sigma q(t)\varphi^+\pigr)_{j-\frac{1}{2}}\Big)
		\, dt \nonumber
		\\ & \quad 
		- \int_{I_j} \varphi q(t)\, dx \, dW_t,
		\label{eq:temp1Weak}\\
		\int_{I_j} \psi q(t)\, dx 
		&= - \int_{I_j} \partial_x\psi
		\sigma u(t)\, dx 
		+\Big(\pigl(\sigma u(t)\psi^-\pigr)_{j+\frac{1}{2}} 
		-\pigl(\sigma u(t)\psi^+\pigr)_{j-\frac{1}{2}}\Big).
		\label{eq:temp2Weak}
	\end{align}
\end{subequations}
We replace the test functions $(\varphi,\psi)$ 
by functions in 
$\mathcal{W}^{k,l}:=\mathcal{V}^k\times\mathcal{V}^l$
(see Section \ref{sec:prelim}), 
which we continue to denote by 
$(\varphi,\psi)$, and approximate $(u,q)$ 
by $(u_h,q_h)$, where 
$(u_h,q_h)(\omega,t)$ belongs to 
$\mathcal{W}^{k,l}$ for all 
$(\omega,t)\in\Omega\times[0,T]$.  
Since functions in $\mathcal{V}^k$ and $\mathcal{Q}^l$ 
may be discontinuous across cell interfaces, 
the traces of $u(t)$ and $q(t)$ at $x=x_{j\pm\frac12}$ 
in \eqref{eq:temp1Weak} are replaced by numerical fluxes 
$\F_u$ and $\widetilde{\F}_q$, respectively, 
while the trace of $u(t)$ in \eqref{eq:temp2Weak} 
is replaced by $\widetilde{\F}_u$. These fluxes are 
required to be consistent approximations 
of the corresponding traces (more on this later). 

The resulting weak formulation reads: for each 
$(\omega,t)\in\Omega\times[0,T]$, find 
$$
(u_h(\omega,t),q_h(\omega,t))\in\mathcal{W}^{k,l}
$$ 
such that the equations
\begin{subequations}
	\begin{align}
		d\int_{I_j} \varphi u_{h}(t)\, dx 
		&= \Big(a_j \bigl (u_{h}(t), \varphi \bigr) 
		+ b_j\bigl(q_{h}(t), \varphi \bigr)\Big)\, dt 
		- \int_{I_j} \varphi q_{h}(t)\, dx \, dW_t, 
		\label{eq:weak1} \\ 
		\int_{I_j} \psi q_{h}(t)\, dx &= - \int_{I_j} 
		\partial_x (\psi) \sigma u_{h}(t)\, dx 
		+ \Big( \pigl(\sigma 
		\widetilde{\mathcal{F}}_u \psi^- \pigr)_{j+\frac{1}{2}} 
		- \pigl(\sigma \widetilde{\mathcal{F}}_u 
		\psi^+ \pigl)_{j-\frac{1}{2}} \Big), 
		\label{eq:weak2}
	\end{align}
\end{subequations}
hold for all test functions 
$(\varphi,\psi)\in\mathcal{W}^{k,l}$, 
where we have introduced the bilinear forms
\begin{equation}\label{eq:bilinearForms}
	\begin{split} 
		a_j(u_{h}, \varphi) :& 
		= -\frac{1}{2} \int_{I_j}
		\partial_x(\sigma'\varphi)
		\sigma u_{h}\, dx
		\\ & \qquad + \frac{1}{4} \Big( 
		\pigl( (\sigma^2)'
		\mathcal{\F}_u \varphi^-\pigr)_{j+\frac{1}{2}} 
		- \pigl ((\sigma^2)' \mathcal{\F}_u \varphi^+
		\pigr)_{j-\frac{1}{2}} \Big), \\
		b_j(q_{h}, \varphi) :&= - \frac{1}{2} \int_{I_j}
		\partial_x(\sigma \varphi) q_{h} \, dx
		+\frac{1}{2} \Big( \pigl(\sigma 
		\widetilde{\mathcal{\F}}_q \varphi^- \pigr)_{j+\frac{1}{2}} 
		- \pigl(\sigma \widetilde{\mathcal{\F}}_q 
		\varphi^+ \pigr)_{j-\frac{1}{2}} \Big). 
	\end{split}
\end{equation}
In addition, we apply the orthogonal $L^2$-projection $\Pi_k$ 
to map the initial data $\bar{u}$ into the space 
$\mathcal{V}^k$, which is equivalent to requiring that
 \begin{equation}\label{eq:weak3}
	\int_{I_j} \vartheta u_{h}(0)\, dx 
	= \int_{I_j} \vartheta \bar{u}\, dx, 
	\hspace{0.35cm} \text{for all } 
	\vartheta \in \mathcal{V}^k. 
\end{equation}

\subsubsection{Numerical fluxes}\label{sec:num-fluxes}
It remains to choose the numerical fluxes 
$\F_u$, $\widetilde{\F}_u$, and 
$\widetilde{\F}_q$. A natural choice for 
$\F_u$ is 
\begin{align}\label{eq:firstFlux1}
	\F_u := \gamma \widecheck{\mathcal{F}}
	\pigl(u_h,\big((\sigma^2)'\big) \pigr) 
	+ (1-\gamma)\avg{u_h} 
	+ \widetilde{\gamma} \, 
	\mathrm{sgn}\pigl(\bigl(\sigma^2\bigr)'\pigr)
	\llbracket u_h \rrbracket, 
\end{align}
for $\gamma \in [0, 1]$ and a penalty parameter 
$\widetilde{\gamma} \geq 0$. 
We refer to Section \ref{sec:prelim} and \eqref{eq:downwindFlux} for 
the notations $\widecheck{\mathcal{F}}$, $\avg{\cdot}$ 
and $\llbracket \cdot \rrbracket$. 
The parameter $\gamma$ interpolates 
between central and upwind fluxes and thus 
controls the amount of natural numerical 
dissipation associated with the first two terms. 
The choice $\gamma=0$ (with 
$\widetilde{\gamma}=0$) yields a nondissipative 
scheme, as established in the 
proof of Theorem \ref{thm:L2StabilityEst}, 
whereas $\gamma=1$ maximizes the dissipation 
associated with the first two terms in $\F_u$.

As we will see later, preserving the 
hyperbolic character of 
\eqref{eq:rewrite2} requires the numerical 
fluxes $\widetilde{\F}_u$ and
$\widetilde{\F}_q$ to be chosen 
in a consistent, coupled manner. 
There are several admissible 
choices and by introducing a 
parameter $\theta \in [0, 1]$, we 
can write them transparently as follows: 
\begin{align}\label{eq:firstFlux2}
	&\widetilde{\F}_u
	= \theta u_h^- + (1-\theta)u_h^+ 
	+ \eta_u\,\mathrm{sgn}(\sigma)
	\,\llbracket q_h\rrbracket, \nonumber 
	\\ &
	\widetilde{\F}_q
	= (1-\theta)q_h^- + \theta q_h^+ 
	+ \eta_q\,\mathrm{sgn}(\sigma)
	\,\llbracket u_h\rrbracket,	
\end{align}
which makes explicit the alternating 
nature of the fluxes (modulo penalty 
terms), where $\eta_u, \eta_q \geq 0$. For $\theta=\tfrac12$ the 
first two terms reduce to central fluxes, 
while for $\theta=0, 1$ they reduce to alternating fluxes.   

\begin{remark}
In the multi-dimensional extension presented later, 
when the noise amplitude $\sigma$ is 
divergence-free (or constant in 1D), 
pathwise $L^2$ estimates hold at the 
continuous level \cite{FlandoliLuongo2023}. 
Although we later establish several 
stability estimates in the mean-square sense 
for more general numerical flux choices, 
pathwise stability for the LDG approximations 
is recovered only with $\theta=\tfrac{1}{2}$ and $\eta_u=0$ in 
\eqref{eq:firstFlux2}. 
Of course, the resulting central fluxes may 
induce oscillations for discontinuous 
profiles (see Section \ref{sec:NumericalExamples} 
for related numerical examples).
\end{remark}

\subsubsection{Boundary conditions}\label{sec:BC}
For the numerical treatment, we restrict 
attention to a bounded computational 
domain, which we assume (unless stated otherwise) 
to be chosen large enough to 
contain the support of $u(t)$ for all $t\in[0,T]$, 
recalling that \eqref{eq:problem1} is a 
hyperbolic equation. This assumption allows us to 
impose boundary conditions by 
zeroth-order extrapolation, namely,
\begin{align*}
	\F_{u, m+\frac{1}{2}} &= \begin{cases}
	u_h^+,&m = 0, \\
	u_h^-, &m = N, 
	\end{cases}
\end{align*}
and, regardless of the particular choice of 
$\widetilde{\F}_u$ and $\widetilde{\F}_q$, we set
\begin{align*}
	\widetilde{\F}_{u, m+\frac{1}{2}} &= \begin{cases}
	u_h^+, &m=0, \\
	u_h^-, &m=N, 
	\end{cases}
	\hspace{0.45cm} \text{and} \hspace{0.45cm}
	\widetilde{\F}_{q, m+\frac{1}{2}} = \begin{cases}
	q_h^+, &m=0, \\
	q_h^-, &m=N. \end{cases}
\end{align*}

\subsection{Well-posedness of the schemes}\label{sec:wellPosedSDE}
The LDG semi-discretization converts 
\eqref{eq:rewrite1}--\eqref{eq:rewrite2} into a 
matrix-valued system of SDEs that is globally 
well posed under natural assumptions on the 
numerical fluxes. We verify this claim below, but 
since the argument closely follows 
\cite[Sec.~3.2]{Li:2021aa}, we state only the 
relevant results and leave the proofs for the interested reader.

Since $(u_h, q_h)(\omega, t) \in \mathcal{W}^{k, l}$ 
for each $(\omega, t) \in \Omega \times [0, T]$, one can 
on any control volume $I_j$ 
expand these functions as
\begin{equation}\label{eq:basisExpansion}
	u_{h}(\omega, t, x) = 
	\sum_{l=0}^k u_l^j(\omega, t)\phi_l^j(x),
	\quad
	q_{h}(\omega, t, x) = 
	\sum_{l=0}^l q_l^j(\omega, t)\widetilde{\phi}_l^j(x), 
\end{equation}
for some local bases $\{\phi_l^j\}_{l=0}^k$ of 
$\mathcal{P}^k(I_j)$ and $\{\widetilde{\phi}_l^j\}$ 
of $\mathcal{P}^l(I_j)$. 
To simplify the notation, we assume $l=k$, so that 
$\widetilde{\phi}_l^j=\phi_l^j$ for all $l=0,\ldots,k$. 
Inserting these expansions into 
\eqref{eq:weak1}--\eqref{eq:weak3} and choosing the basis 
functions as test functions yields a matrix-valued 
system of SDEs for the unknown coefficient matrices
\begin{align*}
	\bold{u} (t)= \begin{pmatrix}
	u_0^1(t) & u_0^2(t) & \hdots & u_0^{N}(t) \\
	u_1^1(t) & u_1^2(t) & \hdots & u_1^N(t) \\
	\vdots & \vdots & \hdots & \vdots \\
	u_k^1(t) & u_k^2(t) & \hdots & u_k^N(t) \\
	\end{pmatrix}, 
	\quad
	\bold{q}(t) = \begin{pmatrix}
	q_0^1(t) & q_0^2(t) & \hdots & q_0^{N}(t) \\
	q_1^1(t) & q_1^2(t) & \hdots & q_1^N(t) \\
	\vdots & \vdots & \hdots & \vdots \\
	q_k^1(t) & q_k^2(t) & \hdots & q_k^N(t) \\
	\end{pmatrix},
\end{align*}
where we suppress the dependency on $\omega \in \Omega$. 
More precisely, we derive a system for $\bold{u}(t)$, 
while $\bold{q}(t)$ can be eliminated, since 
the flux $\widetilde{\F}_u$ which appears 
in \eqref{eq:weak2} will be 
assumed to only depend on $u_h$.  

Denote by $\bold{u}^j(t) \in \R^{k+1}$ 
the $j$-th column of $\bold{u}(t)$, 
$\bold{q}^j(t) \in \R^{k+1}$ the 
$j$-th column of $\bold{q}(t)$, and let
\begin{equation*}
	\boldsymbol{\phi}^j(x) 
	= \begin{pmatrix}
	\phi_0^j(x) & \hdots & \phi_k^j(x)
	\end{pmatrix}^T 
	\hspace{0.65cm} 
	\text{for } j \in \{1,\ldots, N \}. 
\end{equation*} 
In addition, for brevity, let 
$\boldsymbol{\phi}^j_{j\pm \frac{1}{2}}
=\boldsymbol{\phi}^j(x_{j\pm \frac{1}{2}})$. 
With this notation, \eqref{eq:basisExpansion} 
can be recast as
\begin{equation*}
	u_{h}(t, x) = \bold{u}^j(t) \cdot 
	\boldsymbol{\phi}^j(x)
	\hspace{0.35cm}
	\text{and}
	\hspace{0.35cm}
	q_{h}(t, x)=\bold{q}^j(t)
	\cdot\boldsymbol{\phi}^j(x)
	\hspace{0.35cm} \text{for } x \in I_j. 
\end{equation*}
 
Instead of inserting for the fluxes $\F_u$ 
from \eqref{eq:firstFlux1} and 
$(\widetilde{\F}_u, \widetilde{\F}_q)$  
from \eqref{eq:firstFlux2}, 
we here prove well-posedness for a more 
general class of numerical fluxes. 
In particular, we allow both $\F_u$ and 
$\widetilde{\F}_q$ to depend on $u_h$ and $q_h$ 
at either side of the cell interface, that is,  
\begin{equation*}
	\F_{u, j+\frac{1}{2}}=\F\bigl(\bold{u}^{j}
	\cdot \boldsymbol{\phi}^j_{j+\frac{1}{2}}, \bold{u}^{j+1}
	\cdot \boldsymbol{\phi}^{j+1}_{j+\frac{1}{2}},
	\bold{q}^j \cdot \boldsymbol{\phi}^j_{j+\frac{1}{2}}, 
	\bold{q}^{j+1} \cdot 
	\boldsymbol{\phi}^{j+1}_{j+\frac{1}{2}}\bigr)
\end{equation*}
and 
\begin{equation*}
	\widetilde{\F}_{q, j+\frac{1}{2}} 
	= \widetilde{\F}\bigl(\bold{u}^{j}
	\cdot \boldsymbol{\phi}^j_{j+\frac{1}{2}}, 
	\bold{u}^{j+1} \cdot 
	\boldsymbol{\phi}^{j+1}_{j+\frac{1}{2}}, 
	\bold{q}^j \cdot \boldsymbol{\phi}^j_{j+\frac{1}{2}}, 
	\bold{q}^{j+1} \cdot 
	\boldsymbol{\phi}^{j+1}_{j+\frac{1}{2}}\bigr),
\end{equation*}
while $\widetilde{\F}_u$ is only allowed to depend 
on $u_h$, such that we can eliminate $q_h$ 
locally, namely there is a numerical 
flux $\F^{\star}$ such that 
\begin{equation*}
	\widetilde{\F}_{u, j+\frac{1}{2}}
	=\F^{\star}\bigl(\bold{u}^{j}
	\cdot\boldsymbol{\phi}^j_{j+\frac{1}{2}},\bold{u}^{j+1}
	\cdot\boldsymbol{\phi}^{j+1}_{j+\frac{1}{2}}\bigr). 
\end{equation*}
The flux function $\F(\cdot,\cdot,\cdot,\cdot)$ 
is assumed to satisfy the following conditions: 
\begin{enumerate}[label=(\roman*)]
	\item it is locally Lipschitz in all its arguments, 
	that is, for any $n \in \mathbb{N}$ there exists 
	a constant $K_n > 0$ such that for any $c_m, d_m \in \R$ 
	for $m=1,\ldots, 4$ with $\max_{m=1,\ldots,4}|c_m|
	\vee |d_m|\leq n$ one has
	\begin{equation}\label{eq:localLipschitz}
		\bigl | \F(c_1, c_2, c_3, c_4) 
		- \F(d_1, d_2, d_3, d_4) \bigr | 
		\leq K_n \sum_{m=1}^4|c_j - d_j|, 
	\end{equation}

	\item it is of at most linear growth in all arguments, i.e., 
	there exists $C_{\F}>0$ such that
	\begin{equation}\label{eq:linearGrowth}
		\bigl| \F(c_1, c_2, c_3, c_4) \bigr| 
		\leq C_{\F}\Bigl(1 + \sum_{m=1}^4 |c_m|\Bigr). 
	\end{equation}
\end{enumerate}
Similarly, we assume the existence of 
constants $(K_n^{\star},C_{\F^{\star}})$ and 
$(\widetilde{K}_n, C_{\widetilde{\F}})$ 
such that both $\F^{\star}(\cdot, \cdot)$ 
and $\widetilde{\F}(\cdot, \cdot, \cdot, \cdot)$ 
are locally Lipschitz and at most 
linearly growing in their respective arguments. 

From \eqref{eq:weak2}, when testing 
with $\psi = \phi_m^j$, for $m=0, \dots k$, 
we find
\begin{align*}
	\int_{I_j} \bold{q}^j(t) \cdot \boldsymbol{\phi}^j\phi_m^j\, dx 
	&= \sum_{l=0}^k q_l^j(t) \int_{I_j} \phi_l^j \phi_m^j \, dx 
	\\ & 
	= - \int_{I_j}\sigma \bold{u}^j(t) 
	\cdot \boldsymbol{\phi}^j \partial_x \phi_m^j \, dx 
	+ \sigma_{j+\frac{1}{2}}\F^{\star}\Big(\bold{u}^j(t)
	\cdot \boldsymbol{\phi}^j_{j+\frac{1}{2}}, 
	\bold{u}^{j+1}(t) \cdot \boldsymbol{\phi}^{j+1}_{j+\frac{1}{2}}\Big) 
	\phi_{m, j+\frac{1}{2}}^j 
	\\ & \qquad \qquad  
	- \sigma_{j-\frac{1}{2}}\F^{\star}
	\Big(\bold{u}^{j-1}(t) \cdot \boldsymbol{\phi}^{j-1}_{j-\frac{1}{2}},
	\bold{u}^{j}(t) \cdot \boldsymbol{\phi}^{j}_{j-\frac{1}{2}}
	\Big)\phi_{m, j-\frac{1}{2}}^j. 
\end{align*}

Introduce the local mass matrix 
$M^j \in \R^{(k+1)\times (k+1)}$, 
whose entries are given by 
\begin{equation}\label{eq:massEntries}
	M^j_{l m}:=\int_{I_j} 
	\phi_l^j\phi_m^j \,dx. 
\end{equation}
The sparsity of the matrix $M^j$ 
depends on the choice 
of basis: for example, Legendre 
polynomials yield a diagonal mass matrix that 
can be inverted efficiently, 
whereas other bases may lead to a full matrix. 
In all cases, $M^j$ is invertible; we therefore 
denote its inverse by $\mathbb{M}^j=(M^j)^{-1}$ and 
write $(\mathbb{M}^j)_l$ for its $l$th row. 
This allows us to express the components of 
$\bold{q}(t)$ explicitly in terms of $\bold{u}(t)$ as follows:
\begin{equation}\label{eq:qAsFunctionOfu1}
	\bold{q}^j(t) = \bold{Q}^j(\bold{u}(t))
	= \bold{Q}^j\big(\bold{u}^{j-1}(t), 
	\bold{u}^j(t), \bold{u}^{j+1}(t)\big),
\end{equation}
where $\bold{Q}^j(\bold{u}(t))
=\begin{pmatrix}
Q_0^j(\bold{u}(t)),\ldots, Q_k^j(\bold{u}(t))
\end{pmatrix}^T$ and, for $l=0,\ldots,k$, 
\begin{align}\label{eq:qAsFunctionOfu2}
	Q_l^j(\bold{u}(t)) &= - \int_{I_j} 
	\sigma \big(\bold{u}^j(t)\cdot \boldsymbol{\phi}^j \big)
	\big((\mathbb{M}^j)_l \cdot \partial_x \boldsymbol{\phi}^j \big)
	\,dx \nonumber
	\\ & \quad 
	+\sigma_{j+\frac{1}{2}} \F^{\star}\Big(\bold{u}^j(t)
	\cdot \boldsymbol{\phi}^j_{j+\frac{1}{2}},\bold{u}^{j+1}(t)
	\cdot \boldsymbol{\phi}^{j+1}_{j+\frac{1}{2}}\Big)
	(\mathbb{M}^j)_l \cdot \boldsymbol{\phi}^j_{j+\frac{1}{2}} 
	\nonumber
	\\ & \quad 
	+\sigma_{j-\frac{1}{2}} \F^{\star}\Big(\bold{u}^{j-1}(t)
	\cdot \boldsymbol{\phi}^{j-1}_{j-\frac{1}{2}},
	\bold{u}^{j}(t) \cdot \boldsymbol{\phi}^{j}_{j-\frac{1}{2}}
	\Big) (\mathbb{M}^j)_l \cdot 
	\boldsymbol{\phi}^j_{j-\frac{1}{2}}, 
	\raisetag{-15pt}
\end{align}
and $\partial_x \boldsymbol{\phi}^j 
= \begin{pmatrix} \partial_x \phi_0^j &  
\hdots & \partial_x \phi_k^j  \end{pmatrix}^T$. 

Next, by inserting $\varphi = \phi_m^j$ for 
$m=0,\ldots,k$ into \eqref{eq:weak1} 
and using \eqref{eq:qAsFunctionOfu1}, we infer
\begin{align*}
	\int_{I_j} &d\bold{u}^j(t) 
	\cdot \boldsymbol{\phi}^j \phi_m^j  \,dx 
	\\ &= - \frac{1}{2}\int_{I_j} \sigma \bold{u}^j(t) 
	\cdot \boldsymbol{\phi}^j \partial_x(\sigma' \phi_m^j) \,dx \, dt
	\\ & \qquad 
	+ \frac{1}{4} \big(\sigma^2\big)'_{j+\frac{1}{2}}
	\F\Bigl(\bold{u}^j(t) \cdot \boldsymbol{\phi}^j_{j+\frac{1}{2}},
	\bold{u}^{j+1}(t)\cdot \boldsymbol{\phi}_{j+\frac{1}{2}}^{j+1}, 
	\bold{Q}^j(\bold{u}(t)) 
	\cdot \boldsymbol{\phi}^j_{j+\frac{1}{2}},
	\bold{Q}^{j+1}(\bold{u}(t)) 
	\cdot \boldsymbol{\phi}^{j+1}_{j+\frac{1}{2}}
	\Bigr) \phi_{m, j+\frac{1}{2}}^j\,dt
	\\ & \qquad 
	-\frac{1}{4} \big(\sigma^2\big)'_{j-\frac{1}{2}}
	\F\Bigl(\bold{u}^{j-1}(t) 
	\cdot \boldsymbol{\phi}^{j-1}_{j-\frac{1}{2}},
	\bold{u}^{j}(t)\cdot \boldsymbol{\phi}_{j-\frac{1}{2}}^{j}, 
	\bold{Q}^{j-1}(\bold{u}(t)) 
	\cdot \boldsymbol{\phi}^{j-1}_{j-\frac{1}{2}},
	\bold{Q}^{j}(\bold{u}(t)) \cdot \boldsymbol{\phi}^{j}_{j-\frac{1}{2}}
	\Bigr) \phi_{m, j-\frac{1}{2}}^j\,dt
	\\ & \qquad - \frac{1}{2}\int_{I_j} \bold{Q}^j(\bold{u}(t))
	\cdot \boldsymbol{\phi}^j \partial_x(\sigma \phi_m^j)\, dx\, dt
	\\ & \qquad 
	+ \frac{1}{2}\sigma_{j+\frac{1}{2}}
	\widetilde{\F}\Bigl(\bold{u}^j(t) \cdot 
	\boldsymbol{\phi}^j_{j+\frac{1}{2}},
	\bold{u}^{j+1}(t)\cdot \boldsymbol{\phi}_{j+\frac{1}{2}}^{j+1}, 
	\bold{Q}^j(\bold{u}(t))\cdot \boldsymbol{\phi}^j_{j+\frac{1}{2}},
	\bold{Q}^{j+1}(\bold{u}(t))\cdot 
	\boldsymbol{\phi}^{j+1}_{j+\frac{1}{2}} \Bigr)
	\phi_{m, j+\frac{1}{2}}^j\,dt
	\\ & \qquad 
	- \frac{1}{2}\sigma_{j-\frac{1}{2}}
	\widetilde{\F}\Bigl(\bold{u}^{j-1}(t) 
	\cdot \boldsymbol{\phi}^{j-1}_{j-\frac{1}{2}}, \bold{u}^{j}(t)
	\cdot \boldsymbol{\phi}_{j-\frac{1}{2}}^{j},
	\bold{Q}^{j-1}(\bold{u}(t))
	\cdot \boldsymbol{\phi}^{j-1}_{j-\frac{1}{2}},
	\bold{Q}^{j}(\bold{u}(t))\cdot 
	\boldsymbol{\phi}^{j}_{j-\frac{1}{2}} 
	\Bigr)\phi_{m, j-\frac{1}{2}}^j\,dt
	\\ & \qquad 
	- \int_{I_j} \bold{Q}^j(\bold{u}(t)) 
	\cdot \boldsymbol{\phi}^j \phi_m^j\,dx \,dW_t. 
\end{align*}
As a consequence, we obtain the following 
matrix-valued system of SDEs: 
\begin{equation}\label{eq:sDESystem}
	d\bold{u}(t) = F(\bold{u}(t))\, dt 
	+ G(\bold{u}(t))\, dW_t, 
\end{equation}
where the drift matrix $F:\R^{(k+1)\times N} 
\rightarrow \R^{(k+1)\times N}$ and the diffusion 
matrix $G: \R^{(k+1)\times N} \rightarrow \R^{(k+1)\times N}$ 
have elements given by 
\begin{align*}
	F_l^j(\bold{u}) &:= - \frac{1}{2} \int_{I_j} 
	\sigma \bold{u}^j \cdot \boldsymbol{\phi}^j (\mathbb{M}^j)_l 
	\cdot \partial_x \big(\sigma' \boldsymbol{\phi}^j \big)\, dx 
	\nonumber 
	\\ & \qquad 
	+ \frac{1}{4} \big(\sigma^2 \big)'_{j+\frac{1}{2}} 
	\F \Bigl( \bold{u}^j \cdot \boldsymbol{\phi}^j_{j+\frac{1}{2}},
	\bold{u}^{j+1}\cdot \boldsymbol{\phi}^{j+1}_{j+\frac{1}{2}}, 
	\bold{Q}^j(\bold{u}) \cdot \boldsymbol{\phi}^j_{j+\frac{1}{2}},
	\bold{Q}^{j+1}(\bold{u})\cdot 
	\boldsymbol{\phi}^{j+1}_{j+\frac{1}{2}}\Bigr) 
	(\mathbb{M}^j)_l \cdot \boldsymbol{\phi}^j_{j+\frac{1}{2}}
	\nonumber \\ & \qquad 
	- \frac{1}{4} \big(\sigma^2 \big)'_{j-\frac{1}{2}} 
	\F \Bigl( \bold{u}^{j-1} \cdot \boldsymbol{\phi}^{j-1}_{j-\frac{1}{2}},
	\bold{u}^{j}\cdot \boldsymbol{\phi}^{j}_{j-\frac{1}{2}}, 
	\bold{Q}^{j-1}(\bold{u}) \cdot \boldsymbol{\phi}^{j-1}_{j-\frac{1}{2}},
	\bold{Q}^{j}(\bold{u}) \cdot \boldsymbol{\phi}^{j}_{j-\frac{1}{2}}\Bigr)
	(\mathbb{M}^j)_l \cdot \boldsymbol{\phi}^j_{j-\frac{1}{2}}\nonumber
	\\ & \qquad 
	- \frac{1}{2}\int_{I_j} \pigl(\bold{Q}^j(\bold{u}) 
	\cdot \boldsymbol{\phi}^j \pigr) \big(\mathbb{M}^j 
	\big)_l \cdot \partial_x \big(\sigma \boldsymbol{\phi}^j \big) 
	\, dx \nonumber
	\\ & \qquad + \frac{1}{2}\sigma_{j+\frac{1}{2}}
	\widetilde{\F}\Bigl( \bold{u}^j \cdot 
	\boldsymbol{\phi}^j_{j+\frac{1}{2}}, \bold{u}^{j+1}
	\cdot \boldsymbol{\phi}^{j+1}_{j+\frac{1}{2}}, \bold{Q}^j(\bold{u})
	\cdot \boldsymbol{\phi}^j_{j+\frac{1}{2}},
	\bold{Q}^{j+1}(\bold{u})\cdot 
	\boldsymbol{\phi}^{j+1}_{j+\frac{1}{2}} \Bigr)
	\big(\mathbb{M}^j\big)_l \cdot 
	\boldsymbol{\phi}_{j+\frac{1}{2}}^j\nonumber
	\\ & \qquad 
	- \frac{1}{2}\sigma_{j-\frac{1}{2}}
	\widetilde{\F}\Bigl(\bold{u}^{j-1} 
	\cdot \boldsymbol{\phi}^{j-1}_{j-\frac{1}{2}},
	\bold{u}^{j}\cdot \boldsymbol{\phi}^{j}_{j-\frac{1}{2}}, 
	\bold{Q}^{j-1}(\bold{u})\cdot \boldsymbol{\phi}^{j-1}_{j-\frac{1}{2}},
	\bold{Q}^{j}(\bold{u})\cdot \boldsymbol{\phi}^{j}_{j-\frac{1}{2}}
	\pigr) \big(\mathbb{M}^j\big)_l \cdot 
	\boldsymbol{\phi}_{j-\frac{1}{2}}^j, 
	\\ 
	G_l^j(\bold{u}) & 
	:= -\int_{I_j}\bold{Q}^j(\bold{u}) 
	\cdot \boldsymbol{\phi}^j \Big((\mathbb{M}^j)_l 
	\cdot \boldsymbol{\phi}^j \Big)\,dx. 
\end{align*}

\begin{remark}
These expressions show that the width 
of the numerical stencil is determined by 
the choice of numerical fluxes rather 
than by the local polynomial degree. In 
general, the evolution of $u_h$ in a 
cell $I_j$ depends not only on its immediate 
neighbors but also on their neighbors, 
typically resulting in a five-point stencil 
(see also Section \ref{sec:FD}).
\end{remark}

It follows from \eqref{eq:weak3} that 
the entries of the initial coefficient matrix $\bold{u}(0)$ 
are determined from $\bar{u}$ in the following way, 
 \begin{equation*}
 	u_l^j(0) = \int_{I_j} \bar{u}(x) 
 	\bigl( \mathbb{M}^j\bigr)_l 
 	\cdot \boldsymbol{\phi}^j(x)\,dx,
\end{equation*}
and since $\bar{u}\in L^2(\R)$ is assumed to 
be deterministic so is $\bold{u}(0)$. 
Thus, for any $p \geq 1$, 
\begin{equation}\label{eq:Lp-integrable}
	\E \pigl [ \|\bold{u}(0)\|_p^p \pigr]
	= \sum_{j=0}^N
	\sum_{l=0}^k|u_l^j(0)|^p<\infty. 
\end{equation}
 
Under the assumed local Lipschitz 
continuity and linear growth conditions on 
the numerical fluxes (cf.~\eqref{eq:localLipschitz} 
and \eqref{eq:linearGrowth}), the drift and 
diffusion matrices $F$ and $G$ are 
locally Lipschitz and of at most linear growth 
in $\bold{u}$. This can be proved by following the argument in
\cite[Sec.~3.2]{Li:2021aa}. By Theorem \ref{thm:SDEStandard} and 
\eqref{eq:Lp-integrable}, it  thus follows that 
\eqref{eq:sDESystem} admits a unique 
probabilistic strong $L^p$ solution, 
up to indistinguishability. We summarize 
these results in the following theorem.

\begin{theorem}\label{thm:wellPosed}
Fix the discretization parameter $h > 0$ 
and pick three numerical fluxes 
$\F(\cdot, \cdot, \cdot, \cdot)$, 
$\widetilde{\F}(\cdot, \cdot, \cdot, \cdot)$, 
and $\F^{\star}(\cdot, \cdot)$ 
that are locally Lipschitz and grow no 
faster than linearly in their respective arguments, i.e., 
they satisfy \eqref{eq:localLipschitz} 
and \eqref{eq:linearGrowth}. Then the 
system \eqref{eq:sDESystem} admits a unique 
solution $\bold{u}:\Omega \times [0, T] 
\rightarrow \R^{(k+1)\times N}$, 
such that for any $p \in [1,\infty)$,
\begin{equation}\label{eq:numStabu}
	\E \Bigl[\, \sup_{0\leq t\leq T}
	|\bold{u}(t)|^p \Bigr]<\infty. 
\end{equation}
\end{theorem}

\begin{remark}\label{rem:GrowthQ}
Since $\bold{Q}$ grows no faster 
than linearly in $\bold{u}$ by 
\eqref{eq:qAsFunctionOfu1}--\eqref{eq:qAsFunctionOfu2}, 
it follows from \eqref{eq:numStabu} that
$$\E \Bigl[ \sup_{0 \leq t \leq T}|\bold{q}(t)|^p \Bigr] < \infty.$$ 
\end{remark}

\subsection{Stability estimates}\label{sec3:StabEstimates}
As discussed in the introduction, the 
numerical fluxes and auxiliary variables 
in the LDG discretization are carefully 
chosen to preserve the hyperbolic energy structure 
induced by gradient noise, and in particular 
the delicate cancellation between 
the It\^{o}--Stratonovich correction and 
the quadratic variation. This structure 
is essential for preventing spurious 
dissipation or artificial energy growth at 
the discrete level.

Having established that the proposed 
LDG formulation is well posed, we now derive 
stability estimates in the present 
model setting. When $\sigma$ is constant we obtain 
pathwise stability estimates, while more generally we 
prove $L^2$ stability in the mean-square sense. 
The pathwise estimates extend 
to the multidimensional, 
divergence-free case $\nabla\cdot\sigma=0$, 
which is treated in a later section together with 
nonlinear equations. Throughout 
this section, we restrict attention to 
equal-order approximation spaces.

As our starting point, we exploit 
\eqref{eq:semimartingaleID} 
with $X_t = u_h(t) = Y_t$ (or equivalently \eqref{eq:ItoFormula}),   that is, 
\begin{align}\label{eq:basis}
	|u_h(t)|^2 &= |\bar{u}_h|^2
	+2\!\int_0^t u_h\,du_h(s)
	+\big \langle u_{h}(\cdot), u_{h}(\cdot) \big\rangle_t.
\end{align}
The quadratic covariation of $u_h$ is 
given by the next lemma.

\begin{lemma}\label{lem:expressionCoVar}
Let $u_h$ and $q_h$ be computed by
the LDG equations \eqref{eq:weak1}--\eqref{eq:weak3}.
The following identity holds: 
\begin{align*}
	\int_{I_j} \big \langle u_{h}(\cdot),
	u_{h}(\cdot) \big\rangle_t
	\, dx &= \int_0^t \|q_h(s)\|_{L^2(I_j)}^2\, ds 
	\\ &
	= -\int_0^t \int_{I_j}\partial_x(q_{h}(s))
	\sigma u_h(s)\, dx \, ds 
	+ \int_0^t\Big( \Big(\sigma 
	\widetilde{\F}_u(s) q_{h}^-(s)\Big)_{j+\frac{1}{2}}
	-\Big(
	\sigma \widetilde{\F}_{u}(s)q_{h}^+(s)
	\Big)_{j-\frac{1}{2}}\Big)\, ds, 
\end{align*}
where $\widetilde{\F}_u$ is a general 
numerical flux (see Section \ref{sec:num-fluxes}). 
\end{lemma}

\begin{proof}
By integrating \eqref{eq:weak1} 
in time, from $0$ to $t$, we can write
\begin{align}\label{eq:tempProofStab}
	\int_{I_j} \varphi u_{h}(\omega, t)\,dx 
	&= R_j(\omega, t;\varphi)
	-\int_0^t\int_{I_j}\varphi q_{h}(\omega, s)
	\,dx \,dW_s(\omega), 
\end{align}
where $R_j(t; \varphi)$ 
is the bounded variation process given by
\begin{align*}
	R_j(t; \varphi) & = 
	\int_{I_j} \varphi \bar{u}_{h}\,dx 
	+ \int_0^t \bigl(a_j(u_h(s), \varphi)
	+ b_j(q_h(s), \varphi) \bigr)\, ds,
\end{align*}
and $a_j(\cdot, \cdot)$ and $b_j(\cdot, \cdot)$ 
denote the bilinear forms defined 
in \eqref{eq:bilinearForms} (suppressing 
the $\omega$-dependency).

As a consequence of \eqref{eq:vanishingCovar} 
and \eqref{eq:tempProofStab}, 
it holds for any continuous semimartingale 
$Y = \{Y_t \! \mid 0 \leq t \leq T \}$ that 
\begin{align}\label{eq:exploitCovar}
	\int_{I_j} \varphi \big\langle u_{h}(\cdot), Y \big \rangle_t \, dx 
	&=  \left \langle \int_{I_j}\varphi u_{h}(\cdot)
	\, dx, Y \right \rangle_{t} 
	=\left \langle -\int_0^{\cdot} \int_{I_j}
	\varphi q_{h}(s)\, dx\, dW_s, Y\right \rangle_{t}\!. 
\end{align}
Let $\{\phi_l\}_{l=0}^k$ be an 
arbitrary basis of $\mathcal{P}^k(I_j)$. The numerical 
solution $u_h$ admits the local expansion
\begin{equation}\label{eq:expansionStab}
	u_h(t,x)=\sum_{l=0}^k u_l(t)\phi_l(x),
	\qquad x\in I_j,
\end{equation}
where, for notational simplicity, we have 
suppressed the cell indices 
used in \eqref{eq:basisExpansion}, allowing 
the basis to vary from cell to cell. 
This expansion, in conjunction 
with \eqref{eq:exploitCovar} 
for $\varphi = \phi_l$, yields
\begin{multline*}
	\int_{I_j} \big \langle u_{h}(\cdot), u_{h}(\cdot)
	\big \rangle_t \, dx 
	= \int_{I_j} \left \langle u_{h}(\cdot), 
	\sum_{l=0}^k u_l(\cdot) \phi_l \right \rangle_t \,dx
	\\ = \sum_{l=0}^k \int_{I_j} \phi_l 
	\big \langle u_{h}(\cdot), u_l(\cdot) \big \rangle_t \, dx 
	= \sum_{l=0}^k \left \langle -\int_0^{\cdot}
	\int_{I_j} \phi_l(x)q_{h}(s)\, dx 
	\, dW_s, u_l(\cdot) \right\rangle_t.
\end{multline*}
Furthermore, as the stochastic integrand
$\left \{ \int_{I_j} \phi_lq_{h}(s)
\,dx \mid 0 \leq s \leq T \right\}$ 
is locally bounded and adapted for 
every $l=0,\ldots,k$, it follows from 
\cite[Thm.~29 in Chap.~2]{Protter:2005aa} that
\begin{equation}\label{eq:temp1}
	\begin{split}
		\sum_{l=0}^k & \bigg \langle -\int_0^{\cdot} \int_{I_j} \phi_lq_{h}(s)
		\, dx \, dW_s, u_l(\cdot) \bigg\rangle_t 
		= -\sum_{l=0}^k \int_0^t \int_{I_j}\phi_lq_{h}(s)
		\, dx\, d\big \langle W, u_l(\cdot)\big \rangle_s  
		\\ & 
		= - \int_{I_j} \int_0^t q_{h}(s) \, d\left \langle W, 
		\sum_{l=0}^k u_l(\cdot)\phi_l \right \rangle_s \, dx 
		= -\int_{I_j} \left \langle \int_0^{\cdot}q_{h}(s)
		\, dW_s, u_{h}(\cdot) \right \rangle_t \, dx.
	\end{split}
\end{equation}
Moreover, since $q_{h}(s)$ 
belongs to $\mathcal{V}^k$ (see 
Section \ref{sec:prelim}) it admits an expansion 
on the same form as in \eqref{eq:expansionStab}. 
Inserting this expansion into \eqref{eq:temp1} yields 
\begin{align}\label{eq:quad-var-uh}
	-\int_{I_j} \bigg \langle \int_0^{\cdot}q_{h}(s) 
	\, dW_s, u_{h}(\cdot) \bigg \rangle_t \, dx 	
	&= -\sum_{l=0}^k \int_{I_j}\phi_l \left \langle u_{h}(\cdot),
	\int_0^{\cdot}q_l(s)\, dW_s \right \rangle_t \, dx \nonumber
	\\ &
	= -\sum_{l=0}^k \left \langle -\int_0^{\cdot}\int_{I_j}\phi_lq_{h}(s)
	\, dx \, dW_s, \int_0^{\cdot}q_l(s)\, dW_s \right \rangle_t
	\\ &
	= \sum_{l=0}^k \int_0^t \int_{I_j}\phi_lq_{h}(s) q_l(s) 
	\, dx  \, d \big\langle W, W \big\rangle_s \nonumber
	=  \int_0^t \int_{I_j} q_{h}^2(s) \, dx \, ds, \nonumber
\end{align}
where we used \eqref{eq:exploitCovar} 
with $\varphi = \phi_l$ and 
$Y = \int_0^{\cdot}q_l(s)\, dW_s$, and 
again exploited the local boundedness and 
adaptedness of $\int_{I_j}\phi_{l} q_h(s)\, dx$. 
Finally,  by leveraging \eqref{eq:weak2} with $\psi= q_{h}$ 
we obtain the asserted identity. 
\end{proof}

\begin{remark}
The proof presented above works for unequal-order 
approximation spaces $\mathcal{V}^k$ and 
$\mathcal{Q}^l$ as long as the basis 
elements of $\mathcal{Q}^l$ 
belongs to $\mathcal{V}^k$. 
\end{remark}

Next, we want to examine the second term in \eqref{eq:basis}, 
to which we insert the test function $\varphi
=u_{h}(s)$ in \eqref{eq:weak1} and integrate 
in time from $0$ up to $t$. This yields
\begin{align}\label{eq:expudu}
	 \int_{I_j} \int_0^t u_{h}(s)\, du_{h}(s)\, dx 
	 &= \int_0^t \bigl(a_j(u_h(s), u_h(s)) 
	 + b_j(q_h(s), u_h(s)) \bigr)\, ds 
	-  \int_0^t \int_{I_j}u_h(s)q_h(s)\, dx \, dW_s. 
\end{align}

To proceed from this point on, we start by 
deducing an alternative expression for 
the quantity $a_j(u_h, u_h)$.

\begin{lemma}\label{lem:aj}
Let $u_h$ and $q_h$ be computed by
the LDG equations \eqref{eq:weak1}--\eqref{eq:weak3}. 
The term $a_j(u_h, u_h)$, as defined 
by \eqref{eq:bilinearForms}, can 
be expressed as follows:
\begin{align*}
	a_j(u_h(t), u_h(t)) &= \frac{1}{2}\int_{I_j}
	\Bigl((\sigma')^2-\frac{1}{4}(\sigma^2)'' \Bigr)
	u_h^2(t)\, dx 
	+ \Bigl( \Phi_{j+\frac{1}{2}}(t)
	-\Phi_{j-\frac{1}{2}}(t)
	+\Psi_{j-\frac{1}{2}}(t) \Bigr),
\end{align*}
where the entropy flux 
$\Phi_{j+\frac{1}{2}}$ 
and the remainder term 
$\Psi_{j-\frac{1}{2}}$ are defined by 
\begin{align}
	\Phi_{j+\frac{1}{2}} & := 
	\frac{1}{8}\pigl( \bigl(\sigma^2\bigr)'\bigl(2\F_u 
	-u_{h}^-\bigr)u_h^- \pigr)_{j+\frac{1}{2}}, 
	\label{eq:phi} \\ 
	\Psi_{j-\frac{1}{2}} & := 
	\frac{1}{8}\pigl(\bigl(\sigma^2\bigr)'
	\bigl(\llbracket u_{h}^2 \rrbracket 
	- 2\F_u \llbracket u_{h}
	\rrbracket \bigr) \pigr)_{j-\frac{1}{2}}. 
	\label{eq:psi}
\end{align}
The numerical flux $\F_u$ is 
defined in \eqref{eq:firstFlux1}, 
and the jump notation $\llbracket\cdot\rrbracket$ 
is introduced in Section \ref{sec:prelim}.
\end{lemma}

\begin{proof}
Since $\sigma \in C^2(\R)$, applying 
the product rule, the chain rule, 
and integrating by parts, 
enables us to write
\begin{align*}
	- \int_{I_j} \partial_x(\sigma'u_{h})\sigma u_{h}\, dx\, 
	& =  -\int_{I_j}\Big(\sigma \sigma''u_{h}^2
	+\frac{1}{2} \sigma \sigma' \partial_x\big(u_h^2\big)\Big)\, dx 
	\nonumber
	\\ &
	= \int_{I_j}\Big(\frac{1}{4}\big(\sigma^2\big)''
	- \sigma \sigma'' \Big)u_{h}^2 \, dx 
	- \frac{1}{4}\Big ( \pigl(\bigl(\sigma^2\bigr)'
	\bigl(u_{h}^-\bigr)^2\pigr)_{j+\frac{1}{2}}
	-\pigl(\bigl(\sigma^2\bigr)'
	\bigl(u_{h}^+\bigr)^2\pigr)_{j-\frac{1}{2}} \Big).
\end{align*}
Moreover, since $\sigma\sigma''
=\tfrac12(\sigma^2)''-(\sigma')^2$, 
recalling the definition \eqref{eq:bilinearForms} 
of $a_j(\cdot,\cdot)$ yields
\begin{align*}
	a_j(u_h, u_h)&
	= \frac{1}{2}\int_{I_j} 
	\Bigl( (\sigma')^2 - \frac{1}{4}(\sigma^2)'' \Bigr)u_h^2\, dx 
	\\ & \quad 
	+ \frac{1}{8} \pigl( \bigl(\sigma^2\bigr)' 
	\bigl(2\F_u - u_h^-)u_h^-\pigr)_{j+\frac{1}{2}}
	-\frac{1}{8} \pigl( \bigl(\sigma^2\bigr)' \bigl( 2\F_u
	-u_h^+\bigr)u_h^+ \pigr)_{j-\frac{1}{2}};
\end{align*}
by adding and subtracting 
$\frac{1}{8}\pigl( \bigl(\sigma^2\bigr)' 
\bigl( 2\F_u - u_h^-\bigr)u_h^- \pigr)_{j-\frac{1}{2}}$ 
and subsequently introducing $\Phi_{j+\frac{1}{2}}$ 
and $\Psi_{j-\frac{1}{2}}$ 
as stated, the asserted identity follows. 
\end{proof}

We continue with an expression for the $b$ term, which 
demonstrates that the scheme exhibits cancellation 
between the quadratic variation and dissipation 
from the correction term, up to numerical fluxes.

\begin{lemma}\label{lem:bj}
Let $u_h$ and $q_h$ be computed by
the LDG equations \eqref{eq:weak1}--\eqref{eq:weak3}. 
The following identity holds: 
\begin{align*}
	2\int_0^t b_j(q_h(s), u_h(s))\, ds 
	+ \int_{I_j} \langle u_h(\cdot), u_h(\cdot)
	\rangle_t \, dx  
	= \int_0^t \Bigl(\Lambda_{j+\frac{1}{2}}(s)
	-\Lambda_{j-\frac{1}{2}}(s)
	+\Theta_{j-\frac{1}{2}}\Bigr)\, ds,
\end{align*}
where $b$ is defined in \eqref{eq:expudu}, 
the entropy flux $\Lambda_{j+\frac{1}{2}}$ 
and remainder term $\Theta_{j-\frac{1}{2}}$ 
are given by 
\begin{align}
	\Lambda_{j+\frac{1}{2}} 
	& := \pigl(\sigma \bigl(u_{h}^-\widetilde{\F}_q+
	\widetilde{\F}_uq_{h}^-
	-u_{h}^-q_{h}^-\bigr)\pigr )_{j+\frac{1}{2}}, 
	\label{eq:lambda} 
	\\ \Theta_{j-\frac{1}{2}} 
	& := \Big( \sigma \pigl(  \llbracket u_{h}q_{h}\rrbracket 
	- \widetilde{\F}_u\llbracket q_{h}\rrbracket 
	- \llbracket u_{h} \rrbracket \widetilde{\F}_q \pigr) 
	\Big)_{j-\frac{1}{2}},
	\label{eq:theta}
\end{align}
and $\widetilde{\F}_u$, $\widetilde{\F}_q$ are 
general numerical fluxes 
(see Section \ref{sec:num-fluxes}). 
The jump notation $\llbracket\cdot\rrbracket$ 
is introduced in Section \ref{sec:prelim}.
 
\end{lemma}

\begin{proof}
Recalling the identity from 
Lemma \ref{lem:expressionCoVar} and the definition 
\eqref{eq:bilinearForms} of $b_j(\cdot,\cdot)$, 
we infer that
\begin{multline*}
	2\int_0^t b_j(q_h(s), u_h(s))\, ds 
	+ \int_{I_j} \langle u_h(\cdot), u_h(\cdot) 
	\rangle_t \, dx 
	= - \int_0^t \int_{I_j} \Bigl( \partial_x 
	\bigl(\sigma u_h(s) \bigr)q_h(s) 
	+ \partial_x \bigl(q_h(s) \bigr)\sigma u_h(s) 
	\Bigr)\, dx \, ds
	\\  
	+ \int_0^t \pigl(\sigma \bigl(\widetilde{\F}_q(s)u_h^-(s)
	+\widetilde{\F}_u(s)q_h^-(s)\bigr)
	\pigr)_{j+\frac{1}{2}}\, ds 
	-\int_0^t \pigl(\sigma \bigl(u_h^+(s)\widetilde{\F}_q(s)
	+\widetilde{\F}_u(s)q_h^+(s)\bigr)\pigr)_{j-\frac{1}{2}}\, ds. 
\end{multline*}
``Differentiation by parts'' yields
\begin{align*}
	- \int_0^t\int_{I_j}
	\Bigl(\partial_x(\sigma u_{h}(s))&q_{h}(s)
	+\partial_x(q_{h}(s))\sigma u_{h}(s) \Bigr)\, dx \, ds
	= - \int_0^t\Big( \pigl( \sigma u_{h}^-(s)q_{h}^-(s)
	\pigr)_{j+\frac{1}{2}}-\pigl(\sigma u_{h}^+(s)q_{h}^+(s) 
	\pigr)_{j-\frac{1}{2}} \Big)\, ds, 
\end{align*}
As a result, we have
\begin{multline*}
	2\int_0^t b_j(q_h(s), u_h(s))\, ds
	+\int_{I_j} \langle u_h(\cdot), u_h(\cdot) \rangle_t \, dx 
	= \int_0^t \Bigl( \pigl( \sigma 
	\bigl(u_{h}^-(s)\widetilde{\F}_q(s)
	+\widetilde{\F}_u(s)q_{h}^-(s) - u_{h}^-(s)q_{h}^-(s)
	\bigr)\pigr)_{j+\frac{1}{2}} \nonumber
	\\
	- \pigl(\sigma \bigl(u_{h}^+(s)\widetilde{\F}_q(s)
	+\widetilde{\F}_u(s)q_{h}^+(s)
	-u_{h}^+(s)v_{h}^+(s)\bigr)
	\pigr)_{j-\frac{1}{2}}  \Big)\, ds, 
\end{multline*}
and by introducing $\Lambda_{j+\frac{1}{2}}$ 
and $\Theta_{j-\frac{1}{2}}$ 
(see \eqref{eq:lambda} and \eqref{eq:theta}) 
the identity follows. 
\end{proof}

Combining Lemma \ref{lem:expressionCoVar}, 
\eqref{eq:expudu}, and Lemmas~\ref{lem:aj}--\ref{lem:bj}, 
we arrive at the following conclusion:
\begin{align}\label{eq:temporarySec3}
	2\!\int_0^t \int_{I_j}  u_h\,du_h(s)\,dx
	&+ \int_{I_j} \bigl\langle u_{h}(\cdot), 
	u_{h}(\cdot) \bigr \rangle_t\, dx
	= 2 \int_0^t a_j(u_h(s), u_h(s))\, ds
	-2\int_{I_j}\int_0^t u_h(s)q_h(s)\, dx \, dW_s 
	\nonumber
	\\ & \qquad 
	+ 2\int_0^t b_j(q_h(s), u_h(s)) \bigr)\, ds dx
	+ \int_{I_j}\langle u_h(\cdot), u_h(\cdot) \rangle_t\, dx 
	\nonumber
	\\ &= \int_0^t\int_{I_j} \Bigl(\bigl(\sigma'\bigr)^2 
	-\frac{1}{4}\bigl(\sigma^2\bigr)''\Bigr)u_h(s)^2\, dx\, ds 
	+2\int_{0}^t \Bigl( \Phi_{j+\frac{1}{2}}(s) 
	- \Phi_{j-\frac{1}{2}}(s) + \Psi_{j-\frac{1}{2}}(s) \Bigr)\, ds 
	\nonumber
	\\ & \qquad  
	+ \int_0^t \Bigl(\Lambda_{j+\frac{1}{2}}(s)
	-\Lambda_{j-\frac{1}{2}}(s) + \Theta_{j-\frac{1}{2}}(s)
	\Bigr)\, ds 
	 - 2\int_0^t \int_{I_j}u_h(s)q_h(s)
	\, dx \, dW_s.
\end{align}
Everything up to this point has been 
with general numerical fluxes $\F_u$, $\widetilde{\F}_u$, 
$\widetilde{\F}_q$, but when $\F_u$ is 
chosen according to \eqref{eq:firstFlux1} 
and the flux pair $(\widetilde{\F}_u, \widetilde{\F}_q)$ 
is on the form \eqref{eq:firstFlux2}, 
the remainder terms involving 
$\Psi$ and $\Theta$ turn out to have 
negative signs. Here $\eta_u$ and $\eta_q$ are 
positive real-numbers, which 
in principle may be cell dependent. 

\begin{lemma}\label{lem:fluxes}
Let $\Psi$ and $\Theta$ be the remainder 
terms in \eqref{eq:psi} and \eqref{eq:theta}, 
respectively. With $\F_u$ from \eqref{eq:firstFlux1}, 
\begin{align}\label{eq:nonnegative1}
	\Psi_{j-\frac{1}{2}} & 
	= -\frac{1}{8}(\gamma+2\widetilde{\gamma})
	\Bigl( \pigl|\bigl(\sigma^2\bigr)'
	\pigr|\llbracket u_h \rrbracket^2 
	\Bigr)_{j-\frac{1}{2}} \leq 0, 
\end{align}
and if the numerical flux pair $(\widetilde{\F}_u, 
\widetilde{\F}_q)$ is selected in 
accordance with \eqref{eq:firstFlux2}, then 
\begin{align}\label{eq:nonnegative2}
	\Theta_{j-\frac{1}{2}} 
	= - \Bigl( \bigl|\sigma \bigr|
	\pigl( \eta_u \llbracket q_h \rrbracket^2 
	+ \eta_q \llbracket u_h \rrbracket^2 \pigr) 
	\Bigr)_{j-\frac{1}{2}} \leq 0. 
\end{align}
\end{lemma} 
\vspace{-0.3cm}

\begin{proof}
We start by proving \eqref{eq:nonnegative1}, 
to which we recall the identity
\begin{equation}\label{eq:jumpFormula}
	\llbracket a b \rrbracket 
	= \avg{a}\llbracket b\rrbracket 
	+ \llbracket a \rrbracket \avg{b},
\end{equation}
where the reader should consult \eqref{eq:jumpNotation} 
and \eqref{eq:avgNotation} 
for the notation. 
Further, note that
\begin{align*}
	\llbracket u_h \rrbracket^2
	=\gamma (u_h^+ + u_h^-) \llbracket u_h \rrbracket 
	+ 2(1-\gamma)  \avg{u_h}\llbracket u_h \rrbracket,
\end{align*}
thus, inserting \eqref{eq:firstFlux1} 
for $\F_u$, we immediately get
\begin{align*}
	\Psi_{j-\frac{1}{2}} &=
	\frac{1}{8} \pigl(\bigl(\sigma^2\bigr)'
	\bigl(\llbracket u_{h}^2 \rrbracket - 2\F_u \llbracket u_{h} 
	\rrbracket \bigr) \pigr)_{j-\frac{1}{2}}
	\\ & 
	= \frac{1}{8} \gamma\pigl(\bigl(\sigma^2\bigr)'
	\bigl((u_h^+ + u_h^-) \llbracket u_h \rrbracket
	-2 \widecheck{\F}_{u_h}\bigl((\sigma^2)'\bigr) 
	\llbracket u_{h} \rrbracket \pigr)_{j-\frac{1}{2}}
	-\frac{1}{4} \widetilde{\gamma} 
	\Bigl(\bigl(\sigma^2\bigr)'\mathrm{sgn}{\bigl(\sigma^2\bigr)'} 
	\llbracket u_h \rrbracket^2\Bigr)_{j-\frac{1}{2}}. 
\end{align*}
The second term is clearly nonpositive. 
For the first term, assume that 
$(\sigma^2)'\ge 0$, so that 
$\widecheck{\F}_{u_h}\bigl((\sigma^2)'\bigr)=u_h^+$, 
from which we infer that
\begin{align*}
	\Psi_{j-\frac{1}{2}} 
	&= \frac{1}{8}\gamma \pigl(\bigl(\sigma^2 \bigr)'
	( (u_h^- -u_h^+)\llbracket u_h 
	\rrbracket)\pigr)_{j-\frac{1}{2}} 
	=-\frac{1}{8}\gamma\pigl(\bigl(\sigma^2 \bigr)'
	\llbracket u_h \rrbracket^2 \pigr)_{j-\frac{1}{2}}
	\leq 0. 
\end{align*}
Similarly, if $(\sigma^2)'<0$, then 
$\widecheck{\F}_{u_h}\bigl((\sigma^2)'\bigr)=u_h^-$, 
and we obtain 
$\Psi_{j-\frac{1}{2}} 
= \frac{1}{8}\gamma \pigl(\bigl(\sigma^2 \bigr)'
\llbracket u_h \rrbracket^2\pigr)_{j-\frac{1}{2}} \leq 0$. 
Hence, we may write
\begin{equation*}
	\Psi_{j-\frac{1}{2}}= - \frac{1}{8}\bigl(\gamma 
	+ 2\widetilde{\gamma}\bigr)
	\Bigl(\pigl| \bigl(\sigma^2 \pigr)' \pigr|
	\llbracket u_h \rrbracket^2 
	\Bigr)_{j-\frac{1}{2}}, 
\end{equation*}
which in turn proves \eqref{eq:nonnegative1}. 

We next turn to \eqref{eq:nonnegative2}. 
Recalling \eqref{eq:theta}, applying 
\eqref{eq:jumpFormula}, and inserting 
the flux pair \eqref{eq:firstFlux2} yields
\begin{align*}
	\Theta_{j-\frac{1}{2}} &
	= \Big( \sigma \pigl(  \llbracket u_hq_{h}\rrbracket 
	- \widetilde{\F}_u \llbracket q_{h}\rrbracket 
	- \llbracket u_{h} \rrbracket \widetilde{\F}_q \pigr) 
	\Big)_{j-\frac{1}{2}}
	\\ &
	= \Bigl( \sigma \pigl( \avg{u_h}\llbracket q_h \rrbracket 
	+ \llbracket u_h \rrbracket \avg{q_h}
	-\bigl(\theta u_h^- + (1-\theta)u_h^+\bigr)
	\llbracket q_h \rrbracket 
	- \llbracket u_h \rrbracket\bigl((1-\theta)q_h^-
	+ \theta q_h^+\bigr)
	\pigr) \Bigr)_{j-\frac{1}{2}}
	\\ &\qquad  
	- \Bigl( \bigl|\sigma \bigr|\pigl( \eta_u 
	\llbracket q_h \rrbracket^2 
	+ \eta_q \llbracket u_h \rrbracket^2 \pigr) 
	\Bigr)_{j-\frac{1}{2}}.
\end{align*}
The term on the last line is clearly nonpositive. 
Moreover, we find that 
\begin{equation*}
	\avg{u_h} \llbracket q_h \rrbracket 
	- \bigl(\theta u_h^- + (1-\theta)u_h^+ \bigr) 
	\llbracket q_h \rrbracket = \left( \theta
	-\frac{1}{2} \right) \llbracket u_h \rrbracket 
	\llbracket q_h \rrbracket, 
\end{equation*}
and 
\begin{equation*}
	 \llbracket u_h \rrbracket \avg{q_h} 
	 - \llbracket u_h \rrbracket\bigl((1-\theta)q_h^- 
	 + \theta q_h^+\bigr) 
	 = - \left( \theta  -\frac{1}{2} \right) 
	 \llbracket u_h \rrbracket \llbracket q_h \rrbracket,
\end{equation*}
which in turn implies that 
\begin{equation*}
	 \Bigl( \sigma \pigl( \avg{u_h}\llbracket q_h \rrbracket 
	+ \llbracket u_h \rrbracket \avg{q_h} - \bigl(\theta u_h^- 
	+ (1-\theta)u_h^+\bigr)
	\llbracket q_h \rrbracket 
	- \llbracket u_h \rrbracket\bigl((1-\theta)q_h^- 
	+ \theta q_h^+\bigr)
	\pigr) \Bigr)_{j-\frac{1}{2}} = 0 \qedhere
\end{equation*}
\end{proof}

\begin{remark}
Lemma \ref{lem:fluxes} shows that 
$\Psi_{j-\frac12}=0$ when $\gamma=\widetilde{\gamma}=0$, 
corresponding to the use of a central flux for $\F_u$ 
and hence to the absence of dissipation 
associated with this flux. Likewise, 
$\Theta_{j-\frac12}=0$ when the penalty 
parameters $\eta_u$ and $\eta_q$ both 
vanish. Thus, the flux pair 
$(\widetilde{\F}_u,\widetilde{\F}_q)$ contributes to 
the dissipation of the LDG 
scheme through penalization of jumps in 
$u_h$ and $q_h$, respectively, 
in close analogy with LDG discretizations of 
second-order elliptic problems 
\cite{DGEllipticProblems,UnifiedAnalysis,ApriorLDGElliptic}. 
From a computational perspective, it is advantageous 
to choose $\eta_u=0$, since this allows the 
elimination of $q_h$ using only 
the inversion of a block-diagonal 
mass matrix \cite{Cockburn:1998ai}.
\end{remark}

The last term in 
\eqref{eq:temporarySec3} has zero expectation. 

\begin{lemma}\label{lem:VanishingExp}
Let $u_h$ and $q_h$ be computed by
the LDG equations \eqref{eq:weak1}--\eqref{eq:weak3}. 
Then
\begin{equation}\label{eq:zero-exp}
	\E\bigg[\int_0^t \int_{I_j} 
	u_h(s) q_h(s)\,dx\,dW_s\bigg]=0.
\end{equation}
\end{lemma}

\begin{proof}
We apply Lemma \ref{lem:Martingale} with
$\mathcal{H}_s := \int_{I_j} u_h(s) q_h(s)\,dx$.
By the Cauchy--Schwarz inequality,
$\mathcal{H}_s^2 \le 
\|u_h(s)\|_{L^2(I_j)}^2\,\|q_h(s)\|_{L^2(I_j)}^2$. 
Consequently,
\begin{align*}
	\E\bigg[\Big(\int_0^t
	\mathcal{H}_s^2\,ds\Big)^{1/2}\bigg]
	\le \sqrt{t}\,
	\E\Big[\sup_{0\le s\le T}\|u_h(s)\|_{L^2(I_j)}
	\sup_{0\le s\le T}\|q_h(s)\|_{L^2(I_j)}\Big].
\end{align*}
Applying Young’s product inequality, 
the right-hand side can be written as a sum 
of the corresponding squared norms. 
The finiteness of this sum follows from the 
($h$-dependent) stability bounds provided by 
Theorem \ref{thm:wellPosed} and Remark~\ref{rem:GrowthQ}. 
Lemma \ref{lem:Martingale} then 
yields the claimed identity.
\end{proof}

\begin{remark}\label{rem:BDG}
By the Burkholder--Davis--Gundy (BDG) 
inequality one obtains the following size estimate 
for the stochastic integral, 
valid for any $p\ge 1$:
\begin{align}\label{eq:BDG1}
	\E\bigg[\sup_{0\le r\le t}\Big|\int_0^r
	\int_{\R} u_h(s) q_h(s)\,dx\,dW_s\Big|^p\bigg]
	\le C_p\,
	\E\bigg[\Big(\int_0^t
	\Big(\int_{\R}|u_h(s)q_h(s)|\,dx\Big)^2 ds\Big)^{p/2}\bigg].
\end{align}
Using Cauchy--Schwarz in $x$ and 
a standard product estimate in time, this 
implies that for any $\varepsilon>0$ there exists a constant 
$C_{p,\varepsilon}>0$ such that
\begin{align*}
	&\E\bigg[\sup_{0\le r\le t}\Big|\int_0^r
	\int_{\R} u_h(s) q_h(s)\,dx\,dW_s\Big|^p\bigg] \\
	&\qquad\le \varepsilon\,\E\Big[\sup_{0\le s\le t}
	\|u_h(s)\|_{L^2(\R)}^{2p}\Big]
	+ C_{p,\varepsilon}\,
	\E\bigg[\Big(\int_0^t \|q_h(s)\|_{L^2(\R)}^2\,ds\Big)^p\bigg],
\end{align*}
where the estimate remains valid for all $p\ge 1$, 
with the case $p=1$ being 
potentially useful for us. 
While the first estimate \eqref{eq:BDG1} only 
involves the $L^2$ norm of the product $u_h q_h$, 
the second separates the $L^\infty_tL^2_x$ 
norm of $u_h$ from the 
$L^2_tL^2_x$ norm of $q_h$. In the present 
hyperbolic setting, however, the 
available energy estimates control 
interface dissipation (jump terms) but do not 
yield a bound for 
$\int_0^t\|q_h(s)\|_{L^2(\R)}^2\,ds$ uniformly in $h$. 
Consequently, the estimate 
above cannot be combined with the mean $L^2$ 
energy argument of Theorem \ref{thm:L2StabilityEst} 
to obtain stronger bounds involving the supremum 
in time inside the expectation, such as
$\E\big[\sup_{0\le s\le t}
\|u_h(s)\|_{L^2(\R)}^2\bigr]$. 
For this reason, we rely only on 
the martingale property \eqref{eq:zero-exp} in 
the mean stability analysis below. 
See, however, Theorem \ref{thm:pathwiseEst} 
for pathwise stability estimates 
available for certain schemes, 
where the product $u_h q_h$ can indeed be 
controlled.
\end{remark}

As a direct consequence of \eqref{eq:basis}, 
\eqref{eq:temporarySec3}, Lemma \ref{lem:fluxes}, 
and Proposition~\ref{lem:VanishingExp}, we obtain the 
following cell energy estimates.

\begin{theorem}[Local in-cell energy estimates]
\label{thm:localStabilityEst}
Consider $S(u)=u^2$ and let the 
numerical fluxes $\F_u$ and 
$(\widetilde{\F}_u,\widetilde{\F}_q)$ 
be given by \eqref{eq:firstFlux1} and 
\eqref{eq:firstFlux2}, respectively. 
Then the LDG approximation $(u_h,q_h)$, 
computed from the equations 
\eqref{eq:weak1}--\eqref{eq:weak3}, 
satisfies the local in-cell energy identity
\begin{equation}\label{eq:cellineq}
	\begin{split}
		d S_j(t) &+ \int_{I_j} S'(u_h(t))\,q_h(t)\,dx\,dW_t
		-\bigl(Q_{j+\frac12}(t)-Q_{j-\frac12}(t)\bigr)\,dt
		\\ & \quad 
		+ \frac{1}{8}\bigl(\gamma 
		+ 2\widetilde{\gamma}\bigr)\pigl(\bigl|(\sigma^2)^{'} 
		\bigr|\llbracket S'(u_h(t)) \rrbracket^2 \pigr)_{j-\frac{1}{2}}\, dt
		+ \pigl( \bigl|\sigma \bigr| \bigl(\eta_u \llbracket q_h(t) \rrbracket^2 
		+ \frac{\eta_q}{2}\llbracket S'(u_h(t))\rrbracket^2 \bigr) \pigr)_{j-\frac12}\, dt
		\\ & \quad\qquad 
		= \int_{I_j}\Bigl((\sigma')^2
		-\tfrac14(\sigma^2)''\Bigr)
		S(u_h(t))\,dx\,dt,
	\end{split}
\end{equation}
where $S_j(t)$ denotes the cell 
energy and $Q_{j+\frac12}(t)$ the associated 
numerical energy flux, defined by
$$
S_j(t):=\int_{I_j}S(u_h(t))\,dx,
\qquad
Q_{j+\frac12}(t):=2\Phi_{j+\frac12}(t)
+\Lambda_{j+\frac12}(t).
$$
See \eqref{eq:phi} and \eqref{eq:lambda} 
for the definitions of 
$\Phi_{j+\frac12}$ and 
$\Lambda_{j+\frac12}$, respectively. 
Moreover, 
taking expectations in \eqref{eq:cellineq} yields 
the mean cell energy inequality
\begin{align*}
	d\,\E \big [S_j(t) \big]
	\le \int_{I_j}\Bigl((\sigma')^2-\tfrac14(\sigma^2)''\Bigr)
	\E[S(u_h(t))]\,dx
	+ \E\Big[Q_{j+\frac12}(t)-Q_{j-\frac12}(t)\Big],
\end{align*}
with equality if we take
$\gamma=\widetilde{\gamma}=0$, see \eqref{eq:firstFlux1},
and $\eta_u=\eta_q=0$, see \eqref{eq:firstFlux2}.
\end{theorem}

Summing the local inequality \eqref{eq:cellineq} 
over all cells $I_j$, the 
energy fluxes telescope under 
periodic boundary conditions or for 
compactly supported initial data. As a consequence, we obtain the following mean energy 
estimate.

\begin{theorem}[Mean-square $L^2$ stability]
\label{thm:L2StabilityEst}
Assume $\sigma\in C^2(\R)$ and 
let $\bar u\in L^2(\R)$ be deterministic with 
compact support, such that the 
computational domain contains the support of 
the exact solution $u(t)$ of \eqref{eq:problem1} 
for all $t\in[0,T]$. Then the LDG approximation 
$\{(u_h,q_h)\}_{h>0}$, obtained 
from \eqref{eq:weak1}--\eqref{eq:weak3} with 
the numerical fluxes \eqref{eq:firstFlux1} 
and \eqref{eq:firstFlux2} (although 
well-posedness of \eqref{eq:sDESystem} 
is shown only for $\eta_u=0$), satisfies,
for every time $t\in(0,T]$,
\begin{equation}\label{eq:stabilityEst}
	\begin{split}
		\E\Big[\|u_h(t)\|_2^2\Big]
		&+\frac14\sum_{j\in\mathbb{Z}}
		(\gamma+2\widetilde{\gamma})
		\Bigl( \pigl|(\sigma^2)'\pigl|
		\int_0^t\E\Big[\llbracket u_h(s)\rrbracket^2\Bigr]
		\,ds\Bigr)_{j-\frac12}
		\\ & 
		+ \sum_{j\in\mathbb{Z}}
		\Bigl(\pigl|\sigma\pigr|
		\int_0^t\eta_u\E\Big[\llbracket q_h(s)
		\rrbracket^2\Bigr]
		+\eta_q\E\Big[\llbracket u_h(s)\rrbracket^2\Big]
		\,ds\Bigr)_{j-\frac12}
		\\ & \quad 
		= \|u_h(0)\|_2^2
		+ \int_0^t\int_{\R}
		\Bigl((\sigma')^2-\tfrac14(\sigma^2)''\Bigr)
		\E\big[u_h^2(s)\big]\,dx\,ds.		
	\end{split}	
\end{equation}
Here, $(\cdot)_{j\pm\frac12}$ denotes 
evaluation at the interface $x_{j\pm\frac12}$, while
the coefficients $|\sigma|$ and $|(\sigma^2)'|$ 
are understood pointwise.  
Consequently, if $\sigma$ satisfies 
the regularity assumptions in 
\eqref{eq:regularitySigma}, then
$$
\E\Big[\|u_h(t)\|_2^2\Bigr]
\le \|\bar u\|_2^2
\exp\Big(\bigl(\|\sigma'\|_\infty^2
+\tfrac14\|(\sigma^2)''\|_\infty\big)t
\Bigr).
$$
\end{theorem}

\begin{proof}
We sum the local mean identities of
Theorem \ref{thm:localStabilityEst} over $j$ and take
expectations. Lemma \ref{lem:VanishingExp} removes the
terms with zero expectation, while the numerical energy
fluxes cancel by telescoping. 
The remaining interface terms are precisely the
nonpositive contributions appearing in
\eqref{eq:stabilityEst}.

The exponential bound follows from the $L^\infty$ bounds
on $\sigma'$ and $(\sigma^2)''$, Gronwall's inequality,
and the stability of the initial projection:
$\|u_h(0)\|_2=\|\Pi_k\bar u\|_2\le\|\bar u\|_2$.
\end{proof}

\begin{remark}
The estimate \eqref{eq:stabilityEst} 
shows that the mean  jump contributions of the LDG 
approximation $u_h$, and also 
of $q_h$ when $\eta_u\neq 0$, namely
$\int_0^t \E\Big[\llbracket u_h(s)
\rrbracket_{j-\frac12}^2\Big]\,ds$
and $\int_0^t \E\Big[\llbracket q_h(s)\rrbracket_{j-\frac12}^2
\Big]\,ds$,
are controlled on $[0,T]$ by the $L^2(\R)$ 
norm of the initial data. This holds for 
all polynomial degrees $k\ge 0$ and 
reflects the built-in dissipation mechanism 
of the LDG method associated with 
the flux choices \eqref{eq:firstFlux1} 
and \eqref{eq:firstFlux2}.

By contrast, setting $\gamma=\widetilde{\gamma}=0$
and $\eta_u=\eta_q=0$, corresponding to a central flux
$\F_u$ and generalized alternating fluxes
$(\widetilde{\F}_u,\widetilde{\F}_q)$, introduces no
numerical dissipation. In this case, control of the mean
jump terms is lost, which may lead to oscillations near
sharp solution gradients; see
Section \ref{sec:NumericalExamples}.
\end{remark}

\begin{remark}
The source term in \eqref{eq:stabilityEst} 
vanishes for noise amplitudes 
$\sigma$ satisfying $(\sigma')^2=\tfrac14(\sigma^2)''$, 
which includes the cases 
of constant $\sigma$ and exponential 
profiles $\sigma(x)=Ae^{\pm x}$. For such 
coefficients, choosing $\gamma=\widetilde{\gamma}=0$ 
and $\eta_u=\eta_q=0$ yields exact 
preservation of the mean energy.
\end{remark}

A further motivation for using central fluxes for 
$(\widetilde{\F}_u,\widetilde{\F}_q)$---despite 
their tendency to induce oscillations for 
discontinuous solutions---is that 
they yield \textit{pathwise} 
stability when $\sigma$ is constant. This result 
extends to the multidimensional setting for 
divergence-free noise amplitudes, a standard 
assumption in many stochastic fluid models (see, e.g., 
\cite[Ch.~2]{FlandoliLuongo2023}).
 
\begin{theorem}[Pathwise $L^2$-estimate]
\label{thm:pathwiseEst}
Assume $\sigma(x) = \bar{\sigma} \in \R$ and 
let $\bar u\in L^2(\R)$ be deterministic with 
compact support, such that the 
computational domain contains the support of 
the exact solution $u(t)$ of \eqref{eq:problem1} 
for all $t\in[0,T]$. Consider the LDG approximation 
$\{(u_h,q_h)\}_{h>0}$, obtained 
from \eqref{eq:weak1}--\eqref{eq:weak3} with 
the numerical fluxes $\F_u$ defined in 
\eqref{eq:firstFlux1}, and 
$\widetilde{\F}_u$, $\widetilde{\F}_q$
defined in \eqref{eq:firstFlux2} 
with $\eta_u = 0$:
\begin{align*}
	\widetilde{\F}_u=\avg{u_h}  
	\qquad 
	\widetilde{\F}_q=\avg{q_h}
	+\eta_q \, \mathrm{sgn}(\sigma)
	\llbracket u_h \rrbracket.
\end{align*}
Then, almost surely, 
\begin{align*}
	\|u_h(t)\|_2 
	\leq \|u_h(0) \|_2, 
	\qquad t > 0,
\end{align*}
with equality if $\eta_q = 0$. 
\end{theorem} 

\vspace{-0.25cm}

\begin{proof}
Combining \eqref{eq:basis} and 
\eqref{eq:temporarySec3}, summing over 
$j\in\mathbb{Z}$, and using that the $Q_{j\pm \frac{1}{2}}$ terms telescope, we obtain
\begin{align*}
	&\|u_h(t)\|_2^2
	= \|\bar u_h\|_2^2
	+ \int_0^t \int_{\R}
	\Bigl((\sigma')^2-\tfrac14(\sigma^2)''\Bigr)
	u_h^2(s)\,dx\,ds  
	\\ & \qquad
	+ \sum_{j\in\mathbb{Z}}\int_0^t
	\bigl(2\Psi_{j-\frac12}(s)+\Theta_{j-\frac12}(s)\bigr)\,ds
	-2\sum_{j\in\mathbb{Z}}\int_0^t\int_{I_j}
	u_h(s)q_h(s)\,dx\,dW_s.
\end{align*}
Since $\sigma$ is constant, the second term on 
the right-hand side vanishes and 
$\Psi_{j-\frac12}=0$ for all $j$ (see \eqref{eq:psi}), 
so there is no contribution 
from the flux $\F_u$. Moreover, for the choice 
$(\widetilde{\F}_u,\widetilde{\F}_q)$ 
in \eqref{eq:firstFlux2}, we have 
$\Theta_{j-\frac12}\le 0$ for all $j$, 
cf.~\eqref{eq:nonnegative2}. It therefore 
remains to analyze the stochastic term (see 
also Remark \ref{rem:BDG}).

Choosing $\psi=u_h(t)$ in \eqref{eq:weak2} 
and using that $\sigma=\bar\sigma$ is 
constant, we obtain after 
straightforward manipulations that
\begin{align*}
	\int_{I_j} u_hq_h\, dx 
	& = -\int_{I_j} \partial_x(u_h)\sigma u_h\, dx 
	+\pigl(\sigma \widetilde{\F}_u u_h^-\pigr)_{j+\frac{1}{2}} 
	-\pigl(\sigma \widetilde{\F}_u u_h^+\pigr)_{j-\frac{1}{2}}
	\\ &
	= \bar{\sigma} \pigl(\bigl(\widetilde{\F}_u
	-\frac{1}{2}u_h^-\bigr)u_h^-\pigr)_{j+\frac{1}{2}}
	-\bar{\sigma} \pigl( \bigl(\widetilde{\F}_u
	-\frac{1}{2}u_h^+\bigr) u_h^+\pigr)_{j-\frac{1}{2}}.
\end{align*}
By arguing as in the proof of 
Lemma \ref{lem:fluxes}, we thus arrive at
\begin{align*}
	\int_{\R} u_h(t)q_h(t)\, dx 
	& =\frac{1}{2}\bar{\sigma}\sum_{j \in \mathbb{Z}} 
	\pigl( \llbracket u_h^2 \rrbracket 
	- 2\llbracket u_h \rrbracket 
	\widetilde{\F}_u \pigr)_{j-\frac{1}{2}}. 
\end{align*}
Since $\llbracket u_h^2\rrbracket
=2\llbracket u_h\rrbracket\avg{u_h}$, the 
summand vanishes identically when 
$\widetilde{\F}_u=\avg{u_h}$, that is, 
for the central flux. Consequently, the 
stochastic integral term is zero pathwise.
\end{proof}

\begin{remark}
Note that the proof of Theorem \ref{thm:pathwiseEst} 
does not directly extend to 
unequal-order approximation spaces, 
since in that case $u_h(t)$ may fail to 
belong to $\mathcal{Q}^l$. Consequently, 
$u_h(t)$ cannot in general be used as a 
test function in \eqref{eq:weak2} 
and the argument breaks down.
\end{remark}

\subsection{Comments on the choice of $q$}
\label{sec:auxiliaryOtherPart}
Let us briefly discuss the choice of the 
auxiliary variable $q$. Our definition 
of $q$ and the associated splitting of the 
It\^{o}--Stratonovich correction term 
\eqref{eq:rewriteCorrection} are motivated by the identity in 
Lemma \ref{lem:bj}, which reveals a 
cancellation of the quadratic covariation of 
$u_h$, up to numerical flux contributions. 
Preserving this cancellation at the 
discrete level serves as a guiding principle in 
the construction of stable 
LDG schemes for stochastic equations with gradient noise 
and strongly influences the choice of 
the auxiliary variable $q$.

When $\sigma$ is constant---or, more generally, when 
$\sigma$ is divergence-free in multiple 
dimensions---the It\^{o}--Stratonovich 
correction simplifies to
$$
\partial_x\bigl(\sigma \partial_x(\sigma u)\bigr)
= \partial_x\bigl(a\partial_x u\bigr),
\qquad a=\sigma^2,
$$
it is then natural to choose $q=\sqrt{a}\,\partial_x u$ and
this choice coincides with the auxiliary variable 
commonly used in LDG discretizations 
of second-order elliptic operators 
(see, e.g., \cite{Cockburn:1998ai}), 
and yields a formulation that is compatible 
with the mixed hyperbolic-parabolic structure 
of the problem.

A natural alternative reformulation of 
the It\^{o}--Stratonovich correction term 
is to write it in conservative form,
$$
\partial_x\bigl(\sigma \partial_x(\sigma u)\bigr)
= \partial_{xx}^2(\sigma^2 u)-\partial_x(\sigma\sigma' u),
$$
which suggests defining $q=\partial_x(\sigma^2 u)$. 
However, with this choice the 
analysis of the quadratic covariation 
becomes significantly more involved. 
Indeed, following the argument in the proof 
of Lemma \ref{lem:expressionCoVar} 
leads to terms of the form
$$
-\sum_{l=0}^k \int_0^t \int_{I_j}
\phi_l\,\partial_x\bigl(\sigma u_h(s)\bigr)
\,dx \,d \bigl\langle W,u_l(\cdot)\bigr\rangle_s,
$$
or variants obtained after integration 
by parts that also contain additional 
numerical flux contributions. 
Such terms cannot be handled as in 
\eqref{eq:temp1} and subsequent steps, 
since neither $\sigma'\phi_l$ nor 
$\sigma\,\partial_x\phi_l$ generally belongs to 
the discrete space $\mathcal{V}^k$.

In the case where $\sigma$ is 
constant, this difficulty disappears and 
the conservative reformulation can 
be treated without further complications. In 
the general case, however, one would 
need to introduce suitable projection 
operators and establish their 
stability properties or additional auxiliary variables, which substantially 
complicates the analysis. For this reason, we 
do not pursue this approach.

\medskip

It may also seem natural to treat the term 
$\sigma'\partial_x(\sigma u)$ appearing in 
\eqref{eq:rewriteCorrection} as a source term, 
that is, to rewrite the equation as
$$
du=\tfrac12\bigl(\sigma' q
+\sigma \partial_x q\bigr)\,dt - q\,dW_t,
$$
and then proceed with a DG discretization in space. 
However, when turning to the stability analysis, 
a similar difficulty arises. 
In particular, when attempting 
to derive a stability estimate solely 
in terms of $u_h$, one cannot eliminate 
$q_h$, since $\sigma' u_h$ 
is not an admissible test function. Projecting this 
product into $\mathcal{V}^k$ by $\Pi_k$ 
and applying \eqref{eq:weak2} leads 
to terms of the form
$$
-\int_{I_j}
\partial_x\bigl(\sigma'\Pi_k(u_h(t))\bigr)
\,\sigma u_h(t)\,dx
+\Bigl(\sigma \widetilde{\F}_u
\,\Pi_k(\sigma' u_h)\Bigr)
\Big|_{j-\frac12}^{j+\frac12},
$$
which are not amenable to the cancellations 
exploited in the present framework. 
If $\sigma'$ is constant, one may still proceed, 
but the resulting formulation 
imposes more restrictive conditions on the flux pair 
$(\widetilde{\F}_u,\widetilde{\F}_q)$. 
For this reason, we do not pursue this 
alternative approach further.

\subsection{Noise on nonconservative form}
\label{sec:noiseTransport}
Now, suppose the gradient 
noise instead appears in the form
\begin{equation}\label{eq:problem1-noncons}
	du+\sigma(x)\partial_x u \circ dW_t = 0,
\end{equation}
which we refer to as the transport form, in 
contrast to the continuity 
equation \eqref{eq:prob1Strat}.

In this case the It\^{o} 
formulation reads
\begin{equation}\label{eq:reformulatedLinear}
	du-\frac{1}{2}\sigma
	\partial_x(\sigma \partial_x u)\,dt 
	+\sigma \partial_x u\,dW_t = 0.
\end{equation}
To simplify the computation of the 
quadratic covariation, like in 
Lemma \ref{lem:expressionCoVar}, 
and ensure that we preserve the 
cancellation mechanism of the quadratic 
covariation from Lemma \ref{lem:bj}, we set $q=\sigma \partial_x u$ 
and reformulate the It\^{o}--Stratonovich 
correction term as follows:
\begin{align*}
	\frac{1}{2}\sigma \partial_x\bigl(\sigma \partial_x u \bigr) 
	& =\frac{1}{2} \partial_x \bigl(\sigma^2 \partial_x u\bigr) 
	-\frac{1}{2}\sigma \sigma' \partial_x u \nonumber
	= \frac{1}{2} \partial_x\bigl(\sigma q\bigr) 
	- \frac{1}{4}\bigl(\sigma^2 \bigr)' \partial_x u. 
\end{align*}
Consequently, \eqref{eq:reformulatedLinear} 
can be recast as the following first-order system: 
\begin{subequations}
	\begin{align}
		du &= \Bigl(\frac{1}{2} \partial_x\bigl(\sigma q\bigr) 
		-\frac{1}{4}\bigl(\sigma^2\bigr)'\partial_x u\Bigr)
		\, dt - q\, dW_t,
		\label{sec3:nonConsrecast1} 
		\\
		q &= \sigma \partial_x u, 
		\label{sec3:nonConsrecast2}
	\end{align}
\end{subequations}
to which we apply an elementwise DG-discretization. 

This leads to the following semi-discrete 
formulation. Find $(u_h, q_h)$ 
such that $(u_h, q_h)(\omega, t) 
\in \mathcal{V}^k \times \mathcal{V}^k$ 
(see Section \ref{sec:prelim}) 
for $(\omega, t) \in \Omega \times [0, T]$ 
and the equations 
\begin{subequations}
	\begin{align}	
		\int_{I_j}  \varphi \, du_{h}(t)\, dx 
		&= c_j(u_h(t), \varphi) + d_j(q_h(t), \varphi)
		- \int_{I_j} \varphi q_h(t)\, dx \, dW_t, 
		\label{sec3:nonCons1}
		\\
		\int_{I_j} \psi q_{h}(t)\, dx 
		&= -\int_{I_j} \partial_x(\sigma \psi) u_h(t)\, dx 
		+ \pigl(\sigma \widetilde{\F}_u \psi^-\pigr)_{j+\frac{1}{2}}
		-\pigl(\sigma \widetilde{\F}_u \psi^+\pigr)_{j-\frac{1}{2}}, 
		\label{sec3:nonCons2} 
	\end{align}
\end{subequations}
hold for all $(\varphi, \psi) \in 
\mathcal{V}^k \times \mathcal{V}^k$. Here, we have 
introduced the bilinear forms
\begin{align}\label{sec3:nonconsBilinear}
	c_j(u_h, \varphi) 
	&:= \frac{1}{4} \int_{I_j} \partial_x\pigl(\bigl(\sigma^2\bigr)'
	\varphi\pigr) u_h\, dx 
	-\frac{1}{4}\Bigl(
	\pigl(\bigl(\sigma^2\bigr)'
	\F_u \varphi^-\pigr)_{j+\frac{1}{2}}
	-\pigl(\bigl(\sigma^2\bigr)'
	\F_u \varphi^+\pigr)_{j-\frac{1}{2}}\Bigr), 
	\nonumber 
	\\
	d_j(q_h, \varphi) & := - \frac{1}{2}\int_{I_j}
	\bigl(\partial_x \varphi\bigr)\sigma q_h\, dx 
	+ \frac{1}{2} \Big( \pigl(\sigma \widetilde{\mathcal{\F}}_q
	\varphi^- \pigr)_{j+\frac{1}{2}}
	-\pigl(\sigma \widetilde{\mathcal{\F}}_q 
	\varphi^+ \pigr)_{j-\frac{1}{2}} \Big).
\end{align}
These are closely related to the bilinear forms 
in \eqref{eq:bilinearForms}, only the ``volume'' 
terms differ. Furthermore, the numerical initial data 
is given by the $L^2$-projection, 
$\bar{u}_h = \Pi_k \bar{u}$ and 
the numerical fluxes are now selected as 
\begin{align}\label{sec3:upwindPart}
	\F_u := \gamma \widehat{\F}_{u_h}
	\bigl(\bigl(\sigma^2\bigr)'\pigr) 
	+ (1-\gamma) \avg{u_h} 
	- \widetilde{\gamma}\, \mathrm{sgn}
	\pigl(\bigl(\sigma^2\bigr)'\pigr)
	\llbracket u_h \rrbracket, 
\end{align}
for adjustable parameters $\gamma \in [0, 1]$ 
and $\widetilde{\gamma} \geq 0$,  
and the pair $(\widetilde{\F}_u, \widetilde{\F}_q)$ 
can in general be selected as in \eqref{eq:firstFlux2}. 
The first term in \eqref{sec3:upwindPart} is now an 
upwind flux in contrast to the downwind flux 
in \eqref{eq:firstFlux1}, the reason for this 
becomes apparent in Lemma \ref{lem:cj}. 

\begin{remark}
It may be tempting to rewrite the correction term as
$\tfrac12\,\sigma \partial_x\bigl(\sigma \partial_x u\bigr)
= \tfrac12\,\sigma \partial_x q$, but this formulation 
destroys the volume cancellation in Lemma \ref{lem:bj}. In 
particular, it introduces the residual term
$-\int_{I_j}\sigma' q_h(s)u_h(s)\,dx$, which cannot be 
eliminated, since $\sigma' u_h$ is not an admissible test 
function. Projecting this product leads 
to difficulties analogous to those 
described in Section \ref{sec:auxiliaryOtherPart}.
\end{remark}
	
Since the stochastic component of the recast system 
\eqref{sec3:nonConsrecast1}--\eqref{sec3:nonConsrecast2} 
retains its original form, we can proceed 
along the lines of reasoning employed in the proof 
of Lemma \ref{lem:expressionCoVar} 
to derive the following covariation identity:

\begin{lemma}\label{lem:CoVarLinear} 
It holds that 
\begin{align*}
	\int_{I_j} \bigl \langle 
	u_h(\cdot), u_h(\cdot)\rangle_t\, dx 
	& = 
	- \int_0^t \int_{I_j} \partial_x(\sigma q_h(s)) u_h(s)
	\, dx \, ds 
	+ \int_0^t \Bigl( \pigl( \sigma 
	\widetilde{\F}_{u} q_h^- \pigr)_{j+\frac{1}{2}}
	-\pigl( \sigma \widetilde{\F}_u q_h^+\pigr)_{j-\frac{1}{2}}
	\Bigr)\, ds .
\end{align*}
\end{lemma}

To deduce mean-square $L^{2}$-stability, 
one starts from \eqref{eq:basis}. 
Inserting $\varphi=u_h(s)$ into 
\eqref{sec3:nonCons1} and 
integrating over $[0, t]$ yields
\begin{align*}
	\int_0^t \int_{I_j}u_h(s)\, du_h(s)\,ds 
	&= \int_0^t \pigl(c_j \bigl(u_h(s), u_h(s)\bigr) 
	+ d_j\bigl(q_h(s), u_h(s)\bigr) \pigr)\, ds 
	-\int_0^t\int_{I_j}u_h(s)q_h(s)\, dx\,dW_s.
\end{align*}
Following the arguments presented 
in Section \ref{sec3:StabEstimates}, 
particularly the proof of Lemma \ref{lem:bj}, we 
conclude that the quadratic 
covariation cancels out to leave only flux terms. 

\begin{lemma}\label{lem:cancellation}
The following identity holds true:
\begin{align*}
	 2 \int_0^t d_j\bigl(q_h(s), u_h(s)\bigr)\, ds 
	+\int_{I_j} \bigl \langle u_h(\cdot), u_h(\cdot)
	\bigr \rangle_t\, dx 
	= \int_0^t \Bigl(\Lambda_{j+\frac{1}{2}}(s)
	-\Lambda_{j-\frac{1}{2}}(s)
	+\Theta_{j-\frac{1}{2}}\Bigr)\, ds,
\end{align*}
where the entropy flux $\Lambda_{j+\frac12}$ 
and the remainder term $\Theta_{j-\frac12}$ 
are defined in \eqref{eq:lambda} and \eqref{eq:theta}, 
respectively.
\end{lemma}

The key difference from the conservative 
noise case lies in the identity given below for 
$c_j(u_h,u_h)$, which contrasts with 
the corresponding identity for 
$a_j(u_h,u_h)$ in Lemma \ref{lem:aj}.

\begin{lemma}\label{lem:cj}
It holds that
\begin{align*}
	c_j\bigl(u_h(t),u_h(t)\bigr)
	=\frac{1}{8}\int_{I_j}\bigl(\sigma ^2 \bigr)''
	u_h^2(t)\, dx -\pigl(\Phi_{j+\frac{1}{2}}(t)
	-\Phi_{j-\frac{1}{2}}(t)
	+\Psi_{j+\frac{1}{2}}(t) \pigr), 
\end{align*}
where $c_j(u_h,u_h)$ is defined 
in \eqref{sec3:nonconsBilinear}, and the entropy 
flux $\Phi_{j\pm\frac12}$ and the 
remainder term $\Psi_{j\pm\frac12}$ are given by 
\eqref{eq:phi} and \eqref{eq:psi}, respectively.
\end{lemma}

\begin{proof}
The volume contribution in $c_j(u_h(t),u_h(t))$ 
is obtained by expanding the spatial
derivative using the product rule 
and then redistributing derivatives by
integration by parts. More precisely,
\begin{align*}
	 \frac{1}{4} \int_{I_j}
	\partial_x\pigl(\bigl(\sigma^2\bigr)'u_h(t) \pigr)
	u_h(t)\, dx 
	&= \frac{1}{4} \int_{I_j}
	\pigl(\bigl(\sigma^2\bigr)''u_h^2(t)
	+\bigl(\sigma^2\bigr)'
	u_h(t)\partial_x(u_h(t)) \pigr)\, dx 
	\\ &= \frac{1}{8} \int_{I_j}\bigl(\sigma^2\bigr)''
	u_h^2(t)\, dx 
	+\frac{1}{8} \Bigl(\pigl( \bigl(\sigma^2\bigr)'
	(u_h^-)^2 \pigr)_{j+\frac{1}{2}}
	-\pigl(\bigl(\sigma^2\bigr)'
	(u_h^+)^2\pigr)_{j-\frac{1}{2}}
	\Bigr).
\end{align*}
Here we used the identity 
$u_h\partial_x u_h=\tfrac12\partial_x(u_h^2)$ 
and integrated by parts over the cell $I_j$, which 
produces the interface contributions.
Recalling the definition of $c_j(\cdot,\cdot)$ 
from \eqref{sec3:nonconsBilinear} and
comparing the resulting boundary terms 
with \eqref{eq:phi}, the 
asserted identity follows.
\end{proof}

Observe that, in Lemma \ref{lem:cj}, 
the flux terms appear with a minus sign. 
This necessitates the use of an upwind 
contribution in the flux $\F_u$ from 
\eqref{sec3:upwindPart} when the 
noise is written in nonconservative form, in 
contrast to the downwind choice in 
\eqref{eq:firstFlux1} used for the continuity 
equation. The penalty term 
also appears with a negative sign. 
Since the cancellation mechanism 
in Lemma \ref{lem:cancellation} 
is entirely analogous to that in Lemma \ref{lem:bj}, 
the remaining flux pair can be chosen 
exactly as before. As a result, Lemma \ref{lem:fluxes} and 
Proposition~\ref{lem:VanishingExp} 
remain valid, leading to the following 
stability result (compare with 
Theorem \ref{thm:localStabilityEst}).

\begin{theorem}[Stability for nonconservative noise]
\label{thm:stabilityNonCons} 
Assume that $\sigma\in C^2(\R)$. 
Set $S(u) = u^2$ and let the numerical 
fluxes $\F_u$ and $(\widetilde{\F}_u, \widetilde{\F}_q)$ 
be given by \eqref{sec3:upwindPart} 
and \eqref{eq:firstFlux2}, 
respectively. Then the LDG approximation, computed 
from  \eqref{sec3:nonCons1}--\eqref{sec3:nonCons2}, 
satisfies the following local in-cell 
energy inequality:
\begin{align}\label{eq:localInCellNon}
	dS_j(t) + \int_{I_j} S'(u_h(t))q_h(t)\, dx dW_t 
	-\pigl(\widetilde{Q}_{j+\frac{1}{2}}(t)
	-\widetilde{Q}_{j-\frac{1}{2}}(t)\pigr)\, dt
	\leq \frac{1}{8} \int_{I_j} 
	\bigl(\sigma^2\bigr)''S(u_h(t))\, dx \, dt,
\end{align}
where $S_j(t):=\int_{I_j}S(u_h(t))\,dx$ is the local 
energy and $\widetilde{Q}_{j+\frac{1}{2}}(t)
:= \Lambda_{j+\frac{1}{2}}(t)
-2\Phi_{j+\frac{1}{2}}(t)$ is the corresponding 
energy flux. Here, $\Phi_{j+\frac{1}{2}}$ 
and $\Lambda_{j+\frac{1}{2}}$ are defined in 
\eqref{eq:phi} and \eqref{eq:lambda}, respectively. 
If $\gamma=\widetilde{\gamma}=0$ and $\eta_u=\eta_q=0$, 
then equality holds in \eqref{eq:localInCellNon}.

Furthermore, assume that the deterministic 
initial data $\bar u\in L^2(\R)$ is 
compactly supported, with the 
computational domain chosen so as to contain 
the support of the exact solution 
$u(t)$ of \eqref{eq:problem1-noncons} 
for all $t\in[0,T]$. Then the LDG 
approximation satisfies following 
global energy identity:  
\begin{multline*}
	\E\Bigl[\|u_h(t)\|_2^2\Bigr]
	+\frac14\sum_{j\in\mathbb{Z}}
	(\gamma+2\widetilde{\gamma})
	\Bigl(\pigl|(\sigma^2)'\pigr|
	\int_0^t\E\Bigl[
	\llbracket u_h(s)\rrbracket^2\Bigr]
	\,ds\Bigr)
	\\
	+ \sum_{j\in\mathbb{Z}}
	\Bigl(\pigl|\sigma\pigr|
	\int_0^t
	\eta_u\E\Bigl[\llbracket q_h(s)\rrbracket^2\Bigr]
	+ \eta_q\E\Bigl[\llbracket u_h(s)\rrbracket^2\Bigr]
	\,ds\Bigr) 
	= \|u_h(0)\|_2^2 
	+\frac18\int_0^t\int_{\R}(\sigma^2)''\,
	\E\Bigl[u_h^2(s)\Bigr]\,dx\,ds.
\end{multline*}
In particular, if $(\sigma^2)''\in L^\infty(\R)$, 
it follows that $\E\bigl[\|u_h(t)\|_2^2\bigr]
\le \|\bar u\|_2^2\exp\Bigl(\tfrac18\|(\sigma^2)''\|_\infty\,t\Bigr)$.
\end{theorem}

\section{Multidimensional nonlinear equations}
\label{sec:multiTransport}

We now extend the proposed LDG framework 
to stochastic conservation laws of the type
\begin{align}\label{sec4:stratonovichSPDE}
	du + \sum_{\ell \in L}
	\div_x\bigl(\sigma_{\ell}(x) g_{\ell}(u)\bigr)
	\circ dW^{\ell}_t = 0,
	\qquad (t, x) \in [0, T] \times D.
\end{align}
The spatial domain $D \subset \R^d$ 
is bounded and chosen sufficiently large to 
contain the support of the solution $u(t)$ 
for all $t \in [0,T]$, given compactly supported initial data 
$u(0)=\bar u \in L^2(\R^d)$. The problem is 
spatially multidimensional ($d \ge 1$), and the noise 
is multidimensional in the sense that the finite index set $L$ 
labels several independent noise modes. 
The driving processes $\{W^{\ell}\}_{\ell\in L}$
are mutually independent one-dimensional Wiener processes.
We assume that each noise mode is associated with a
divergence-free spatial vector field $\sigma_\ell$ 
and a nonlinear flux function 
$g_\ell$ of controlled growth.

\begin{assumption}\label{sec4:regularityg}
For each $\ell\in L$, the vector field
$\sigma_{\ell}:\R^d\to\R^d$ belongs to 
$\in [L^2(\R^d)]^d$, with normal traces 
$\sigma_{\ell} \cdot n$ being well-defined 
in $L^2(\R^d)$ and $\div_x\sigma_{\ell}=0$.
Moreover, the scalar flux function
$g_{\ell}:\R\to\R$ belongs to $C^2(\R)$
and has at most polynomial growth, i.e.,
there exist constants $C_{\ell}>0$ and $r_{\ell}\in\mathbb{N}$
such that $|g_{\ell}(u)| \le C_{\ell}(1+|u|^{r_{\ell}})$ 
for $u\in\R$.
\end{assumption}

Under these assumptions, we prove 
that the corresponding LDG schemes are 
mean-square $L^2$ stable and give rise to 
well-posed systems of SDEs. 
Global well-posedness is established 
by a Khasminskii-type argument---without 
imposing a linear growth condition 
on $g_{\ell}$---which exploits the $L^2$ 
estimate to extend local solutions
to the full time interval $[0,T]$.
We also derive sufficient 
conditions for pathwise stability.

\medskip 

By \eqref{eq:StratIntegral},
the Stratonovich 
SPDE \eqref{sec4:stratonovichSPDE}
can be rewritten in It\^o form as 
the second-order equation
\begin{align}\label{sec4:Ito}
	du+\sum_{\ell \in L} 
	\div_x(\sigma_{\ell}(x) g_l(u)) dW^{\ell}_t 
	=\frac{1}{2}\sum_{\ell \in L}
	\div_x \bigl(\sigma_{\ell} g_{\ell}'(u)
	\div_x(\sigma_{\ell}g_{\ell}(u))\bigr)\, dt.
\end{align}
We aim to approximate possibly non-smooth 
solutions of \eqref{sec4:Ito}. 
To this end, for each $\ell\in L$ 
we introduce the auxiliary
scalar variable 
$q_{\ell}=\div_x(\sigma_{\ell} g_{\ell}(u))$.
Since $\sigma_{\ell}$ is divergence-free, 
we are guided by \eqref{eq:rewriteCorrection} 
and rewrite the It\^o--Stratonovich correction 
in the following form:
\begin{align*}
	\frac{1}{2}\sum_{\ell \in L}
	\mathrm{div}_x\bigl(\sigma_{\ell}g_{\ell}'(u)
	\mathrm{\div}_x(\sigma_{\ell}g_{\ell}(u))\bigr) 
	= \frac{1}{2}\sum_{\ell \in L} \sigma_{\ell}
	\cdot \nabla \bigl(g_{\ell}'(u)q_{\ell}\bigr) . 
\end{align*}
We may therefore recast \eqref{sec4:Ito} 
as the following stochastic 
system of first-order equations:
\begin{subequations}
	\begin{align}
		du(t) &= \frac{1}{2} \sum_{\ell \in L}
		\sigma_{\ell} \cdot \nabla
		\bigl(g_{\ell}'(u)q_{\ell}\bigr)\, dt 
		-\sum_{\ell \in L} q_{\ell} \, dW^{\ell}_t,
		\label{sec4:rewrite1} \\
		q_{\ell} &= \mathrm{div}_x 
		\bigl(\sigma_{\ell}g_{\ell}(u) \bigr), 
		\qquad \ell \in L. \label{sec4:rewrite2}
	\end{align}
\end{subequations}

Let $\T_h = \{K\}$ denote a shape-regular 
$d$-dimensional polyhedral mesh of the domain 
$D$ (see Section \ref{sec:prelim}). 
We proceed by discretizing this system elementwise 
by a DG-method, just like in Section \ref{sec:LDGFormulation}. 
This yields the following mixed DG-formulation: for 
each $(\omega, t ) \in \Omega \times [0, T]$, find
\begin{equation*}
	u_h(\omega, t) \in \mathcal{V}^k 
	\hspace{0.35cm} \text{and} \hspace{0.35cm}  
	\{q_{h, \ell}(\omega, t)\}_{\ell \in L} 
	\subset \mathcal{V}^k,
\end{equation*}
such that the equations 
\begin{subequations}
	\begin{align}
		d\int_{K}\varphi u_h(t)\, dx 
		&= \sum_{\ell \in L} 
		b_K\bigl(u_h(t), q_{h, \ell}(t), \varphi\bigr)\, dt 
		- \sum_{\ell \in L} \int_{K}
		\varphi \, q_{h, \ell}(t)\, dx \, dW_t^{\ell},
		\label{sec4:ldg1} \\ 
		\int_{K} \psi_{\ell} q_{h, \ell}(t)\, dx 
		& = - \int_{K} \bigl(\nabla \psi_{\ell}
		\cdot \sigma_{\ell} \bigr) g_{\ell}(u_h(t))\, dx 
		+ \sum_{e \subset \partial K} 
		\int_{e} \psi_{\ell} \F_{g_{\ell}(u)}^{e}
		\, d\mathcal{H}^{d-1}(x),
		\hspace{0.65cm} \ell \in L, 
		\label{sec4:ldg2}
	\end{align}
\end{subequations}
hold for every element $K\in\T_h$ and for all test functions
$\varphi,\psi_\ell\in\mathcal{V}^k$, where we have 
introduced the form (nonlinear in its 
first argument and linear in the remaining two)
\begin{align}\label{sec4:semilinear}
	b_K\bigl(u_h, q_{h, \ell}, \varphi \bigr)
	& :=-\frac{1}{2}\int_{K} 
	\bigl(\nabla \varphi \cdot \sigma_{\ell}\bigr)
	g_{\ell}'(u_h)q_{h, \ell} \, dx
	+ \frac{1}{2}\sum_{e \subset \partial K}\int_{e}
	\varphi  \F_{g_{\ell}'(u)q_{\ell}}^{e}
	\, d\mathcal{H}^{d-1}(x).
\end{align}
The numerical initial data is given by 
\begin{equation}\label{sec4:ldg3}
	\int_{K} \vartheta u_h(0)\, dx
	=\int_{K} \vartheta \bar{u}\, dx,
	\qquad \text{for all 
	$\vartheta \in \mathcal{V}^k$ 
	and $K \in \T_h$}. 
\end{equation}

It remains to specify the numerical fluxes
$\bigl\{(\F_{g'_{\ell}(u)q_{\ell}}^{e},
\F_{g_{\ell}(u)}^{e})\bigr\}_{\ell\in L}$.
These fluxes are required to be local, 
in the sense that, at an interface 
$e=\partial K_+\cap\partial K_-$ with $K_\pm\in\T_h$, 
they depend only on the traces of $u_h$ 
and $q_h$ on the adjacent elements, that is,
on $u_h|_{K_\pm}$ and $q_h|_{K_\pm}$. 
They must also be consistent, namely,
\begin{equation*}
	\F_{g_{\ell}'(u)q_{\ell}}^{e}(u,u,q,q) 
	=(\sigma_{\ell} \cdot n_e)g_{\ell}'(u)q, 
	\qquad 
	\F_{g_{\ell}(u)}(u,u,q,q)
	=(\sigma_{\ell}\cdot n_e) g_{\ell}(u),
\end{equation*}
for all $u,q\in \R$,  where $n_e$ denotes the 
unit normal vector on $e$ oriented outward 
with respect to either $K_+$ or $K_-$.
In addition, the fluxes are 
required to satisfy the following 
conservation properties:
\begin{align}\label{eq:conservationProp}
	\F_{g_{\ell}'(u)q_{\ell}}^{e, K_+}
	= -\F_{g_{\ell}'(u)q_{\ell}}^{e, K_-}, 
	\qquad
	\F_{g_{\ell}(u)}^{e, K_+} 
	= -\F_{g_{\ell}(u)}^{e, K_-},
\end{align}
where the roles of the traces in 
$\F_{g_{\ell}'(u)q_{\ell}}^{e,K_+}$ 
and $\F_{g_{\ell}'(u)q_{\ell}}^{e,K_-}$ 
are interchanged and $n\lvert_{K_+}=-n\lvert_{K_-}$ 
(see Section \ref{sec2:FEMNotation}), 
and similarly for the numerical 
flux with the $g_{\ell}(u)$-subscript. 

With these requirements in mind, we now specify 
the numerical fluxes appearing
in \eqref{sec4:ldg1}--\eqref{sec4:ldg2}. 
These fluxes arise from the
discretization of the It\^o--Stratonovich 
correction term and are carefully chosen so as
to preserve the skew-symmetric (hyperbolic) structure 
of \eqref{sec4:Ito}, thereby avoiding the 
introduction of spurious numerical dissipation. 
In particular, for each
$\ell\in L$ (see again Section \ref{sec2:FEMNotation} 
for the notation),
\begin{equation}\label{sec4:numericalFluxes}
	\begin{split}
		\F_{g_{\ell}'(u)q_{\ell}}^{e}
		& =(\sigma_{\ell} \cdot n_e)
		\biggl(\frac{\llbracket g_{\ell}(u_h)\rrbracket}
		{\llbracket u_h\rrbracket}\avg{q_{h, \ell}}
		+\eta_{q,\ell}^{e}\,\mathrm{sgn}(\sigma_{\ell} \cdot n_e)
		\llbracket u_h \rrbracket \biggr), 
		\\ 
		\F_{g_{\ell}(u)}^{e} & = 
		(\sigma_{\ell}\cdot n_e)
		\pigl( \avg{g_{\ell}(u_h)} + \eta_{u, \ell}^{e}
		\, \mathrm{sgn}(\sigma_{\ell} \cdot n_e) 
		\llbracket q_{h, \ell} \rrbracket\pigr), 		
	\end{split}	
\end{equation}
with the interpretation that 
$\tfrac{\llbracket g_{\ell}(u)\rrbracket}
{\llbracket u \rrbracket} = g_{\ell}'(u)$ 
whenever $\llbracket u \rrbracket = 0$. 
Here $\eta_{u,\ell}^e,\eta_{q,\ell}^e\ge0$ 
are nonnegative constants, which may 
depend on the interface $e$.
It is computationally advantageous 
that the flux $\F_{g_{\ell}(u)}^{e}$ is
independent of $q_{h,\ell}$ (as in the 
previous section), which allows the
auxiliary variables $\{q_{h,\ell}\}_{\ell\in L}$ 
to be eliminated locally from
the final scheme. The dependence on the 
neighboring elements $K_+$ and $K_-$
sharing the face $e$ is suppressed 
in view of the conservation property 
\eqref{eq:conservationProp}.

If the interface $e$ lies on the physical 
boundary, i.e.~$e=\partial K\cap\partial D$,
we impose zeroth-order extrapolation and set
\begin{align}\label{sec4:numericalFluxesBD}
	\F_{g_{\ell}'(u)q_{\ell}}^{e} 
	=(\sigma_{\ell} \cdot n_e)
	g_{\ell}'(u_h\lvert_K)q_{h, \ell}\lvert_K, 
	\qquad
	\F_{g_{\ell}(u)}^{e} 
	= (\sigma_{\ell} \cdot n_e)
	g_{\ell}(u_h\lvert_K).
\end{align}

\subsection{Well-posedness and stability}
\label{sec4:stability}
We now show that the semi-discretizations
\eqref{sec4:ldg1}--\eqref{sec4:ldg2}
give rise to well-posed finite-dimensional systems of SDEs.
Unlike in Section \ref{sec:wellPosedSDE}, the coefficients
$g_{\ell}$, for $\ell\in L$, may be nonlinear and need not satisfy
a global linear growth condition. Instead, we use the polynomial
growth assumption in Assumption \ref{sec4:regularityg}.
Thus Theorem \ref{thm:SDEStandard} cannot be applied directly
to obtain global well-posedness. We first prove local
well-posedness and then combine this with the energy estimates
for the stopped LDG approximation
(Theorem \ref{sec4:StabilityTheorem}) to rule out finite-time
explosion, following a Khasminskii-type 
argument \cite{Khasminskii:2012aa}.

As in Section \ref{sec:wellPosedSDE}, the LDG
semi-discretization transforms
\eqref{sec4:rewrite1}--\eqref{sec4:rewrite2}
into a finite-dimensional system of SDEs. Since the derivation
is identical, we only state the resulting system. Let
$$
\bold{u}:\Omega\times[0,T]\to
\R^{(N_k+1)\times|\T_h|}
$$
denote the coefficient matrix associated with the primary
unknown $u_h$. For each $K\in\T_h$, we write $\bold{u}^K$
for the corresponding column.

We do not yet substitute the specific fluxes from
\eqref{sec4:numericalFluxes}--\eqref{sec4:numericalFluxesBD}.
Instead, for each $\ell\in L$, we write
$$
\F_{g'_{\ell}(u)q_{\ell}}
=
\F_{\ell}(\cdot,\cdot,\cdot,\cdot),
\qquad
\F_{g_{\ell}(u)}
=
\widetilde{\F}_{\ell}(\cdot,\cdot).
$$
We assume that $\widetilde{\F}_{\ell}$ does not depend
explicitly on $q_{h,m}$, $m\in L$, so that the auxiliary
variables $\{q_{h,\ell}\}_{\ell\in L}$ can be eliminated
locally. The factor $\sigma_{\ell}\cdot n$ is kept outside
the definition of these numerical fluxes. We also assume that
the fluxes are locally Lipschitz and locally of linear growth. The latter means that 
for all $R>0$ and all
$c_1,c_2,c_3,c_4\in\R$ with
$\max_m |c_m|\le R$, one has 
\begin{align*}
	|\F_{\ell}(c_1,c_2,c_3,c_4)|
	&\leq C_R
	\pigl(1+\sum_{m=1}^4 |c_m|\pigr),
	\qquad
	|\widetilde{\F}_{\ell}(c_1,c_2)|
	\leq \widetilde C_R
	\pigl(1+|c_1|+|c_2|\pigl),
\end{align*}
where $C_R,\widetilde C_R>0$ may depend on $R$ and $\ell$.

Testing \eqref{sec4:ldg2} 
with the local basis functions gives
$$
\bold{q}_{\ell}^K(t)
=
\bold{Q}_{\ell}^K(\bold{u}(t))
=
\bigl(
Q_{\ell,0}^K(\bold{u}(t)),
\ldots,
Q_{\ell,N_k}^K(\bold{u}(t))
\bigr)^T.
$$
For $m=0,\ldots,N_k$,
\begin{align}\label{sec4:Q}
	Q_{\ell,m}^K(\bold{u})
	&:=
	-\int_K
	\pigl(
	\bigl(\mathbb{M}^K\bigr)_m
	\cdot
	\bigl(\nabla \boldsymbol{\phi}^K\cdot\sigma_{\ell}\bigr)
	\pigr)
	g_{\ell}
	\bigl(\bold{u}^K\cdot\boldsymbol{\phi}^K\bigr)
	\,dx
	\nonumber
	\\
	&\quad\qquad
	+\sum_{e\subset\partial K}
	\int_e
	(\sigma_{\ell}\cdot n_e)
	\pigl(
	\bigl(\mathbb{M}^K\bigr)_m
	\cdot \boldsymbol{\phi}^K
	\pigr)
	\widetilde{\F}_{\ell}
	\bigl(
	\bold{u}^{K_e}\cdot\boldsymbol{\phi}^{K_e},
	\bold{u}^K\cdot\boldsymbol{\phi}^K
	\bigr)
	\,d\mathcal{H}^{d-1}(x).
\end{align}
Here $K_e$ is the element satisfying
$e=\partial K\cap\partial K_e$. If $e\subset\partial D$,
we use the convention in \eqref{sec4:numericalFluxesBD}.
Moreover,
$$
\nabla 
\boldsymbol{\phi}^K\cdot\sigma_{\ell}
=
\begin{pmatrix}
	\nabla\phi_0^K\cdot\sigma_{\ell} \\
	\vdots \\
	\nabla\phi_{N_k}^K\cdot\sigma_{\ell}
\end{pmatrix},
$$
and $\mathbb{M}^K$ denotes the inverse of the local mass
matrix $M^K$, whose entries are given by
\eqref{eq:massEntries}. The number of columns on which
$\bold{Q}_{\ell}^K$ depends is determined by the number of
neighbors of $K$. For instance, if $\T_h$ consists of
simplices, then $K$ has at most $d+1$ faces, and
$\bold{Q}_{\ell}^K$ depends on at most $d+2$ columns.

After eliminating the auxiliary variables, we obtain 
the SDE system
\begin{align}\label{sec4:SDE-system}
	d\bold{u}(t)
	=
	F(\bold{u}(t))\,dt
	+
	\sum_{\ell\in L}
	G_{\ell}(\bold{u}(t))\,dW_t^{\ell}.
\end{align}
The entries of $G_{\ell}$ are
\begin{align}\label{sec4:GL}
	G_{\ell,m}^K(\bold{u})
	:=
	-\int_K
	\bigl(
	\bold{Q}_{\ell}^K(\bold{u})
	\cdot\boldsymbol{\phi}^K
	\bigr)
	\pigl(
	\bigl(\mathbb{M}^K\bigr)_m
	\cdot\boldsymbol{\phi}^K
	\pigr)
	\,dx.
\end{align}
Furthermore,
\begin{equation}\label{sec4:F}
	F(\bold{u})
	=
	\sum_{\ell\in L} F_{\ell}(\bold{u}),
\end{equation}
where the entries of $F_{\ell}$ are
\begin{align}\label{sec4:FL}
	F_{\ell,m}^K(\bold{u})
	&:=
	-\frac{1}{2}
	\int_K
	\pigl(
	\bigl(\mathbb{M}^K\bigr)_m
	\cdot
	\bigl(\nabla\boldsymbol{\phi}^K
	\cdot\sigma_{\ell}\bigr)
	\pigr)
	g_{\ell}'
	\bigl(\bold{u}^K\cdot\boldsymbol{\phi}^K\bigr)
	\nonumber
	\bigl(
	\bold{Q}_{\ell}^K(\bold{u})
	\cdot\boldsymbol{\phi}^K
	\bigr)
	\,dx
	\nonumber
	\\
	&\qquad
	+\frac{1}{2}
	\sum_{e\subset\partial K}
	\int_e
	(\sigma_{\ell}\cdot n_e)
	\pigl(
	\bigl(\mathbb{M}^K\bigr)_m
	\cdot\boldsymbol{\phi}^K
	\pigr)
	\nonumber
	\\
	&\qquad\qquad\qquad\quad
	\times
	\F_{\ell}
	\bigl(
	\bold{u}^{K_e}\cdot\boldsymbol{\phi}^{K_e},
	\bold{u}^K\cdot\boldsymbol{\phi}^K,
	\bold{Q}_{\ell}^{K_e}(\bold{u})
	\cdot\boldsymbol{\phi}^{K_e},
	\bold{Q}_{\ell}^K(\bold{u})
	\cdot\boldsymbol{\phi}^K
	\bigr)
	\,d\mathcal{H}^{d-1}(x).
\end{align}

Under Assumption \ref{sec4:regularityg} and the above
conditions on the fluxes $\F_{\ell}$ and $\widetilde{\F}_{\ell}$,
the system \eqref{sec4:SDE-system} admits a unique local
strong solution
$$
	\bold{u}:
	\Omega\times[0,\tau_{\max})
	\to
	\R^{(N_k+1)\times|\T_h|},
$$
where $\tau_{\max}\in(0,T]$ is the maximal lifetime.
Moreover, if
\begin{equation}\label{eq:SDE-stopping-time}
	\tau_R
	:=
	\inf\{t\in[0,T]:
	\|\bold{u}(t)\|_F\ge R\}\wedge T,	
\end{equation}
then
\begin{equation}\label{eq:mean-stopped}
	\E\left[
	\sup_{t\in[0,T]}
	\|\bold{u}(t\wedge\tau_R)\|_F^2
	\right]
	<\infty
	\qquad
	\text{for every $R>0$.}	
\end{equation}
A proof of the local well-posedness of
\eqref{sec4:SDE-system} is given in
Appendix~\ref{sec:appWell}. For simplicity, we assume there
that the parameters $\eta_{u,\ell}^e$ appearing in
\eqref{sec4:numericalFluxes} vanish.

\medskip

We next exclude finite-time blow-up of $\bold{u}$ 
almost surely, showing that the solution extends 
globally to $[0,T]$. The central ingredient 
is a uniform mean-square $L^2$ estimate. 
Crucially, this bound reflects that the LDG schemes inherit,
up to flux terms, the cancellation mechanism of the continuous
problem (see \eqref{eq:entropy-eq}). More precisely, the
quadratic covariation of the stochastic term is balanced by
the dissipation generated by the It\^o--Stratonovich
correction. In this sense, the schemes retain the hyperbolic
character of the stochastic conservation law
\eqref{sec4:stratonovichSPDE}.

Starting from \eqref{eq:basis}, 
testing \eqref{sec4:ldg1} with 
$\varphi=u_h(s)$, and integrating over $[0,t]$ yields
\begin{align}\label{sec4:temp1}
	\int_0^t  \int_{K}u_h(s)du_h(s)\, dx 
	& = \sum_{\ell \in L}\int_0^t b_{K}
	\bigl(u_h(s), q_{h,\ell}(s),u_h(s)\bigr)
	\, ds 
	-\sum_{\ell \in L} \int_0^t 
	\int_{K} u_h(s) q_{h,\ell}(s)
	\, dx \,dW_s^{\ell}. 
\end{align}
The following lemma contains 
the crucial cancellation property.

\begin{lemma}\label{sec4:cancellationCovar} 
For any $K \in \T_h$ and $t \in [0, \tau_{\mathrm{max}})$, 
where $\tau_{\mathrm{max}}$ is the existence 
time of the local solution, the following identity holds: 
\begin{align*}
	&2\sum_{\ell \in L} \int_0^t 
	b_{K}\bigl(u_h(s), q_{h, \ell}(s), u_h(s)\bigr)\, ds 
	+\int_K \bigl \langle u_h(\cdot), u_h(\cdot) 
	\bigr \rangle_t \, dx
	\\
	& \qquad
	= \sum_{\ell \in L} \sum_{e \subset \partial K}
	\int_0^t \int_{e}(\sigma_{\ell} \cdot n_e)
	\pigl(u_h \widetilde{\F}^{e}_{g_{\ell}'(u)q_{\ell}}
	+q_{h,\ell}\widetilde{\F}^e_{g_{\ell}(u)}
	-g_{\ell} (u_h)q_{h, \ell} \pigr)
	\, d\mathcal{H}^{d-1}(x)\,ds, 
\end{align*}
where 
$\widetilde{\F}^{e}_{g_{\ell}'(u)q_{\ell}}$ 
and $\widetilde{\F}^{e}_{g_{\ell}(u)}$ denote the fluxes 
without the factor $\sigma_{\ell}\cdot n_e$---see 
\eqref{sec4:numericalFluxes} 
and \eqref{sec4:numericalFluxesBD}:
\begin{equation}\label{sec4:fluxesWithout}
	\F^{e}_{g_{\ell}'(u)q_{\ell}} 
	= (\sigma_{\ell} \cdot n_e)
	\widetilde{\F}^{e}_{g_{\ell}'(u)q_{\ell}},
	\qquad
	\F^{e}_{g_{\ell}(u)} 
	= (\sigma_{\ell} \cdot n_e)
	\widetilde{\F}^{e}_{g_{\ell}(u)}.
\end{equation}
\end{lemma}

\begin{proof}
We begin by analyzing the covariation of $u_h$. 
Since the noise enters in the same form, through 
$q_{h,\ell}$, as in \eqref{eq:rewrite1}--\eqref{eq:rewrite2}, 
the argument from the proof of Lemma \ref{lem:expressionCoVar} 
applies. Integrating \eqref{sec4:ldg1} over $(0,t)$ gives
\begin{align*}
	\int_{K}\varphi u_h(t)\, dx 
	&= \int_{K} \varphi u_h(0)\, dx 
	+ \sum_{\ell \in L}
	\int_0^t b_K(u_h(s),q_{h, \ell}(s),\varphi)\, ds 
	-\sum_{\ell \in L}\int_0^t \int_{K} \varphi
	q_{h,\ell}(s)\,dx\, dW_s^{\ell}. 
\end{align*}
Adapting the quadratic variation computations 
underlying \eqref{eq:quad-var-uh} to the present 
$L$-dimensional noise setting, and using the 
independence of the Wiener processes 
$\{W^\ell\}_{\ell\in L}$, we arrive at
\begin{align*}
	\int_{K} \bigl \langle u_h(\cdot),
	u_h(\cdot) \bigr \rangle_t\, dx
	= \sum_{\ell \in L} \int_0^t
	\int_{K}q_{h, \ell}^2(s)\, dx \, ds. 
\end{align*}
In view of \eqref{sec4:ldg2}, we therefore obtain 
\begin{align}\label{sec4:covar}
	\int_{K} \bigl\langle u_h(\cdot),
	u_h(\cdot)\bigr \rangle_t \, dx 
	& =-\sum_{\ell \in L}  \int_0^t
	\int_K \bigl( \nabla q_{h, \ell} (s)
	\cdot\sigma_{\ell}\bigr)
	g_{\ell}(u_h(s))\, dx \, ds
	\nonumber 
	\\
	& \qquad 
	+\sum_{e \subset \partial K} \sum_{\ell \in L} 
	\int_0^t \int_{e} q_{h, \ell}(s)
	\F^{e}_{g_{\ell}(u)}(s)
	\, d\mathcal{H}^{d-1}(x) \, ds. 
\end{align}

Note that $g_{\ell}'(u_h)\nabla u_h
=\nabla g_{\ell}(u_h)$ on each $K\in\T_h$. 
Since $\sigma_{\ell}$ is divergence-free, it follows that
\begin{align*}
	 \int_{K} \pigl( \bigl(\nabla u_h 
	\cdot \sigma_{\ell}\bigr) g_{\ell}'(u_h) q_{h, \ell} 
	+ \bigl(\nabla q_{h, \ell} \cdot \sigma_{\ell}\bigr)
	g_{\ell}(u_h)\pigr)\, dx &
	= \int_{K} \mathrm{div}_x
	\pigl(\sigma_{\ell} g_{\ell}(u_h)q_{h, \ell}\pigr)\, dx  
	\\ 
	&= \int_{\partial K}
	(\sigma_{\ell} \cdot n)g_{\ell}(u_h)
	\, q_{h, \ell} \, d\mathcal{H}^{d-1}(x), 
\end{align*}
where the final step follows from the divergence theorem.
This identity is understood under the normal-trace 
assumption on $\sigma_{\ell}$ used in the definition of the
numerical fluxes. Indeed, since $u_h,q_{h,\ell}\in
\mathcal{P}^k(K)\subset L^\infty(K)$ and $g_{\ell}$ is bounded 
on compact sets, the field 
$\sigma_{\ell}g_{\ell}(u_h)q_{h,\ell}$ belongs to $[L^2(K)]^d$. 
Moreover, since $\div_x\sigma_{\ell}=0$ in the distributional
sense and $\nabla u_h,\nabla q_{h,\ell}$ are bounded on $K$, its
divergence belongs to $L^2(K)$. Hence its normal trace is
well-defined in $H^{-1/2}(\partial K)$, and, when
$\sigma_{\ell}\cdot n_e$ has the assumed $L^2$ trace on faces,
the boundary integral above is meaningful. 
Recalling the definition of $b_K(u_h,q_{h,\ell},u_h)$ in
\eqref{sec4:semilinear}, and combining it with
\eqref{sec4:covar}, we deduce that
\begin{align*}
	& 2\int_0^t b_K\bigl(u_h(s),
	q_{h, \ell}(s), u_h(s)\bigr)\, ds 
	+ \int_0^t \int_{K}q_{h, \ell}^2(s)\, dx \, ds
	\\
	& \qquad\quad 
	= \sum_{e \subset \partial K} 
	\int_0^t \int_{e} 
	(\sigma_{\ell} \cdot n_e)
	\Bigl( u_h \widetilde{\F}^{e}_{g_{\ell}'(u)q_{\ell}} 
	+q_{h,\ell}\widetilde{\F}^{e}_{g_{\ell}(u)} 
	- g_{\ell}(u_h)q_{h, \ell} \Bigr)
	\, d\mathcal{H}^{d-1}(x)\,ds,
\end{align*}
from which the result follows after summing 
over the noise modes $\ell \in L$.
\end{proof}

The previous lemma applies for any time 
$t < \tau_{\mathrm{max}}$. Let us now define an increasing
sequence $\{\tau_n\}$ of stopping times:
\begin{align}\label{sec4:stoppingTimes}
	\tau_n
	:=\inf \Bigl\{t \geq 0:
	\|\bold{u}(t)\|_F \geq n
	\Bigr\}, 
	\qquad \tau_n(t):=t \wedge \tau_n.
\end{align}
Then $\lim_{n \rightarrow \infty} \tau_n= \tau_{\max}$.
Moreover, since
$\bold{q}_{\ell}=\bold{Q}_{\ell}(\bold{u})$ and
$\bold{Q}_{\ell}$ is locally bounded, the auxiliary
variables are bounded on each stopped interval $[0,\tau_n]$.
If we combine \eqref{eq:basis}, \eqref{sec4:temp1}, and 
Lemma \ref{sec4:cancellationCovar}, and thereafter sum over
$K\in \T_h$, we get
\begin{align}\label{sec4:tempL2}
	& \|u_h(\tau_n(t))\|_2^2 = \|u_h(0)\|_2^2
	-2\sum_{K \in \T_h} \sum_{\ell \in L}
	\int_0^{\tau_n(t)}
	\!\!\int_{K}u_h(s)q_{h, \ell}(s)\, dx \, dW_s^{\ell} 
	\nonumber
	\\
	& \quad\quad
	+\sum_{K \in \T_h}\sum_{e \subset \partial K}
	\sum_{\ell \in L} \int_0^{\tau_n(t)}\!\! 
	\int_{e}(\sigma_{\ell}\cdot n_e)
	\pigl(u_h \widetilde{\F}_{g_{\ell}'(u)q_{\ell}}^{e} 
	+q_{h,\ell}\widetilde{\F}_{g_{\ell}(u)}^{e}
	-g_{\ell}(u_h)q_{h, \ell}\pigr)
	\, d\mathcal{H}^{d-1}(x)\, ds 
	\nonumber
	\\
	& \quad
	= \|u_h(0)\|_2^2
	-2\sum_{K \in \T_h}
	\sum_{\ell \in L}\int_0^{\tau_n(t)} \!
	\int_{K}u_h(s)q_{h, \ell}(s)\, dx \, dW_s^{\ell} 
	\nonumber
	\\
	& \quad\quad
	+ \sum_{e \in \mathcal{E}_h}  \int_0^{\tau_n(t)}
	\Biggl(\, \underbrace{\sum_{\ell \in L}
	\int_{e} \bigl(\sigma_{\ell}\cdot n_e\bigr) 
	\pigl(\llbracket g_{\ell}(u_h) q_{h, \ell} \rrbracket 
	- \llbracket u_h \rrbracket
	\widetilde{\F}_{g_{\ell}'(u)q_{\ell}}^{e}
	-\llbracket q_{h, \ell} \rrbracket
	\widetilde{\F}_{g_{\ell}(u)}^{e}\pigr)(s)
	\, d\mathcal{H}^{d-1}(x)}_{:=\Phi_e(s)}
	\, \Biggr)ds.	 
\end{align}
In the final step, we reorganized the sum over element
boundaries into a sum over interfaces
$e\in\mathcal{E}_h$ using the conservation property
\eqref{eq:conservationProp} of the numerical fluxes.
If $e=\partial K_+\cap\partial K_-$ is an interior interface,
the corresponding edge contribution appears twice in the
preceding sum, once with outward normal
$n_{K_+}=n_e$ and once with
$n_{K_-}=-n_e$. By consistency and conservation of the 
numerical fluxes, these two terms combine into 
the jump expression defining $\Phi_e(s)$.

If $e$ lies on the physical boundary,
$e=\partial K\cap\partial D$, the term appears only once.
With the boundary convention
\eqref{sec4:numericalFluxesBD}, the boundary contribution is
written in the same jump form by using the prescribed exterior
state. In the zeroth-order extrapolation case, the exterior
trace is taken equal to the interior value, and therefore
$$
\llbracket u_h \rrbracket = 0,
\qquad
\llbracket q_{h,\ell} \rrbracket = 0,
\qquad
\llbracket g_{\ell}(u_h) q_{h,\ell} \rrbracket = 0.
$$
Thus, with this boundary convention, the
boundary contribution vanishes identically. In other words,
the boundary treatment does not produce any additional 
energy contribution.

\medskip 
 
To obtain energy stability of the LDG schemes, it 
suffices to ensure that $\sum_{e\in\mathcal{E}_h}
\int_0^{\tau_n(t)}\Phi_e(s)\,ds\le 0$, while 
the stochastic integral contribution admits 
two possible treatments:
\begin{enumerate}[label=(\roman*)]
	\item it may vanish in expectation:
	\begin{equation}\label{sec4:vanishExpectation}
		\E \left[\, 2
		\sum_{K \in \T_h} \sum_{\ell \in L} 
		\int_0^{\tau_n(t)}\!\!\int_{K} u_h(s)
		q_{h, \ell}(s)\, dx \, dW_s^{\ell}
		\right] = 0,
	\end{equation}
	in which case we obtain mean-square $L^2$-stability, or 

	\item it may vanish pathwise: 
	\begin{equation}\label{sec4:vanishPathwise}
		2
		\sum_{K \in \T_h} \sum_{\ell \in L}
		\int_0^{\tau_n(t)}\!\!\int_{K} 
		u_h(s)q_{h, \ell}(s)\, dx \, dW_s^{\ell} = 0,
	\end{equation}
	which typically imposes more restrictive 
	conditions on the numerical fluxes, but may 
	yield pathwise $L^2$ estimates.
\end{enumerate}
We postpone the discussion of pathwise stability 
and focus on proving (i). Firstly, by substituting 
the numerical fluxes 
\eqref{sec4:numericalFluxes}--\eqref{sec4:numericalFluxesBD} 
into $\Phi_e(s)$ and using the identity 
$\llbracket ab\rrbracket=\avg{a}\llbracket b\rrbracket
+\llbracket a\rrbracket\avg{b}$, we obtain
\begin{align*}
	\Phi_e(s)
	=-\sum_{\ell \in L}\int_{e}
	\bigl|\sigma_{\ell}\cdot n_e\bigr|
	\pigl( \eta_{q, \ell}^{e} \llbracket u_h \rrbracket ^2
	+ \eta_{u, \ell}^{e}\llbracket q_{h, \ell} \rrbracket^2 
	\pigr)(s)\, d\mathcal{H}^{d-1}(x). 
\end{align*}
As a consequence, 
\begin{align}\label{sec4:nonpositiveBoundary}
	\sum_{e \in \mathcal{E}_h}
	\int_0^{\tau_n(t)} \Phi_e(s) ds \leq 0.
\end{align}

\begin{remark} 
By restricting to quadrilateral meshes, one may 
alternatively employ generalized alternating fluxes---a 
multidimensional extension of the fluxes in 
\eqref{eq:firstFlux2}---in place 
of \eqref{sec4:numericalFluxes}, as in 
\cite{Cheng:2017aa}, while still 
retaining the property \eqref{sec4:nonpositiveBoundary}.
\end{remark} 

It remains to establish \eqref{sec4:vanishExpectation}, 
which is the content of the next lemma. 

\begin{lemma}\label{sec4:vanishingStochastic}
For any $h>0$, $K\in\T_h$, and $t\ge 0$, define
\begin{equation*}
	B_K^{\ell}(t):=
	\int_0^t\int_K u_h(s)q_{h,\ell}(s)\,dx\,dW_s^{\ell},
	\qquad \ell\in L.
\end{equation*}
For each fixed $n$, the process 
$t\mapsto B_K^{\ell}(\tau_n(t))$, stopped at 
$\tau_n(t)$ as defined in \eqref{sec4:stoppingTimes}, 
is a square-integrable martingale. Consequently,
\begin{align*}
	\E\left[\, \sum_{\ell\in L} B_K^{\ell}(\tau_n(t))\right]
	&=\sum_{\ell\in L}\E\left[\int_0^{\tau_n(t)}\!\!
	\int_K u_h(s)q_{h,\ell}(s)\,dx\,dW_s^{\ell}\right]
	=0.
\end{align*}
\end{lemma}

\begin{proof}
On the set $\{s\le\tau_n\}$, 
we have $\|\bold{u}(s)\|_F\le n$.
Moreover, since
$\bold{q}_{\ell}(s)=\bold{Q}_{\ell}(\bold{u}(s))$ and
$\bold{Q}_{\ell}$ is locally bounded, 
there exists a constant
$C_{n,h,\ell}>0$ such that 
$\|\bold{q}_{\ell}(s)\|_F
\le C_{n,h,\ell}$, for $s\le\tau_n$. 
Using finite-dimensional 
norm equivalence on each element,
equivalently the mass-matrix representation in
\eqref{sec4:localEquivalence}, we obtain
\begin{multline*}
	\biggl(\int_K u_h(s)q_{h,\ell}(s)\,dx\biggr)^2
	\le \|u_h(s)\|_{L^2(K)}^2
	\|q_{h,\ell}(s)\|_{L^2(K)}^2
	\\ =
	\bigl(\bold{u}^K(s)\bigr)^T M^K \bold{u}^K(s)
	\,\bigl(\bold{q}_{\ell}^K(s)\bigr)^T
	M^K\bold{q}_{\ell}^K(s)
	\le
	\lambda_{\max}^2(M^K)\,
	\|\bold{u}(s)\|_F^2\,\|\bold{q}_{\ell}(s)\|_F^2
	\le \widetilde{C}_{n,h,\ell},
\end{multline*}
for some constant $\widetilde{C}_{n, h, \ell}$. Hence,
$\int_K 1_{\{s\le\tau_n\}}u_h(s)q_{h,\ell}(s)\,dx$ 
is square-integrable in time over any finite 
interval. By the stopping identity 
for It\^o integrals,
\begin{align*}
	B_K^{\ell}(\tau_n(t))
	&= \int_0^{t\wedge \tau_n}\int_K u_h(s)
	q_{h,\ell}(s)\,dx\,dW_s^{\ell}
	=\int_0^{t}\int_K 1_{\{s\le\tau_n\}}
	u_h(s)q_{h,\ell}(s)\,dx\,dW_s^{\ell}.
\end{align*}
Since the integrand admits a predictable version and is 
square-integrable on $[0,t]$, 
the It\^o integral has zero mean, and therefore
$\E\left[B_K^{\ell}(\tau_n(t))\right]=0$.
\end{proof}

Combining \eqref{sec4:tempL2}, 
\eqref{sec4:nonpositiveBoundary}, and 
Lemma \ref{sec4:vanishingStochastic}, we obtain 
a uniform energy stability estimate for the stopped 
LDG approximation.

\begin{theorem}[Mean energy stability]
\label{sec4:StabilityTheorem}
Suppose Assumption \ref{sec4:regularityg} holds. 
Let $\bar u\in L^2(\R)$ be deterministic with 
compact support, and suppose that $\T_h$ 
covers the support of $u(t)$ for all 
$t\in[0,T]$. Let $\tau_n(t)$ be 
defined by \eqref{sec4:stoppingTimes}. 
Then the LDG approximation 
\eqref{sec4:ldg1}--\eqref{sec4:ldg2}, 
with initial data \eqref{sec4:ldg3} and 
numerical fluxes 
\eqref{sec4:numericalFluxes}--\eqref{sec4:numericalFluxesBD}, 
satisfies, for all $t\in[0,T]$,
\begin{align*}
	& \E\pigl[\|u_h(\tau_n(t))\|_2^2\pigr]
	-\|u_h(0)\|_2^2
	\\
	& \qquad 
	+\E\Biggl[\, 
	\sum_{e\in\mathcal{E}_h}
	\int_0^{\tau_n(t)}\sum_{\ell\in L}\int_e
	\bigl|\sigma_{\ell}\cdot n_e\bigr|
	\Bigl(\eta_{q,\ell}^e\llbracket u_h\rrbracket^2
	+ \eta_{u,\ell}^e\llbracket q_{h,\ell}\rrbracket^2\Bigr)(s)
	\, d\mathcal{H}^{d-1}(x)\,ds
	\Biggr]\le 0,
\end{align*}
where the term on the second line is non-negative and can be
discarded. Thus, in particular,
\begin{equation*}
	\E\pigl[\|u_h(\tau_n(t))\|_2^2\pigr]
	\le \|u_h(0)\|_2^2.
\end{equation*}

After global existence has been established below in
Corollary \ref{sec4:GlobalExistence}, one may let
$n\to\infty$ and obtain the corresponding estimate without
stopping for each deterministic $t\in[0,T]$.
\end{theorem}

The previous theorem allows us 
to extend the local solution $\bold{u}$ of 
the SDE system to a global one.

\begin{corollary}[Global well-posedness of SDE system]
\label{sec4:GlobalExistence}
Assume that the hypotheses 
of Theorem \ref{sec4:StabilityTheorem} 
are satisfied. Then \eqref{sec4:SDE-system} 
admits a unique strong solution 
$\bold{u}:\Omega\times[0,T]\to\R^{(N_k+1)\times|\T_h|}$
such that $\E\bigl[\|\bold{u}(t)\|_F^2\bigr]\lesssim 1$, 
uniformly in $t$ for $t\in [0,T]$.
\end{corollary}

\begin{proof} 
By Theorem \ref{sec4:StabilityTheorem},
\begin{align}\label{sec4:uniformExp}
	\sup_{t\in[0,T]}
	\E\bigl[ \|u_h(\tau_n(t))\|_2^2 \bigr]
	\le \|u_h(0)\|_2^2\le \|\bar u\|_2^2,
\end{align}
which implies a uniform bound 
on $\E\bigl[\|\bold{u}(\tau_n(t))\|_F^2\bigr]$. 
This bound is the key to extending 
the local strong solution to a 
global strong solution on $[0,T]$. 
Indeed, exploiting that $u_h(t) \in \mathcal{V}^k$ 
for any $t \in [0, T]$, we have a local 
expansion on $K\in\T_h$ of the form
\begin{equation}\label{sec4:multiExp}
	u_h(t, x)
	= \sum_{m=0}^{N_k}u_m^K(t)\phi_m^K(x), 
	\qquad \text{for $x \in K$}, 
\end{equation}
where $N_k+1$ is equal to the number of degrees of 
freedom in $\mathcal{P}^k(K)$, for instance, 
$N_k+1= \tfrac{(k+d)!}{d!k!}$ in the case 
of simplices in dimension $d$.
Using the representation 
\eqref{sec4:multiExp}, we compute
\begin{align}\label{sec4:localEquivalence}
	\|u_h(t)\|_{L^2(K)}^2
	&= \int_K \left(\, \sum_{m=0}^{N_k}
	u_m^K(\omega,t)\phi_m^K(x)\right)
	\left(\, \sum_{l=0}^{N_k} u_l^K(\omega,t)\phi_l^K(x)
	\right)\,dx \nonumber
	\\ & 
	= \sum_{m=0}^{N_k}\sum_{l=0}^{N_k}
	u_m^K(\omega,t)\,u_l^K(\omega,t)
	\int_K \phi_m^K(x)\phi_l^K(x)\,dx 
	= \bigl(\bold{u}^K(t)\bigr)^T M^K \bold{u}^K(t).
\end{align}
In particular, the spatial 
$L^2(K)$-norm is represented 
by a quadratic form in the coefficient vector 
$\bold{u}^K(t)$. Defining the global 
block-diagonal mass matrix 
$M := \mathrm{diag}\bigl(\{M^K\}_{K\in\T_h}\bigr)$, 
we obtain
\begin{align}\label{sec4:equivalenceL2}
	\|u_h(t)\|_{2}^2
	&=\sum_{K \in \T_h} \|u_h(t)\|_{L^2(K)}^2
	=\sum_{K \in \T_h}\bigl(\bold{u}^K(t)\bigr)^T
	M^{K} \bold{u}^K(t)
	=\mathrm{Tr}\pigl(\bold{u}^T(t)
	M \bold{u}(t)\pigr). 
\end{align}
Now introduce the Lyapunov function
\begin{equation}\label{sec4:Lyapunov}
	V(\bold{u}) := 
	\mathrm{Tr}\bigl(\bold{u}^TM \bold{u}\bigr)
	\ge 0.
\end{equation}
From here on, we follow the main steps in 
the proof of \cite[Thm.~3.5]{Khasminskii:2012aa}. 
Since each local mass matrix $M^K$ is symmetric 
and positive definite, there exists, for every 
$K\in\T_h$, a minimal eigenvalue 
$\lambda_{\min}(M^K)>0$ such that
$(\bold{u}^K)^T M^K \bold{u}^K 
\ge \lambda_{\min}(M^K)\,|\bold{u}^K|^2$. 
Consequently,
\begin{align*}
	\lambda^{*}\|\bold{u}\|_F^2
	\le \sum_{K\in\T_h}
	\lambda_{\min}(M^K)\,|\bold{u}^K|^2
	\le V(\bold{u}),
\end{align*}
where
$\lambda^{*} := \min\limits_{K\in\T_h}\lambda_{\min}(M^K) > 0$.
In particular, the Lyapunov function is coercive, i.e.,
\begin{equation}\label{sec4:coercive}
	\inf_{\|\bold{u}\|_F \ge R} V(\bold{u}) \to \infty
	\qquad \text{as } R \to \infty .
\end{equation}

Let $\bold{u}(t)$ denote the unique local solution of the 
SDE system \eqref{sec4:SDE-system} defined on 
$[0,\tau_{\max})$, and let $\{\tau_n\}_{n\in\mathbb{N}}$ 
be the sequence of stopping times 
introduced in \eqref{sec4:stoppingTimes}.
By \eqref{sec4:uniformExp}, 
\eqref{sec4:equivalenceL2}, and 
\eqref{sec4:Lyapunov}, we have
\begin{align*}
	\E\bigl[V(\bold{u}(\tau_n(t)))\bigr] 
	\le \|\bar u\|_2^2 .
\end{align*}
Applying Markov's inequality, or more 
precisely \cite[Lem.~1.4]{Khasminskii:2012aa}, yields
\begin{align*}
	\P\bigl(\tau_n \le t\bigr)
	&\le \frac{\E\bigl[V(\bold{u}(\tau_n(t)))\bigr]}
	{\inf_{\|\bold{w}\|_F \ge n} V(\bold{w})}
	\le \frac{\|\bar u\|_2^2}
	{\inf_{\|\bold{w}\|_F \ge n} V(\bold{w})}.
\end{align*}
Since the events $\{\tau_n \le t\}$ form a decreasing 
sequence, letting $n\to\infty$ and using 
\eqref{sec4:coercive} gives
\begin{align*}
	\P\bigl(\tau_{\max} \le t\bigr)
	=\lim_{n\to\infty}
	\P\bigl(\tau_n \le t\bigr)=0.
\end{align*}
Since this holds for every deterministic
$t<T$, it follows that $\tau_{\max}\ge T$ almost surely.
Equivalently, if the maximal lifetime has been truncated at
$T$, then $\tau_{\max}=T$ almost surely. 
This concludes the proof.
\end{proof}

\begin{remark}
By the proof of Theorem \ref{sec4:GlobalExistence}, the 
global LDG solution satisfies 
the estimate
\begin{equation*}
	\sup_{0 \leq t \leq T}
	\E\pigl[\|u_h(t)\|_2^2\pigr]
	=\sup_{0 \leq t \leq T}
	\E\pigl[V(\bold{u}(t))\pigr]
	\leq \|\bar{u}\|_2^2,
\end{equation*}
where $V(\cdot)$ is the Lyapunov 
function from \eqref{sec4:Lyapunov}. 
\end{remark}

\subsection{Achieving pathwise stability} 
We now seek to ensure that the 
stochastic term vanishes pathwise 
\eqref{sec4:vanishPathwise}. 
To this end, we will impose a specific 
structure on the numerical fluxes $\F_{g_{\ell}(u)}^{e}$ 
appearing in \eqref{sec4:ldg2}.

While Assumption \ref{sec4:regularityg} alone 
is insufficient to guarantee the nonpositivity 
of $\sum_{e\in\mathcal{E}_h}\Phi_e(s)$, with 
$\Phi_e(\cdot)$ defined in \eqref{sec4:tempL2}, this 
property can be restored under additional 
regularity assumptions on $g_{\ell}(\cdot)$, provided 
the penalty parameters $\eta_{q,\ell}^{e}$ 
satisfy a uniform positive lower bound.

\begin{lemma}[Pathwise vanishing of 
stochastic integral]\label{sec4:lemNoiseVanish}
Suppose Assumption \ref{sec4:regularityg} holds. 
Define the numerical flux $\F_{g_{\ell}(u)}^{e}$ 
in \eqref{sec4:ldg2} by
\begin{align}\label{sec4:pathwiseFlux}
	\F_{g_{\ell}(u)}^{e}
	= \bigl(\sigma_{\ell}\cdot n_e\bigr)
	\begin{cases}
		\displaystyle 
		\frac{\llbracket G_{\ell}(u_h)\rrbracket}
		{\llbracket u_h\rrbracket},
		& e=\partial K_+\cap\partial K_-,\\[1ex]
		g_{\ell}(u_h|_{K}),
		& e=\partial K\cap\partial\mathcal{E},
	\end{cases}
	\qquad \ell\in L,
\end{align}
where
\begin{equation}\label{sec4:G}
	G_{\ell}(u):=\int_0^{u} g_{\ell}(z)\,dz.
\end{equation}
Then the pathwise vanishing 
identity \eqref{sec4:vanishPathwise} holds.
\end{lemma}

\begin{proof}
Applying the LDG equation \eqref{sec4:ldg2} 
with $\psi_{\ell} = u_h(s)$ as the 
test function yields
\begin{align}\label{sec4:temp1Path}
	  \int_{K} u_h(s)q_{h, \ell}(s)\, dx 
	 = - \int_{K} 
	 \bigl(\nabla u_h(s)\cdot\sigma_{\ell}\bigr)
	 g_{\ell}(u_h(s))\, dx
	 + \sum_{e \subset \partial K}\int_{e} 
	 (\sigma_{\ell} \cdot n_e)u_h(s) 
	 \widetilde{\F}_{g_{\ell}(u)}^{e}
	 \, d\mathcal{H}^{d-1}(x),
\end{align}
where we used the notation from \eqref{sec4:fluxesWithout}. 
Using that $\sigma_{\ell}$ is divergence-free and introducing 
$G_{\ell}(u)$ as in \eqref{sec4:G}, an 
application of the divergence theorem yields
\begin{align}\label{sec4:temp2Path}
	&- \int_{K}\bigl(\nabla u_h(s)\cdot \sigma_{\ell}\bigr)
	g_{\ell}(u_h(s))\, dx
	= - \int_{K} \mathrm{div}_x
	\bigl(\sigma_{\ell} G_{\ell}(u_h(s))\bigr)\, dx
	\\ & \qquad 
	= - \int_{\partial K}(\sigma_{\ell} \cdot n)
	G_{\ell}(u_h(s))\, d\mathcal{H}^{d-1}(x)
	= -\sum_{e \subset \partial K} \int_{e} 
	(\sigma_{\ell} \cdot n_e) G_{\ell}(u_h(s))
	\, d\mathcal{H}^{d-1}(x). 
	\nonumber 
\end{align}
Combining \eqref{sec4:temp1Path} and 
\eqref{sec4:temp2Path}, and using the conservation 
property \eqref{eq:conservationProp} of the 
flux $\F_{g_{\ell}(u)}^{e}$, we obtain
\begin{align*}
	& \sum_{K \in \T_h} \sum_{\ell \in L}
	\int_0^t\int_{K} u_h(s)q_{h, \ell}(s)
	\, dx\, dW_s^{\ell} 
	\\ & \qquad 
	= \sum_{K \in \T_h} \sum_{e \subset \partial K}
	\sum_{\ell \in L} \int_0^t\int_{e} 
	(\sigma_{\ell} \cdot n_e)
	\pigl(u_h(s)\widetilde{\F}_{g_{\ell}(u)}^{e}
	-G_{\ell}(u_h(s))\pigr)
	\, d\mathcal{H}^{d-1}(x)\, dW_s^{\ell} 
	\\ & \qquad
	= \sum_{e \in \mathcal{E}_h} \sum_{\ell \in L}
	\int_0^t \int_{e}(\sigma_{\ell} \cdot n_e)
	\pigl(\llbracket G_{\ell}(u_h(s)\rrbracket - \llbracket u_h(s) \rrbracket
	\widetilde{\F}_{g_{\ell}(u)}^{e}
	\pigr)
	\, d\mathcal{H}^{d-1}(x)\, dW_s^{\ell}=0,
\end{align*}
where it remains to justify why the integrand vanishes.
If $e=\partial K_+\cap\partial K_-$ is an interior interface,
the definition \eqref{sec4:pathwiseFlux} yields 
$\llbracket u_h\rrbracket
\widetilde{\F}_{g_{\ell}(u)}^{e}
-\llbracket G_{\ell}(u_h)\rrbracket=0$.
If $e=\partial K\cap\partial D$ lies on 
the physical boundary, zeroth–order extrapolation 
identifies the exterior trace with
the interior value, hence
$\llbracket u_h\rrbracket=0$ and
$\llbracket G_{\ell}(u_h)\rrbracket=0$,
so the integrand vanishes there as well.
This establishes the claimed cancellation.
\end{proof}

\begin{remark}\label{sec4:remarkLinear}
For $g_{\ell}(u)=u$, so that 
$G_{\ell}(u)=\frac12 u^2$, then 
$\widetilde{\F}^{e}_{g_{\ell}(u)}
=\frac{\llbracket G_{\ell}(u_h)\rrbracket}
{\llbracket u_h\rrbracket}
=\avg{u_h}$. Thus the flux reduces to the central 
(averaging) flux, and the interface 
contributions cancel identically. 
This exact cancellation yields pathwise energy 
conservation, in agreement 
with Theorem \ref{thm:pathwiseEst}.
\end{remark}

To obtain pathwise $L^2$ stability 
we assume that each $g_{\ell}\in C^2(\R)$ 
has a bounded second derivative. 
This is not an unreasonable restriction 
in the present setting.
As noted in the introduction, 
stochastic conservation laws with
Stratonovich noise and divergence-free 
coefficients $\sigma_{\ell}$ propagate $L^\infty$ 
bounds from the initial data.
Assuming the global boundedness of 
$g_\ell''(\cdot)$, we may choose the penalty
parameters $\eta_{q,\ell}$ in uniform way
to ensure that $\Phi_e$ (see \eqref{sec4:tempL2}) 
is nonpositive. The following 
result makes this precise.

\begin{theorem}[Pathwise energy stability]
\label{sec4:pathwiseStability} 
Suppose Assumption \ref{sec4:regularityg} holds, 
and that $g_{\ell}$ belongs to $C^2(\R)$ with 
$g_{\ell}''\in L^{\infty}(\R)$.  
In \eqref{sec4:ldg2} and \eqref{sec4:semilinear}, 
choose the consistent and conservative numerical fluxes
\begin{align}\label{sec4:pathwiseFlux1_ref}
	\F_{g_{\ell}(u)}^{e}
	=(\sigma_{\ell}\cdot n_e)
	\frac{\llbracket G_{\ell}(u_h)\rrbracket}
	{\llbracket u_h\rrbracket}, 
	\qquad \ell\in L,
\end{align}
with $G_{\ell}$ defined in \eqref{sec4:G}, and
\begin{align}\label{sec4:pathwiseFlux2_ref}
	\F_{g_{\ell}'(u)q_{\ell}}^{e}
	=(\sigma_{\ell}\cdot n_e)
	\left(
	\frac{\llbracket g_{\ell}(u_h)\rrbracket}
	{\llbracket u_h\rrbracket}\avg{q_{h,\ell}}
	+\eta_{q,\ell}^{e}\,
	\mathrm{sgn}(\sigma_{\ell}\cdot n_e)
	\,\llbracket u_h\rrbracket
	\, \bigl|\llbracket q_{h,\ell}\rrbracket\bigr|
	\right),
	\qquad \ell \in L,
\end{align}
with boundary fluxes given 
by \eqref{sec4:numericalFluxesBD}. 
If the penalty parameters in 
\eqref{sec4:pathwiseFlux2_ref} satisfy
\begin{align}\label{sec4:lowerPenalty_ref}
	\eta_{q,\ell}^{e} \ge
	\frac{1}{12}\|g_{\ell}''\|_{\infty},
	\qquad \text{for all 
	$e\in\mathcal{E}_h$ and $\ell\in L$},
\end{align}
then the following 
pathwise energy estimate holds:
\begin{align*}
	\|u_h(t)\|_2^2 \le \|\bar u\|_2^2,
	\quad \text{almost surely},
	\quad t\in[0,T].
\end{align*}
\end{theorem}

\begin{proof}
By Lemma \ref{sec4:lemNoiseVanish}, the 
stochastic integral vanishes identically. 
Since global existence holds by 
Corollary \ref{sec4:GlobalExistence}, we may 
replace $\tau_n(t)$ by $t$ in \eqref{sec4:tempL2}, 
and thus obtain
\begin{align}\label{sec4:pathwise_energy_identity}
	\|u_h(t)\|_2^2
	=\|u_h(0)\|_2^2
	+\sum_{e\in\mathcal{E}_h}\int_0^t \Phi_e(s)\,ds,
	\qquad t\in[0,T],
\end{align}
where $\Phi_e(s)$ is defined in \eqref{sec4:tempL2}.
Substituting 
\eqref{sec4:pathwiseFlux1_ref}--\eqref{sec4:pathwiseFlux2_ref} 
into $\Phi_e$ and using the jump identity 
$\llbracket ab\rrbracket=\avg{a}\llbracket b\rrbracket
+\llbracket a\rrbracket\avg{b}$, we obtain, for 
any interior interface $e$,
\begin{align}\label{sec4:Phi_pathwise_pre}
	\Phi_e(s)
	=\sum_{\ell\in L}\int_e 
	(\sigma_{\ell}\cdot n_e) & \Bigl(
	-\eta_{q,\ell}^e\,
	\mathrm{sgn}(\sigma_{\ell}\cdot n_e)
	\llbracket u_h\rrbracket^2
	\bigl|\llbracket q_{h,\ell}\rrbracket\bigr|
	-\frac{\llbracket G_{\ell}(u_h)\rrbracket}
	{\llbracket u_h\rrbracket}
	\llbracket q_{h,\ell}\rrbracket
	+\avg{g_{\ell}(u_h)}
	\llbracket q_{h,\ell}\rrbracket
	\Bigr)\,d\mathcal{H}^{d-1}(x).
\end{align}

Let $a:=u_h^+$ and $b:=u_h^-$ on 
an interior interface $e$. 
By \eqref{sec4:G}, we have
\begin{align*}
	\frac{\llbracket G_{\ell}(u_h)\rrbracket}
	{\llbracket u_h\rrbracket}-\avg{g_{\ell}(u_h)}
	& = \frac{1}{a-b}\int_b^a g_{\ell}(z)\,dz
	-\frac{g_{\ell}(a)+g_{\ell}(b)}{2}
	\\ &
	=-\frac{1}{12}g_{\ell}''(u^\star)(a-b)^2
	=-\frac{1}{12}g_{\ell}''(u^\star)
	\llbracket u_h\rrbracket^2,
\end{align*}
for some $u^\star$ between $a$ and $b$. The second
equality is the classical error estimate 
for the trapezoidal-rule. 
Substituting this into 
\eqref{sec4:Phi_pathwise_pre} yields
\begin{align*}
	\Phi_e(s)
	&=-\sum_{\ell\in L}\int_e 
	(\sigma_{\ell}\cdot n_e)\Bigl(
	\eta_{q,\ell}^e\,\mathrm{sgn}(\sigma_{\ell}\cdot n_e)
	\llbracket u_h\rrbracket^2
	\bigl|\llbracket q_{h,\ell}\rrbracket\bigr|
	-\frac{1}{12}g_{\ell}''(u^\star)
	\llbracket u_h\rrbracket^2
	\llbracket q_{h,\ell}\rrbracket
	\Bigr)\,d\mathcal{H}^{d-1}(x)
	\\ & 
	= -\sum_{\ell\in L}\int_e |\sigma_{\ell}\cdot n_e|
	\Bigl(\eta_{q,\ell}^e
	-\frac{1}{12}\mathrm{sgn}(\sigma_{\ell}\cdot n_e)
	\mathrm{sgn}(\llbracket q_{h,\ell}\rrbracket)
	g_{\ell}''(u^\star)\Bigr)
	\llbracket u_h\rrbracket^2
	\bigl|\llbracket q_{h,\ell}\rrbracket\bigr|
	\, d\mathcal{H}^{d-1}(x).
\end{align*}

If \eqref{sec4:lowerPenalty_ref} holds, then 
$\Phi_e(s)\le 0$ for every interior interface $e$.
If $e=\partial K\cap\partial D$ 
lies on the physical boundary,
the fluxes are prescribed 
by \eqref{sec4:numericalFluxesBD}.
With zeroth–order extrapolation the jump terms vanish,
so that $\Phi_e(s)=0$ on boundary faces.
Consequently, $\sum_{e\in\mathcal{E}_h}\Phi_e(s)\le 0$,
and inserting this into
\eqref{sec4:pathwise_energy_identity} yields
$\|u_h(t)\|_2^2 \le \|u_h(0)\|_2^2 \le \|\bar u\|_2^2$.
This completes the proof.
\end{proof}

\subsection{Some final remarks}
A seemingly more natural alternative 
to the LDG formulation 
\eqref{sec4:ldg1}--\eqref{sec4:ldg2}
is to expand the It\^o--Stratonovich 
correction term as follows:
\begin{align*}
	\div_x \pigl(\sigma_{\ell} g_{\ell}'(u)
	\div_x\bigl(\sigma_{\ell} g_{\ell}(u)\bigr)\pigr)
	=\div_x \pigl(
	\sigma_{\ell}\bigl(\sigma_{\ell}\cdot\nabla u\bigr)
	\bigl(g_{\ell}'(u)\bigr)^2\pigr)
	=\div_x\bigl(\sigma_{\ell} g_{\ell}'(u)q_{\ell}\bigr),
\end{align*}
where we define
$q_{\ell}:=g_{\ell}'(u)\sigma_{\ell}\cdot\nabla u$. 
With this choice, \eqref{sec4:Ito} 
can be rewritten as
\begin{align*}
	du=\frac{1}{2}\sum_{\ell\in L}
	\div_x\bigl(\sigma_{\ell} g_{\ell}'(u)
	q_{\ell}\bigr)\,dt
	-\sum_{\ell\in L} q_{\ell}\,dW_t^{\ell},
	\qquad
	q_{\ell}
	=g_{\ell}'(u)\,\sigma_{\ell}\cdot\nabla u,
	\quad \ell\in L.
\end{align*}
Similarly as before, we discretize 
this first-order system 
elementwise by an LDG method.
We introduce consistent numerical fluxes
$\widehat{\F}^{e}_{g_{\ell}'(u)q_{\ell}}$
associated with the divergence term
$\div_x\bigl(\sigma_{\ell} g_{\ell}'(u) q_{\ell}\bigr)$
in the $u$–equation, and
$\widehat{\F}^{e}_{g_{\ell}'(u)u}$
associated with the auxiliary variable $q_{\ell}$.
This yields LDG approximations
$u_h(t)\in\mathcal{V}^k$ and
$\{q_{h,\ell}(t)\}_{\ell\in L}\subset\mathcal{V}^k$.
For brevity, we do not restate the full 
LDG formulation here, as it follows the same 
pattern as the schemes introduced earlier.

To compare with the stability mechanism 
in Section \ref{sec4:stability}, 
write the fluxes in the form
$$
\widehat{\F}^{e}_{g_{\ell}'(u)q_{\ell}}
=(\sigma_{\ell}\cdot n_e)
\,\F^{\star,e}_{g_{\ell}'(u)q_{\ell}},
\qquad
\widehat{\F}^{e}_{g_{\ell}'(u)u}
=(\sigma_{\ell}\cdot n_e)
\,\F^{\star,e}_{g_{\ell}'(u)u},
$$
so that the factor $\sigma_{\ell}\cdot n_e$ is 
kept outside, as before. Proceeding as in the 
derivation of \eqref{sec4:tempL2} (and using 
the analogue of Lemma \ref{sec4:cancellationCovar}), 
one arrives at an energy identity of the form
$$
\|u_h(t)\|_2^2
=\|u_h(0)\|_2^2
+\sum_{e\in\mathcal{E}_h}\int_0^t
\widetilde{\Phi}_e(s)\,ds
+\text{stochastic term},
$$
where the critical interface contribution is
\begin{align}\label{sec4:phiTilde_finalremarks}
	\widetilde{\Phi}_e(s)
	:=\sum_{\ell\in L}\int_e (\sigma_{\ell}\cdot n_e)
	\Bigl(
	\llbracket u_h\rrbracket
	\F^{\star,e}_{g_{\ell}'(u)q_{\ell}}
	+\llbracket q_{h,\ell}\rrbracket
	\F^{\star,e}_{g_{\ell}'(u)u}
	-\llbracket g_{\ell}'(u_h)u_hq_{h,\ell}\rrbracket
	\Bigr)(s)\,d\mathcal{H}^{d-1}(x).
\end{align}
In contrast to \eqref{sec4:tempL2}, the last 
term in \eqref{sec4:phiTilde_finalremarks} 
contains the triple product
$\llbracket g_{\ell}'(u_h)u_hq_{h,\ell}\rrbracket$.
Expanding the triple jump 
$\llbracket g_{\ell}'(u_h)u_hq_{h,\ell}\rrbracket$
by a symmetric jump identity, for example
\begin{align*}
	\llbracket abc\rrbracket
	&=\avg{ab}\llbracket c\rrbracket
	+\avg{ac}\llbracket b\rrbracket
	+\avg{bc}\llbracket a\rrbracket
	-\frac12
	\llbracket a\rrbracket
	\llbracket b\rrbracket
	\llbracket c\rrbracket,
	\\
	\llbracket abc\rrbracket
	&=\avg{a}\avg{b}\llbracket c\rrbracket
	+\avg{a}\avg{c}\llbracket b\rrbracket
	+\avg{b}\avg{c}\llbracket a\rrbracket
	+\frac14
	\llbracket a\rrbracket
	\llbracket b\rrbracket
	\llbracket c\rrbracket,
\end{align*}
with $a=g_{\ell}'(u_h)$, $b=u_h$, and $c=q_{h,\ell}$,
produces, in addition to terms involving 
$\llbracket u_h\rrbracket$ and 
$\llbracket q_{h,\ell}\rrbracket$, 
mixed contributions of the form
$\avg{u_h q_{h,\ell}}\,
\llbracket g_{\ell}'(u_h)\rrbracket$ 
and 
$\avg{u_h}\avg{q_{h,\ell}}\,
\llbracket g_{\ell}'(u_h)\rrbracket$.
These terms seem difficult to incorporate 
into the interface fluxes
$\F^{\star,e}_{g_{\ell}'(u)q_{\ell}}$
and
$\F^{\star,e}_{g_{\ell}'(u)u}$
without either violating 
consistency or introducing additional,
uncontrolled interface contributions. 
In particular, there is no 
evident analogue of the sign-definite 
dissipation mechanism used to establish $\Phi_e(s)\le 0$
in the proof of Theorem \ref{sec4:StabilityTheorem}.
For this reason, we do not pursue 
this alternative LDG formulation further.

\medskip

Finally, the LDG framework extends 
to \eqref{sec4:stratonovichSPDE} without 
imposing $\div_x(\sigma_{\ell})=0$. 
One may then introduce
$$
q_{\ell}=\div_x\bigl(\sigma_{\ell} g_{\ell}(u)\bigr)
\quad\text{or}\quad
q_{\ell}=\sigma_{\ell}\cdot\nabla g_{\ell}(u),
$$
depending on whether the $\ell$th noise component 
is written in conservative or transport form. The 
principal cancellation between quadratic covariation 
and the It\^o--Stratonovich correction 
can be retained, but an additional remainder 
involving $\div_x(\sigma_{\ell})$ remains, as at 
the continuous level. Controlling this term in an 
$L^2$ framework restricts the corresponding growth of
the nonlinearities $g_{\ell}$ to be at most linear.

\section{LDG $(k=0)$ as finite difference schemes}
\label{sec:FD}

In this section we show that the
proposed LDG discretizations with piecewise constant
approximations ($k=0$) can be rewritten as
finite difference schemes.
In this form, the schemes can be studied
independently of the local discontinuous Galerkin 
framework and may therefore be 
of interest in their own right. 

This representation also allows for a direct comparison
with the difference method
proposed in \cite{Fjordholm:2023aa}. 
The method proposed in \cite{Fjordholm:2023aa} 
is formulated for transport noise
written in the form \eqref{eq:reformulatedLinear}, 
and a detailed comparison is
carried out in Section \ref{sec:transport}. 
In that work, convergence is also
established within the classical Lax--Richtmyer 
framework, relying on consistency
and $L^2$ stability of the difference scheme, 
with the main challenge
being to obtain grid-independent bounds 
on the discrete $L^2$ norm. 
The finite difference schemes derived 
in Section \ref{sec:continuity} and 
Section \ref{sec:transport} fit naturally into the same
paradigm: they are consistent discretizations 
of \eqref{eq:problem1} or \eqref{eq:reformulatedLinear} 
and satisfy the stability estimates proved earlier. 
Thus, they can be shown to
converge to weak solutions by the 
arguments of \cite{Fjordholm:2023aa}.

We will start from the LDG 
discretization of the continuity equation
\eqref{eq:problem1}, where 
the stochastic gradient term and the
It\^{o}--Stratonovich correction are 
discretized in a compatible way that preserves the
underlying hyperbolic structure and thus lead to 
$L^2$ stable difference schemes. 
Although the construction extends to multiple space
dimensions, we restrict to one dimension for 
clarity, as it already captures the
essential features of both the LDG and 
finite difference approaches.

Throughout this section, we restrict 
attention to a uniform grid, so that
$\Delta x_j=\Delta x$ for all $j$. 
For $k=0$, the discrete solution 
is piecewise constant and can be written as
$$
u_h(t,x)=u_j(t),
\qquad
q_h(t,x)=q_j(t),
\qquad x\in I_j,
$$
where
$I_j=\bigl[(j-\tfrac12)\Delta x,
(j+\tfrac12)\Delta x\bigr]$ 
denotes the $j$th grid cell.  Moreover, we introduce the forward, 
backward, and centered difference operators
$$
D_{+}u_j=\frac{u_{j+1}-u_j}{\Delta x},
\quad
D_{-}u_j=\frac{u_j-u_{j-1}}{\Delta x},
\quad 
D_0u_j(t)=\frac{u_{j+1}(t)
-u_{j-1}(t)}{2\Delta x}.
$$
These operators may also act at cell 
interfaces: for example,
$D_- g_{j+\frac12}
= \frac{g_{j+\frac12}
-g_{j-\frac12}}{\Delta x}$ 
for a numerical flux $g$.

For simplicity of notation, we 
sometimes write the time argument explicitly,
as in $u_j(t)$, and sometimes suppress it. Both
conventions may be used 
within the same sentence or equation.

\subsection{Continuity equation}\label{sec:continuity}
Inserting $\varphi = 1 = \psi$ into the 
bilinear forms $a_j$ and 
$b_j$ from \eqref{eq:bilinearForms}, 
we find that
\begin{align*}
	a_j(u_h(t), 1) &
	=-\frac{1}{2}u_j(t)\int_{I_j}\sigma \sigma''\,dx 
	+ \frac{1}{4} \Big((\sigma^2)'_{j+\frac{1}{2}}
	\F_{u, j+\frac{1}{2}}(t)- (\sigma^2)'_{j-\frac{1}{2}}
	\F_{u, j-\frac{1}{2}}(t) \Big), 
	\\
	b_j(q_h(t), 1) & 
	=-\frac{1}{2}q_j(t)\int_{I_j}\sigma'\,dx 
	+\frac{1}{2}\Bigl(\sigma_{j+\frac{1}{2}}
	\widetilde{\F}_{q, j+\frac{1}{2}}(t)
	-\sigma_{j-\frac{1}{2}}
	\widetilde{\F}_{q, j-\frac{1}{2}}(t) \Bigr)
	\\ &
	=\frac{1}{2}\sigma_{j+\frac{1}{2}}
	\bigl( \widetilde{\F}_{q, j+\frac{1}{2}}(t)-q_j(t) \bigr) 
	-\frac{1}{2}\sigma_{j-\frac{1}{2}}
	\bigl( \widetilde{\F}_{q, j-\frac{1}{2}}(t)-q_j(t)\bigr).
\end{align*}
Consequently, the weak formulation 
\eqref{eq:weak1}--\eqref{eq:weak2} 
and \eqref{eq:weak3} becomes
\begin{align}
	du_j(t)\Dx &= -\frac{1}{2}u_j(t)
	\int_{I_j}\sigma \sigma''\, dx\, dt
	- q_j(t)\Dx \, dW_t
	\nonumber
	\\ & \quad 
	+ \frac{1}{4} \Big((\sigma^2)'_{j+\frac{1}{2}}
	\F_{u, j+\frac{1}{2}}(t)- (\sigma^2)'_{j-\frac{1}{2}}
	\F_{u, j-\frac{1}{2}}(t) \Big)\,dt 
	\label{eq:PCScheme}
	\\ & \quad 
	+ \frac{1}{2}\Big(
	\sigma_{j+\frac{1}{2}}
	\bigl(\widetilde{\F}_{q, j+\frac{1}{2}}(t)-q_j(t)\bigr) 
	-\sigma_{j-\frac{1}{2}}
	\bigl( \widetilde{\F}_{q, j-\frac{1}{2}}(t) 
	-q_j(t) \bigr)\Big)\,dt, 
	\nonumber 
	\\
	q_j(t) & = \frac{\sigma_{j+\frac{1}{2}}
	\widetilde{\F}_{u, j+\frac{1}{2}}
	-\sigma_{j-\frac{1}{2}}
	\widetilde{\F}_{u, j-\frac{1}{2}}}{\Dx},
	\nonumber
\end{align}
with initial values $u_j(0)=
\frac{1}{\Dx} \int_{I_j} \bar{u}\, dx$. 
This class of LDG methods provides 
approximations to weak solutions of the
continuity equation \eqref{eq:problem1}.

Using difference operators, \eqref{eq:PCScheme} 
can be written compactly, with the 
numerical fluxes left implicit, as
\begin{align*}
	du_j(t) &= -\frac{1}{2\Dx}u_j \int_{I_j}
	\sigma \sigma''\, dx \, dt - q_j(t)\, dW_t 
	+\frac{1}{4}D_-\pigl(
	\bigl(\sigma^2\bigr)'_{j+\frac{1}{2}}
	\F_{u, j+\frac{1}{2}}\pigr)\, dt
	\\ & \qquad 
	+ \frac{1}{2} \Bigl(D_- \pigl(\sigma_{j+\frac{1}{2}}
	\widetilde{\F}_{q, j+\frac{1}{2}}\pigr)
	-q_j(t) D_-\sigma_{j+\frac{1}{2}}\Bigr)\, dt, 
	\\
	q_j(t) &= D_-\bigl(\sigma_{j+\frac{1}{2}}
	\widetilde{\F}_{u, j+\frac{1}{2}}\bigr).
\end{align*}
Substituting the numerical flux $\F_u$ 
from \eqref{eq:firstFlux1} together with
the flux pair $(\widetilde{\F}_u,\widetilde{\F}_q)$ 
from \eqref{eq:firstFlux2}, and 
assuming that the parameters 
$(\widetilde{\gamma},\gamma,\theta,\eta_u,\eta_q)$ 
are uniform (i.e., independent of $j$), 
straightforward algebraic manipulations yield
\begin{align}
	du_j(t) & =
	-\frac{1}{2\Dx} u_j(t)
	\int_{I_j}\sigma \sigma''\, dx\, dt 
	- q_j(t)\, dW_t 
	- \frac{1}{2} q_j(t)D_- \sigma_{j+\frac{1}{2}}\, dt
	\nonumber
	\\ & \qquad 
	+ \frac{\gamma}{4}\biggl(
	 D_+\Bigl(
	\bigl(\bigl(\sigma^2\bigr)'_{j-\frac{1}{2}} \vee 0\bigr)
	u_j \Bigr) 
	+D_-\Bigl(
	\bigl(\bigl(\sigma^2\bigr)'_{j+\frac{1}{2}}
	\wedge 0\bigr)u_j\Bigr) 
	\biggr)\, dt 
	\nonumber
	\\ &\qquad 
	+ \frac{1}{8}(1-\gamma)\biggl(
	D_+ \pigl(
	\bigl(\sigma^2\bigr)'_{j-\frac{1}{2}} 
	u_j \pigr)
	+D_- \pigl(
	\bigl(\sigma^2\bigr)'_{j+\frac{1}{2}} u_j \pigr)
	\Bigr)\, dt 
	\nonumber
	\\ & \qquad 
	+ \frac{\widetilde{\gamma}}{4} \Bigl( D_+\pigl(
	\bigl| \bigl(\sigma^2\bigr)'\bigr|_{j-\frac{1}{2}}u_j \pigr) 
	- D_-\pigl(
	\bigl| \bigl(\sigma^2\bigr)'\bigr|_{j+\frac{1}{2}}u_j \pigr) 
	\Bigr)\, dt 
	\label{eq:generalFD}
	\\ & \qquad 
	+ \frac{1}{2}\Bigl(
	\theta D_+\bigl(\sigma_{j-\frac{1}{2}} 
	q_j \bigr) 
	+(1-\theta)D_-\bigl(\sigma_{j+\frac{1}{2}}
	q_j\bigr)
	\Bigr)\, dt 
	\nonumber
	\\ & \qquad 
	+ \frac{\eta_q}{2}\Bigl(D_+ \bigl(|\sigma_{j-\frac{1}{2}}|
	u_j\bigr) 
	- D_-\bigl(|\sigma_{j+\frac{1}{2}}
	|u_j \bigr) \Bigr)\,dt,
	\nonumber 
	\\ 
	q_j(t) &= 
	(1-\theta)D_+\bigl(\sigma_{j-\frac{1}{2}}
	u_{j}\bigr)
	+\theta D_-\bigl(\sigma_{j+\frac{1}{2}}
	u_j\bigr)
	\nonumber 
	+ \eta_u D_+\bigl(|\sigma_{j-\frac{1}{2}}|
	q_j\bigr)
	-\eta_uD_-\bigl(|\sigma_{j+\frac{1}{2}}|
	q_j\bigr).
	\nonumber
\end{align}

We emphasize that this represents the most 
general LDG discretization in the case $k=0$.
Varying the parameters 
$(\widetilde{\gamma},\gamma,\theta,\eta_u,\eta_q)$ 
yields a family of finite difference schemes that 
consistently approximate \eqref{eq:prob1Strat}.

In practice, one typically sets 
$\eta_u=0$ to eliminate the 
auxiliary variable $q_h$ and then
chooses the remaining parameters to 
obtain a simplified scheme with favorable
dissipation properties.

\begin{remark}
If the noise amplitude is constant, 
$\sigma = \bar{\sigma}\in \R$, then the scheme 
\eqref{eq:generalFD} simplifies considerably. 
In particular, we are left with
\begin{align*}
	du_j(t)=\frac{\bar{\sigma}}{2}
	D_-\widetilde{\F}_{q, j+\frac{1}{2}}(t)\, dt
	-q_j(t)\, dW_t, 
	\qquad
	q_j(t)= \bar{\sigma}
	D_-\widetilde{\F}_{u, j+\frac{1}{2}},  
\end{align*}
or with the fluxes explicitly written out, using that $D_+ - D_- = \Dx D_-D_+$:
\begin{align*}
	du_j(t) &= \frac{\bar{\sigma}}{2}
	\Bigl(
	\theta D_+q_j +(1-\theta)D_-q_j
	\Bigr)\, dt
	+ \frac{|\bar{\sigma}|}{2}\eta_q \Dx D_-D_+u_jdt
	-q_j\, dW_t, 
	\\ 
	q_j(t) & 
	=\bar{\sigma}(1-\theta)D_+u_{j}
	+\bar{\sigma}\theta D_-u_j
	\nonumber
	+|\bar{\sigma}|\eta_u\Dx D_-D_+q_j.
\end{align*}
\end{remark}

\begin{remark}
One may use the identity 
$\sigma\sigma''=\tfrac12(\sigma^2)''-(\sigma')^2$
to remove second derivatives of $\sigma$ 
from the discrete formulation. In this 
way, terms involving $\sigma''$ can be 
rewritten using only $(\sigma')^2$ and 
interface contributions containing $(\sigma^2)'$, which 
may be incorporated into the numerical fluxes. 
This reformulation may be useful when extending 
the scheme to settings where $\sigma$ has less 
regularity. The resulting difference scheme, however, is no 
longer in conservative form.
\end{remark}

We now present the difference 
schemes arising from \eqref{eq:generalFD} 
for representative choices of the 
numerical fluxes $\F_u$, $\widetilde{\F}_u$, 
and $\widetilde{\F}_q$, emphasizing 
the resulting dissipation properties.

\subsubsection{Conservative difference scheme}
We refer to the difference scheme obtained from
\eqref{eq:generalFD} by selecting 
$\F_u$, $\widetilde{\F}_u$, $\widetilde{\F}_q$
(see \eqref{eq:firstFlux1} and \eqref{eq:firstFlux2})
as central fluxes as a \emph{conservative} scheme.
This terminology reflects the exact preservation of 
the discrete energy, that is, the mean of the squared 
$L^2$ norm of $u_h$. When $\sigma$ is constant, 
this energy conservation holds pathwise. 
Recall that any choice $\theta\in[0,1]$ in 
\eqref{eq:firstFlux2}
yields a conservative scheme, 
whereas only the choice $\theta=\tfrac12$ 
preserves the discrete energy pathwise.

Accordingly, we set
$(\widetilde{\gamma},\gamma,\theta,\eta_u,\eta_q)
=(0,0,\tfrac12,0,0)$ in \eqref{eq:generalFD}. 
With this choice, the contribution
$$
-\tfrac12 q_j(t)D_-\sigma_{j+\frac12}
+\tfrac12\Bigl(
\theta D_+\bigl(\sigma_{j-\frac12}q_j(t)\bigr)
+(1-\theta)D_-\bigl(\sigma_{j+\frac12}q_j(t)\bigr)
\Bigr)
$$
reduces, for $\theta=\tfrac12$, to
$$
-\tfrac12 q_j(t)D_-\sigma_{j+\frac12}
+\tfrac14\Bigl(
D_+\bigl(\sigma_{j-\frac12}q_j(t)\bigr)
+D_-\bigl(\sigma_{j+\frac12}q_j(t)\bigr)
\Bigr).
$$
Applying the discrete Leibniz rule 
($D_\pm(h_j g_j)=h_j D_\pm g_j
+g_{j\pm1}D_\pm h_j$) shows that the 
contributions proportional to 
$q_j(t)D_-\sigma_{j+\frac12}$ 
cancel exactly, leaving the centered expression 
$\tfrac14(\sigma_{j+\frac12}D_+q_j
+\sigma_{j-\frac12}D_-q_j)$. 
As a result, \eqref{eq:generalFD} simplifies to
\begin{equation*}
	\begin{split}
		du_j(t)
		&=
		-\frac{1}{2\Delta x}u_j(t)\int_{I_j}
		\sigma\sigma''\,dx\,dt
		-q_j(t)\,dW_t
		+\frac18\Bigl(
		D_+\bigl((\sigma^2)'_{j-\frac12}u_j(t)\bigr)
		+D_-\bigl((\sigma^2)'_{j+\frac12}u_j(t)\bigr)
		\Bigr)\,dt\\
		&\qquad \qquad 
		+\frac14\Bigl(
		\sigma_{j+\frac12}D_+q_j(t)
		+\sigma_{j-\frac12}D_-q_j(t)
		\Bigr)\,dt,		
	\end{split}
\end{equation*}
with $q_j(t)=\tfrac12\bigl(D_+(\sigma_{j-\frac12}u_j(t))
+D_-(\sigma_{j+\frac12}u_j(t))\bigr)$. 
Eliminating $q_j$ yields a closed scheme for $u_j$:
\begin{align}
	du_j(t)
	&=
	-\frac{1}{2\Delta x}u_j(t)
	\int_{I_j}\sigma\sigma''\,dx\,dt
	-\frac12\Bigl(
	D_+\bigl(\sigma_{j-\frac12}u_j(t)\bigr)
	+D_-\bigl(\sigma_{j+\frac12}u_j(t)\bigr)
	\Bigr)\,dW_t
	\nonumber 
	\\
	&\qquad
	+\frac18\Bigl(
	D_+\bigl((\sigma^2)'_{j-\frac12}u_j(t)\bigr)
	+D_-\bigl((\sigma^2)'_{j+\frac12}u_j(t)\bigr)
	\Bigr)\,dt
	\label{eq:FDcentralEliminated} 
	\\
	&\qquad
	+\frac18\Bigl(
	\sigma_{j+\frac12}D_+D_-\bigl(\sigma_{j+\frac12}u_j(t)\bigr)
	+\sigma_{j+\frac12}D_+D_+\bigl(\sigma_{j-\frac12}u_j(t)\bigr)
	\nonumber
	\\
	&\qquad\qquad\qquad
	+\sigma_{j-\frac12}D_-D_-\bigl(\sigma_{j+\frac12}u_j(t)\bigr)
	+\sigma_{j-\frac12}D_-D_+\bigl(\sigma_{j-\frac12}u_j(t)\bigr)
	\Bigr)\,dt.
	\nonumber
\end{align}

For a constant noise amplitude ($\sigma=\bar\sigma\in\R$), 
the scheme \eqref{eq:FDcentralEliminated} 
simplifies to the compact form
\begin{equation}\label{eq:FDcentralConst}
	du_j(t)
	=-\bar\sigma D_0u_j(t)\,dW_t
	+\frac{\bar\sigma^2}{2} D_0^2u_j(t)\,dt ,
\end{equation}
where $D_0u_j=\tfrac12(D_+u_j+D_-u_j)$ and
$$
D_0^2u_j=\tfrac14\bigl(D_+D_+ 
+ D_+D_- + D_-D_+ + D_-D_-\bigr)u_j.
$$
\begin{remark}
Note that $D_0^2u_j$ is a centered 
second-order approximation of 
the Laplacian $u_{xx}$, using a wider (five-point) stencil 
than the standard three-point operator 
$D_-D_+u_j$. The operator $D_0^2$ 
arises as the square of the centered difference 
$D_0$ used in the stochastic transport term. 
By our general stability theory, this leads 
to exact pathwise energy conservation. 
By contrast, difference schemes based 
on the operators $D_-D_+$ or $D_+D_-$ 
(see \eqref{eq:FDupwindConst}) 
do not preserve the energy.
\end{remark}

\subsubsection{Dissipative difference schemes}
In the ``fully'' dissipative case we 
specify $(\gamma,\theta)=(1,1)$ in 
\eqref{eq:generalFD} (one may also choose $\theta=0$) 
and, for simplicity, assume that 
$\widetilde{\gamma}$, $\eta_u$, $\eta_q$ 
are grid-independent nonnegative constants. 
With $\theta=1$, the $q$-dependent 
drift contribution in 
\eqref{eq:generalFD} becomes
\begin{align*}
	& -\tfrac12\,q_j(t)\,D_-\sigma_{j+\frac12}
	+\tfrac12\Bigl(
	\theta D_+\bigl(\sigma_{j-\frac12}
	q_j(t)\bigr)
	+(1-\theta)D_-\bigl(\sigma_{j+\frac12}
	q_j(t)\bigr)
	\Bigr)
	\\ & \qquad \qquad
	=-\tfrac12\,q_j(t)\,D_-\sigma_{j+\frac12}
	+\tfrac12\,D_+\bigl(\sigma_{j-\frac12}q_j(t)\bigr).
\end{align*}
Applying the discrete Leibniz rule yields
$$
D_+\bigl(\sigma_{j-\frac12}q_j\bigr)
=\sigma_{j-\frac12}D_+q_j+q_j D_+\sigma_{j-\frac12},
$$
and since $D_+\sigma_{j-\frac12}=D_-\sigma_{j+\frac12}$, 
the terms proportional 
to $q_j(t)D_-\sigma_{j+\frac12}$ cancel, 
leaving the one-sided expression
$\tfrac12\,\sigma_{j-\frac12}D_+q_j(t)$. 
Consequently, \eqref{eq:generalFD} reduces 
to the dissipative finite difference scheme
\begin{align}
	du_j(t) &=
	-\frac{1}{2\Delta x}u_j(t)\int_{I_j}
	\sigma\sigma''\,dx\,dt
	- q_j(t)\,dW_t \nonumber\\
	&\quad + \frac{1}{4}\Bigl(
	D_+\bigl(((\sigma^2)'_{j-\frac12}\vee 0)\,u_j(t)\bigr)
	+D_-\bigl(((\sigma^2)'_{j+\frac12}\wedge 0)\, u_j(t)\bigr)
	\Bigr)\,dt 
	\nonumber
	\\ &\quad 
	+ \frac{\widetilde{\gamma}}{4}\Bigl(
	D_+\bigl(|(\sigma^2)'|_{j-\frac12}u_j(t)\bigr)
	- D_-\bigl(|(\sigma^2)'|_{j+\frac12}u_j(t)\bigr)\Bigr)\,dt
	\label{eq:FDSchemeUpwind} 
	\\ &\quad 
	+ \frac{1}{2}\sigma_{j-\frac12}D_+q_j(t)\,dt
	+ \frac{\eta_q}{2}\Bigl(
	D_+\bigl(|\sigma_{j-\frac12}|u_j(t)\bigr)
	-D_-\bigl(|\sigma_{j+\frac12}|u_j(t)\bigr)\Bigr)\,dt, 
	\nonumber
	\\
	q_j(t) &= D_-\bigl(\sigma_{j+\frac12}u_j(t)\bigr)
	+ \eta_u\Bigl(D_+\bigl(|\sigma_{j-\frac12}|q_j(t)\bigr)
	- D_-\bigl(|\sigma_{j+\frac12}|q_j(t)\bigr)\Bigr), 
	\nonumber
\end{align}
with initial values 
$u_j(0)= \frac{1}{\Delta x}\int_{I_j}\bar{u}\,dx$.

In the case of a constant noise amplitude 
($\sigma=\bar{\sigma}\in\R$)
we use $D_+-D_- = \Delta x D_-D_+$
to rewrite \eqref{eq:FDSchemeUpwind} as
\begin{align*}
	du_j(t)
	&= \tfrac12\bar{\sigma}D_+ q_j(t)\,dt
	+ \tfrac{\eta_q}{2} |\bar{\sigma}|\,
	\Delta x D_-D_+u_j(t)\,dt
	- q_j(t)\,dW_t, \\
	q_j(t)
	&= \bar{\sigma}D_-u_j(t)
	+ \eta_u|\bar{\sigma}| \Delta x D_-D_+q_j(t).
\end{align*}
In the limiting case $\eta_u=\eta_q=0$, 
this reduces to the scheme
\begin{equation}\label{eq:FDupwindConst}
	du_j(t)
	=\tfrac12\bar{\sigma}^2 D_+D_-u_j(t)\,dt
	-\bar{\sigma} D_-u_j(t)\,dW_t.
\end{equation}

For $\theta=0$, this is unchanged except that $D_+q_j$ in the $u_j$ equation is replaced by $D_-q_j$, while the 
$D_-u_j$ in the $q_j$ equation is replaced by $D_+u_j$, so one obtains $D_-D_+u_j(t)$ in the constant noise intensity case with vanishing penalty parameters. 

\begin{remark}
In \eqref{eq:FDupwindConst}, 
the stochastic transport term is 
discretized using a backward difference,
while the drift involves the 
three-point Laplacian $D_+D_-$, resulting
in built-in numerical dissipation. By contrast, the 
conservative discretization leads to the centered scheme
\eqref{eq:FDcentralConst}, based on 
$D_0$ and $D_0^2$, which preserves the 
energy but does not provide upwind stabilization.
\end{remark}

\subsection{Transport equation}\label{sec:transport}
In this subsection, we compare LDG 
difference schemes with the dissipative 
difference scheme of \cite{Fjordholm:2023aa}, which 
for the non-conservative equation 
\eqref{eq:problem1-noncons} takes the form
\begin{equation}\label{eq:ConvFD}
	\begin{split}
		du_j(t) &= \frac{1}{2}
		\sigma_j D_+\left(\frac{\sigma_{j-1}
		+\sigma_j}{2}D_-u_j\right)\, dt 
		-\frac{1}{2}\sqrt{\sigma_jH_j(\avg{\sigma})}
		\bigl(D_+u_j + D_-u_j\bigr)\, dW_t,		
	\end{split}	
\end{equation}
where it is assumed that $\sigma(\cdot) > 0$, 
$\sigma_j = \frac{1}{\Dx}\int_{I_j}\sigma(x)\, dx$, and 
$H_j(\avg{\sigma})$ is the harmonic average
\begin{align*}
	H_j(\avg{\sigma})
	=\frac{1}{\frac{1}{\sigma_{j-1}+\sigma_j} 
	+\frac{1}{\sigma_j+\sigma_{j+1}}}.
\end{align*}

This scheme approximates solutions of 
the transport equation \eqref{eq:reformulatedLinear} 
with mean energy stability, preserving the 
hyperbolicity of the equation 
by appropriately accounting for the 
It\^{o}--Stratonovich correction.  
For a constant noise function $\sigma(x)=\bar{\sigma}>0$, 
one has $H_j(\avg{\sigma})=\bar{\sigma}$, and 
the scheme \eqref{eq:ConvFD} reduces to
\begin{equation}\label{eq:FDconst}
	du_j(t)=\frac{1}{2}\bar{\sigma}^2 D_+D_-u_j 
	\, dt -D_0u_j\, dW_t,
\end{equation}

Testing the LDG equations 
\eqref{sec3:nonCons1} and \eqref{sec3:nonCons2} 
with the constant function $1$, we 
obtain, after some straightforward algebra,
\begin{align}\label{eq:LDGFD}
	du_j(t) & = \frac{1}{4}u_j
	D_-\pigl(\bigl(\sigma^2\bigr)'\pigr)_{j+\frac{1}{2}}\, dt 
	- q_j(t)\, dW_t 
	- \frac{1}{4} D_-\pigl(\bigl(\sigma^2\bigr)'
	\F_u\pigr)_{j+\frac{1}{2}}\, dt 
	+ \frac{1}{2}D_-\pigl(\sigma 
	\widetilde{\F}_q\pigr)_{j+\frac{1}{2}}\, dt, 
	\nonumber 
	\\
	q_j(t) &= - u_j D_-\sigma_{j+\frac{1}{2}}
	+D_-\bigl(\sigma \widetilde{\F}_u\pigr)_{j+\frac{1}{2}}.
\end{align}
Clearly, the scheme \eqref{eq:ConvFD} from 
\cite{Fjordholm:2023aa} and the 
LDG-induced difference scheme \eqref{eq:LDGFD} 
constitute different discretizations, 
and we do not attempt to match
their numerical fluxes explicitly. Since
$\sigma_{j+\frac12}=\sigma(x_{j+\frac12})$ 
and \eqref{eq:ConvFD} employs harmonic averaging at cell 
interfaces, the two schemes cannot be related 
through any choice of DG-compatible fluxes, 
which must depend only on information from the adjacent
cells $I_j$ and $I_{j+1}$.

However, a meaningful comparison is possible 
for a constant noise
amplitude $\sigma(x)=\bar\sigma>0$. 
In this setting, choosing the central flux
$\widetilde{\F}_u=\avg{u_h}$, the auxiliary 
variable in \eqref{eq:LDGFD} simplifies to
$$
q_j(t)
=D_-\bigl(\bar\sigma\,\widetilde{\F}_u\bigr)_{j+\frac12}
=\frac{\bar\sigma}{2}\bigl(D_+u_j(t)+D_-u_j(t)\bigr)
=\bar\sigma D_0u_j(t).
$$
Consistency then requires
us to take $\widetilde{\F}_q=\avg{q_h}
+\eta_q\,\mathrm{sgn}(\bar\sigma)
\llbracket u_h\rrbracket$ in \eqref{eq:LDGFD}, 
where $\mathrm{sgn}(\bar\sigma)=1$. This leads to
$$
\frac12 
D_-\bigl(\sigma\widetilde{\F}_q\bigr)_{j+\frac12}
=\frac{\bar\sigma}{2}D_0 q_j
+\eta_q\bar\sigma\Delta x D_-D_+u_j.
$$
Inserting this into \eqref{eq:LDGFD} yields
\begin{align}\label{eq:LDGconst}
	du_j(t)
	& =\frac{\bar\sigma^2}{2}D_0^2u_j(t)\,dt
	+\eta_q\bar\sigma\Delta xD_-D_+u_j(t)\,dt
	-\bar\sigma D_0u_j(t)\,dW_t.
\end{align}

The scheme \eqref{eq:FDconst} from \cite{Fjordholm:2023aa} 
and the LDG-induced scheme \eqref{eq:LDGconst} 
both use a centered discretization $D_0$ for the stochastic
transport term. The difference lies in the drift: 
\eqref{eq:FDconst} employs the
standard Laplacian $D_+D_-$ and is inherently dissipative, 
whereas the LDG-induced scheme \eqref{eq:LDGconst} 
is based on the fully centered operator $D_0^2$ 
and allows for an optional dissipative term 
$\eta_q\bar\sigma\Delta x D_-D_+$. In particular, when
$\eta_q=0$, the scheme \eqref{eq:LDGconst} is 
energy-conservative, in contrast
to \eqref{eq:FDconst}.

\medskip

We now return to the general variable 
coefficient case $\sigma=\sigma(x)$.  
A broad class of conservative 
difference schemes is obtained from
\eqref{eq:LDGFD} by choosing the numerical 
fluxes as follows: the flux $\F_u$
is taken from \eqref{eq:firstFlux1} with parameters 
$(\widetilde{\gamma},\gamma)=(0,0)$,
while the flux pair 
$(\widetilde{\F}_u,\widetilde{\F}_q)$ is chosen
according to \eqref{eq:firstFlux2} 
with an arbitrary $\theta\in[0,1]$ and 
$(\eta_u,\eta_q)=(0,0)$.  With these selections 
the formulation \eqref{eq:LDGFD} yields 
the conservative difference schemes
\begin{align}
	du_j(t) &= 
	-\frac{1}{8}\pigl( \bigl(\sigma^2\bigr)'_{j+\frac{1}{2}}D_+u_j 
	+ \bigl(\sigma^2\bigr)'_{j-\frac{1}{2}}D_-u_j \pigr)\, dt 
	- q_j(t)\, dW_t
	\nonumber \\ & \qquad 
	+\frac{1}{2}\Bigl(
	\theta D_+\bigl(\sigma_{j-\frac{1}{2}} q_j\bigr)
	+(1-\theta)D_-\pigl(\sigma_{j+\frac{1}{2}}
	q_j \pigr)
	\Bigr)\, dt
	\label{eq:cons-1dscheme-tmp},
	\\
	q_j(t) & = 
	-u_j D_- \sigma_{j+\frac{1}{2}} 
	+(1-\theta)D_+\bigl(\sigma_{j-\frac{1}{2}}
	u_j \bigr)
	+\theta D_- \bigl(\sigma_{j+\frac{1}{2}}
	u_j\bigr).
	\nonumber
\end{align}
By choosing either $\theta=1$ 
or $\theta=0$ in \eqref{eq:cons-1dscheme-tmp} 
and applying the discrete Leibniz rule in 
the $q_j$ equation, the LDG-induced 
formulation \eqref{eq:cons-1dscheme-tmp} 
reduces to particularly simple 
finite difference schemes, in
which the stochastic transport term 
admits straightforward one-sided
discretizations via
\begin{align}\label{eq:auxiliaryFD}
	q_j(t) =
	\begin{cases}
		\sigma_{j-\frac{1}{2}}D_-u_j,
		& \text{if $\theta = 1$}, 
		\\
		\sigma_{j+\frac{1}{2}}D_+u_j, 
		&\text{if $\theta = 0$}.
	\end{cases}
\end{align}

To obtain dissipative schemes 
suitable for practical
computations, start from the general LDG 
formulation \eqref{eq:LDGFD} and use the
numerical fluxes $\F_u$ and 
$(\widetilde{\F}_u,\widetilde{\F}_q)$ defined in
\eqref{eq:firstFlux1} and \eqref{eq:firstFlux2}. 
In contrast to the conservative 
case, we set $\eta_u=0$ but introduce 
numerical dissipation by
choosing $\eta_q>0$ and $\widetilde{\gamma}>0$, 
and we take $\gamma=1$ in
\eqref{sec3:upwindPart}. Since the choice of $\theta$ 
does not affect the dissipation 
mechanism, we fix $\theta=1$ for simplicity. 
Applying the discrete Leibniz rule 
and eliminating the auxiliary variable
$q_j$ via \eqref{eq:auxiliaryFD} yields 
the following dissipative finite
difference scheme:
\begin{align*}
	du_j(t) &=
	-\frac{1}{4}
	\Bigl(
	\pigl(\bigl(\sigma^2 \bigr)'_{j+\frac{1}{2}}\wedge 0 \pigr)
	D_+u_j
	+\pigl(\bigl(\sigma^2\bigr)'_{j-\frac{1}{2}}\vee 0 \pigr)
	D_-u_j
	\Bigr)\, dt 
	-\sigma_{j-\frac{1}{2}}D_-u_j\, dW_t
	\\ & \qquad 
	+ \frac{\widetilde{\gamma}}{4}
	\Bigl(
	D_+\pigl(\bigl|\bigl(\sigma^2 \bigr)'\bigr|_{j-\frac{1}{2}}
	u_j \pigr)
	- D_-\pigl(\bigl|\bigl(\sigma^2 \bigr)'\bigr|_{j+\frac{1}{2}}
	u_j \pigr) \Bigr)\, dt
	\\ & \qquad 
	+\frac{1}{2}D_+\bigl(\sigma_{j-\frac{1}{2}}D_-u_j\bigr)\, dt  
	+\frac{\eta_q}{2}\Bigl(
	D_+\bigl(\bigl|\sigma_{j-\frac{1}{2}}\bigr|u_j\bigr)
	-D_-\bigl(\bigl|\sigma_{j+\frac{1}{2}}\bigr|u_j\bigr)
	\Bigr)\, dt.
\end{align*}

\subsection{Nonlinear continuity equation}
Finally, we present the LDG induced difference schemes 
for the stochastic conservation law \eqref{sec4:stratonovichSPDE} 
in one spatial dimension ($d=1$), assuming 
that $\sigma_{\ell}\equiv 1$ for all $\ell\in L$: 
$$
du+\sum_{\ell \in L}
\partial_x g_l(u) \circ dW^{\ell}_t = 0.
$$ 
Proceeding as in Section \ref{sec:continuity}, 
and testing \eqref{sec4:ldg1} and \eqref{sec4:ldg2} 
with the constant function $1$, we obtain
\begin{align*}
	du_j(t) &= \frac{1}{2} \sum_{\ell \in L}
	D_-\F_{g_{\ell}'(u)q_{\ell}, j+\frac{1}{2}}\, dt
	- \sum_{\ell \in L} q_{\ell, j} \, dW_t^{\ell}, 
	\qquad 
	q_{\ell, j} = D_- \F_{g_{\ell}(u), j+\frac{1}{2}}, 
	\quad \ell \in L. 
\end{align*}
After substituting the numerical fluxes 
from \eqref{sec4:numericalFluxes}, and choosing 
the parameters $\eta_{q,\ell}$, $\eta_{u,\ell}$ 
as uniform constants, independent 
of the cell interfaces, we arrive at the following 
family of mean-square stable schemes (see 
Theorem \ref{sec4:StabilityTheorem} and 
Corollary \ref{sec4:GlobalExistence}):
\begin{align*}
	du_j(t) &= \frac{1}{4}\sum_{\ell \in L}
	\bigg(D_-\Bigl(\frac{D_-g(u_{j+1})}{D_-u_{j+1}}
	q_{\ell, j+1} \Bigr) 
	+D_- \Bigl( \frac{D_+g(u_j)}{D_+u_j}
	q_{\ell,j} \Bigr) \biggr)\, dt 
	\\ & \qquad 
	+ \frac{1}{2}\sum_{\ell \in L} \eta_{q, \ell} \Dx D_-D_+u_j\, dt 
	- \sum_{\ell \in L} q_{\ell,j}\, dW_t^{\ell}, 
	\\ q_{\ell,j} &= D_0g(u_j) 
	+ \eta_{u,\ell}\Dx D_-D_+q_{\ell, j} 
	\qquad \text{for $\ell \in L$}.
\end{align*}
In the case $\eta_{q,\ell}=\eta_{u,\ell}=0$ 
for each $\ell\in L$, the resulting scheme is 
conservative and preserve the $L^2$-norm of $u_h$ 
in expectation. Since $q_{\ell, j} = D_0g(u_j)$ 
in this case, the scheme reads: 
\begin{align*}
	du_j(t) &= \frac{1}{4} \sum_{\ell \in L}
	\bigg(D_-\Bigl(\frac{D_-g(u_{j+1})}{D_-u_{j+1}}
	D_0g(u_{j+1}) \Bigr) 
	+D_- \Bigl( \frac{D_+g(u_j)}{D_+u_j}
	D_0g(u_j) \Bigr) \biggr)\, dt 
	- \sum_{\ell \in L} D_0g(u_j)\, dW_t^{\ell}.
\end{align*}

If each $g_{\ell}$ 
has a bounded second derivative, then by choosing 
$\F_{g_{\ell}(u)}$ and $\F_{g'_{\ell}(u)q_{\ell}}$ 
according to \eqref{sec4:pathwiseFlux1_ref} 
and \eqref{sec4:pathwiseFlux2_ref}, 
respectively, we obtain
\begin{align*}
	du_j(t) &= \frac{1}{4}\sum_{\ell\in L}
	\bigg(D_-\Bigl(\frac{D_-g(u_{j+1})}{D_-u_{j+1}}
	q_{\ell, j+1} \Bigr)
	+D_- \Bigl(\frac{D_+g(u_j)}{D_+u_j}
	q_{\ell, j}\Bigr) \biggr)\, dt 
	\\ & \qquad 
	+\frac{1}{2}\sum_{\ell \in L} \eta_{q, \ell} 
	\Dx^2 D_- \Bigl( \pigl|D_+q_j \pigr|D_+u_j \Bigr)\, dt  
	-\sum_{\ell \in L} q_{\ell, j}\, dW_t^{\ell},
	\\ q_{\ell, j} & = 
	D_- \biggl(\frac{D_+ G_{\ell}(u_j)}{D_+u_j}\biggr),
	\qquad 
	G_{\ell}(u)=\int_0^{u}g_{\ell}(z)\, dz,
	\qquad \text{for $\ell \in L$}. 
\end{align*}
These schemes are pathwise stable by 
Theorem \ref{sec4:pathwiseStability}
provided $\eta_{q, \ell} \geq \frac{1}{12}
\|g_{\ell}''\|_{\infty}$ for every $\ell \in L$.

\section{Numerical experiments}\label{sec:NumericalExamples}

The LDG schemes \eqref{eq:weak1}--\eqref{eq:weak3} 
and \eqref{sec4:ldg1}--\eqref{sec4:ldg3} are 
implemented in \texttt{NGSolve}~\cite{NGSolve}, 
a high--performance multiphysics finite element 
software widely used in solid mechanics, 
electromagnetics, and fluid dynamics. 
Since \texttt{NGSolve} does not provide a native 
interface for stochastic differential equations, 
the resulting matrix--valued SDE systems, 
i.e.\ \eqref{eq:sDESystem} and \eqref{sec4:SDE-system}, 
are advanced in time manually.

For the temporal discretization, we employ the 
explicit derivative--free SDE solver of 
\cite[Sec.~11.2]{Kloeden:1992aa}, which has 
strong order $\nicefrac{3}{2}$. 
We refer to this method as SRK3/2. 
For polynomial degrees $k \leq 2$, 
this scheme provides a favorable compromise 
between accuracy and computational cost.

Let $b(u)$ denote the drift matrix of the 
matrix--valued SDE system 
(either \eqref{eq:sDESystem} or 
\eqref{sec4:SDE-system}), and let 
$\Sigma_{\ell}(u)$ denote the $\ell$th diffusion 
matrix for $\ell \in L \cup \{0\}$, where we set 
$\Sigma_{0}(u) := b(u)$. 
In the case of multidimensional noise, 
the scheme reads
\begin{align}\label{eq:SDE-solverGeneral}
	\bold{u}^{n+1} &= \bold{u}^n 
	+ b(\bold{u}^n)\Dt_n 
	+ \sum_{\ell \in L} 
	\Sigma_{\ell}(\bold{u}^n)\Delta W_n^{\ell} 
	+ \frac{1}{2\sqrt{\Dt_n}} 
	\sum_{\ell_2\in L\cup\{0\}}
	\sum_{\ell_1 \in L} 
	\pigl(\Sigma_{\ell_2}(\bold{u}_{\ell_1}^+) 
	-\Sigma_{\ell_2}(\bold{u}_{\ell_1}^-)
	\pigr) I_{(\ell_1, \ell_2)} \nonumber
	\\
	&\qquad 
	+ \frac{1}{2\Dt_n} 
	\sum_{\ell_2\in L\cup\{0\}} 
	\sum_{\ell_1 \in L}
	\pigl(\Sigma_{\ell_2}(\bold{u}_{\ell_1}^+) 
	- 2\Sigma_{\ell_2}(\bold{u}^n)
	+\Sigma_{\ell_2}(\bold{u}_{\ell_1}^-)\pigr)
	I_{(0, \ell_2)} \nonumber
	\\
	&\qquad 
	+ \frac{1}{2\Dt_n} 
	\sum_{\ell_1, \ell_2, \ell_3 \in L}
	\pigl( \Sigma_{\ell_3}(\Phi_{\ell_1, \ell_2}^+)
	-\Sigma_{\ell_3}(\Phi_{\ell_1, \ell_2}^-)
	-\Sigma_{\ell_3}(\bold{u}_{\ell_1}^+)
	+\Sigma_{\ell_3}(\bold{u}_{\ell_1}^-)
	\pigr) 
	I_{(\ell_1, \ell_2, \ell_3)} .
\end{align}
Here $\Dt_n = t_n - t_{n-1}$ is the step size 
of a (possibly nonuniform) time discretization 
$\{t_n\}_{n=0}^N$ of $[0,T]$, and
\begin{align*}
	\bold{u}_{\ell}^{\pm} 
	&:= \bold{u}^n
	+\frac{1}{|L|} b(\bold{u}^n)\Dt_n 
	\pm \Sigma_{\ell}(\bold{u}^n)\sqrt{\Dt_n}, 
	\quad 
	\Phi_{\ell_1, \ell_2}^{\pm} 
	:= \bold{u}_{\ell_1}^+ 
	\pm \Sigma_{\ell_2}(\bold{u}_{\ell_1}^+) 
	\sqrt{\Dt_n}.
\end{align*}
The quantities $I_{(\ell_1, \ell_2)}$, 
$I_{(0, \ell_2)}$, and 
$I_{(\ell_1, \ell_2, \ell_3)}$ denote 
multiple It\^o integrals, which are approximated 
using the method of repeated integrals; 
see, e.g., \cite[Sec.~5.8 and 10.4]{Kloeden:1992aa}.

In the case of scalar noise ($|L| = 1$), 
the scheme simplifies considerably to
\begin{align}\label{eq:SDE-solver}
	\bold{u}^{n+1} &= \bold{u}^n 
	+ b(\bold{u}^n)\Dt_n 
	+ \Sigma(\bold{u}^n)\Delta W_n  \nonumber
	\\
	&\qquad 
	+ \frac{1}{2\sqrt{\Dt_n}}
	\bigl(b(\bold{u}^+) - b(\bold{u}^-)\bigr)
	\Delta Z_n  
	+ \frac{1}{4} 
	\bigl(b(\bold{u}^+) 
	- 2b(\bold{u}^n) 
	+ b(\bold{u}^-)\bigr)\Dt_n \nonumber
	\\
	&\qquad 
	+ \frac{1}{4\sqrt{\Dt_n}} 
	\bigl(\Sigma(\bold{u}^+) 
	- \Sigma(\bold{u}^-)\bigr)
	\bigl((\Delta W_n)^2-\Dt_n\bigr) \nonumber
	\\
	&\qquad 
	+ \frac{1}{2\sqrt{\Dt_n}} 
	\pigl(\Sigma(\bold{u}^+) 
	- 2\Sigma(\bold{u}^n) 
	+ \Sigma(\bold{u}^-)\pigr)
	\bigl(\Delta W_n \Dt_n - \Delta Z_n\bigr) 
	\nonumber
	\\
	&\qquad 
	+ \frac{1}{4\Dt_n} 
	\pigl(\Sigma(\Phi^+) 
	- \Sigma(\Phi^-) 
	- \Sigma(\bold{u}^+) 
	+ \Sigma(\bold{u}^-)\pigr)
	\left( \frac{1}{3} 
	(\Delta W_n)^2 - \Dt_n \right) 
	\Delta W_n .
\end{align}
The auxiliary quantities are given by
\begin{align*}
	\bold{u}^{\pm} 
	&:= \bold{u}^n 
	+ b(\bold{u}^n)\Dt_n 
	\pm \Sigma(\bold{u}^n) \sqrt{\Dt_n}, 
	\qquad
	\Phi^{\pm} 
	:= \bold{u}^+ 
	\pm \Sigma(\bold{u}^+)\sqrt{\Dt_n},
\end{align*}
and the correlated random increments 
$\Delta W_n$ and $\Delta Z_n$ are generated via
\begin{align*}
	\Delta W_n = \eta \sqrt{\Dt_n}, 
	\qquad 
	\Delta Z_n = \frac{1}{2} 
	\pigl(\eta + \frac{1}{\sqrt{3}}\xi\pigr) 
	\Dt_n^{\nicefrac{3}{2}},
\end{align*}
where $\eta, \xi \sim N(0,1)$ are independent. 
Both \eqref{eq:SDE-solverGeneral} and 
\eqref{eq:SDE-solver} are applied entrywise, 
i.e.\ to each component of the coefficient 
matrix $\bold{u}$.

If we choose $\Dt_n \sim h^2$ for all $n \geq 0$, 
the temporal discretization error is consistent 
with the expected spatial accuracy of the LDG 
scheme for polynomial degrees $k \leq 2$. 
Moreover, by selecting a sufficiently small 
proportionality constant, the temporal error 
can be rendered negligible compared to the 
spatial discretization error. 
For this reason, we restrict most experiments 
to $k \leq 2$. 
In a few cases where the observed convergence 
rates are inconclusive, we include additional 
tests with $k=3$ to provide further evidence. 
For higher polynomial degrees, one should in 
general employ an SDE solver of correspondingly 
higher order.

\medskip

We approximate the error 
$\sup_{t \in [0, T]}\pigl( 
\E \bigl[ \|u - u_h\|_2^2 \bigr] 
\pigr)^{\frac{1}{2}}$ 
by the estimator
\begin{align*}
	\mathrm{Err}_l^k(T) 
	= \max_{s \in \{t_n\}_{n=0}^N, \, t_N=T}
	\Bigl(
	\frac{1}{M} 
	\sum_{r=1}^{M}  
	\|u(\omega_r, s) - u_h(\omega_r, s)\|_2^2 
	\Bigr)^{\frac{1}{2}},
\end{align*}
where $M$ is the number of realizations. 
Here $l \geq 0$ indexes the mesh resolution, 
with mesh size $h_l = 2^{-(3+l)}$, 
and the superscript $k \geq 0$ denotes the 
polynomial degree of the equal order 
approximation space.

The experimental order of convergence (EOC) 
is estimated by
\begin{align*}
	\mathrm{EOC}^k_l(T) 
	:= 
	\frac{\ln \pigl( 
	\nicefrac{\mathrm{Err}^{k}_{l-1}(T)}
	{\mathrm{Err}^k_{l}(T)} 
	\pigr)}
	{\ln \pigl(
	\nicefrac{h_{l-1}}{h_l}
	\pigr)}.
\end{align*}

In the experiments, we also examine 
how the choice of numerical fluxes influences 
the evolution of the pathwise $L^2_x$ norm 
$\|u_h(t)\|_2$ and the root mean energy, 
i.e.\ the $L^2_{\omega,x}$ norm. 
The latter is approximated by
\begin{align*}
	\|u_h(t)\|_{L^2_{\omega,x}} 
	\approx 
	\Bigl(
	\frac{1}{M} 
	\sum_{r=1}^{M}  
	\|u_h(\omega_r, t)\|_2^2 
	\Bigr)^{\frac{1}{2}}.
\end{align*}

Theoretically, Theorem \ref{thm:pathwiseEst} 
and Theorem \ref{sec4:StabilityTheorem} 
guarantee that the numerical fluxes 
\eqref{eq:firstFlux1}--\eqref{eq:firstFlux2} 
and \eqref{sec4:numericalFluxes} give rise 
to numerical solutions with nonincreasing 
root mean energy for linear and nonlinear 
problems, respectively. 
Furthermore, if the second derivative of 
any stochastic flux is bounded on the 
relevant solution range, then the pathwise 
$L^2_x$ norm is nonincreasing by 
Theorem \ref{sec4:pathwiseStability}, 
provided the numerical fluxes are chosen according to 
\eqref{sec4:pathwiseFlux1_ref}--\eqref{sec4:pathwiseFlux2_ref} 
and the condition 
\eqref{sec4:lowerPenalty_ref} is satisfied.
  
\subsection{Linear problems}
We first consider two instances of the 
linear stochastic continuity equation
\eqref{eq:prob1Strat} to illustrate, 
in the simplest setting, how the
choice of the numerical fluxes $\F_u$ from \eqref{eq:firstFlux1} and
$(\widetilde{\F}_u, \widetilde{\F}_q)$ 
from \eqref{eq:firstFlux2}
affects the dissipation mechanism and 
accuracy of the LDG method.
We then present a two-dimensional 
example showing that the scheme
remains accurate even for very rough spatial correlation
structures, provided they are divergence-free.

Discretizing the linear stochastic continuity equation
\eqref{eq:problem1} or \eqref{sec4:stratonovichSPDE}
with linear stochastic fluxes yields the following
large, sparse, and indefinite matrix-valued system:
\begin{align*}
	M_u d\bold{u}(t)
	&= \bigl(A\bold{u}(t) + B\bold{q}(t)\bigr)\, dt
	+ C\bold{q}(t)\, dW_t, 
	\quad M_q \bold{q}(t)
	= D\bold{u}(t).
\end{align*}
It is therefore natural to eliminate the auxiliary variable
$\bold{q}(t)$. Here $A$ and $B$ are related to the bilinear
forms $a_j(\cdot, \cdot)$ and $b_j(\cdot, \cdot)$ from
\eqref{eq:bilinearForms} in one dimension, while in 
the multidimensional case, 
$B$ is related to \eqref{sec4:semilinear} and $A=0$
since each $\sigma_{\ell}$ is divergence-free. Moreover,
$M_u$ and $M_q$ are block-diagonal 
mass matrices which coincide if $\mathcal{W}^{k, l}
= \mathcal{V}^k \times \mathcal{V}^k$. Hence the
elimination can be carried out efficiently, yielding
\begin{equation}\label{eq:resultingSDE}
	d\bold{u}(t)
	= A^{\star}\bold{u}(t)\, dt
	+ C^{\star}\bold{u}(t)\, dW_t,
\end{equation}
where the effective drift and diffusion matrices are
defined by
\begin{align*}
	A^{\star} :\! &= M_u^{-1}
	\left(A + B M_q^{-1}D\right)\!, \qquad 
	C^{\star} :\! = M_u^{-1} C M_q^{-1}D.
\end{align*}

Numerical experiments indicate that both explicit and
semi-implicit (implicit only in the drift)
integrators suffer from restrictive time-step
constraints, as expected for LDG methods. One
exception is the stochastic $\theta$-method,
especially for $\theta \in [\tfrac{1}{2}, 1]$,
which, when applied to \eqref{eq:resultingSDE}
with $\Dt_n=\Dt$, reads
\begin{align*}
	M^{\star}\bold{u}^{n+1}
	&= \bigl(Id + (1-\theta)\Dt A^{\star}\bigr)
	\bold{u}^n
	+ \Delta W_n C^{\star}\bold{u}^n,
\end{align*}
where $M^{\star}:=Id-\theta \Delta t A^{\star}$. 
However, this method has only strong order
$\nicefrac{1}{2}$ and therefore quickly becomes
the dominant source of error whenever
polynomials of order $k \geq 1$ are used.
Consequently, we use the SRK3/2 method in the
numerical experiments.

\begin{example}[Accuracy test]\label{ex:accuracyTest}
We consider the Cauchy problem
\begin{equation*}
	du + \partial_x \bigl(\sigma(x) u \bigr)\circ dW_t = 0,
	\hspace{0.45cm} u \lvert_{t=0}=\bar{u},
\end{equation*}
with $\sigma(x) = \bar{\sigma} \in \R$ and the $1$-periodic
initial data $\bar{u}(x) = \sin(2\pi x)$ for $x \in [0, 1]$. 
The exact solution is $u(t, x) 
= \bar{u}(x - \bar{\sigma} W_t)$.
Hence $u(t) \in C^{\infty}(\R)$ for every $t \in [0, T]$,
almost surely. Since $\sigma$ is constant, the bilinear form
$a_j(\cdot, \cdot)$ from \eqref{eq:bilinearForms} vanishes for any
$\varphi \in \mathcal{V}^k$, so $\F_u$ does not contribute. This
example therefore isolates the effect of the generalized alternating
fluxes $\widetilde{\F}_u$ and $\widetilde{\F}_q$ on both accuracy and
energy dissipation.

Figure \ref{fig:Ex1polynomials} shows two realizations of the exact
solution together with the numerical approximations for
$k=0,1,2$. Already on the coarse mesh $h=2^{-3}$, the cases
$k=1$ and $k=2$ clearly outperform $k=0$. Figure \ref{fig:meanSquareL2Ex1}
shows the estimated root mean energy ($L^2_{\omega,x}$ norm) for
$h=2^{-4}$ and two batches of $30$ realizations. By
Theorem \ref{thm:L2StabilityEst}, central fluxes conserve the mean
energy, while penalty terms add dissipation. This is exactly what is
observed: the central fluxes
$(\widetilde{\F}_u,\widetilde{\F}_q)=(\avg{u_h},\avg{q_h})$,
denoted $\mathrm{CF}$, match the theoretical energy level, the
alternating pair $(u_h^+,q_h^-)$ approaches the same level as the
number of realizations increases, and the penalized alternating pair
lies below due to the additional dissipation.

By Theorem \ref{thm:pathwiseEst}, the $L^2_x$ norm is preserved
pathwise for central fluxes, but not for other choices of
$(\widetilde{\F}_u,\widetilde{\F}_q)$. Figure \ref{fig:pathwiseEx1}
confirms this: central fluxes preserve $\|u_h(t)\|_2$ exactly,
whereas for $(\widetilde{\F}_u,\widetilde{\F}_q)=(u_h^+,q_h^-)$ the
norm may increase over parts of the time interval, and adding a penalty
term does not restore monotonicity. The computations use the same
parameters as in Figure \ref{fig:meanSquareL2Ex1}.

Table~\ref{tab:Ex1Sec3} reports estimated errors and EOC's for
$k=0,1,2$, using the alternating pair $(u_h^+,q_h^-)$ and the
central pair $(\avg{u_h},\avg{q_h})$. The numerical results indicate
that central fluxes yield $\mathcal{O}(h^{k+1})$ convergence for
even $k$ and $\mathcal{O}(h^k)$ for odd $k$, consistent with
\cite[Thm. 2.2 and Thm. 3.2]{Cockburn:1998ai} and
\cite{Liu:2021aa}. This is further supported by experiments with
$k=3$ (not shown), which give a third-order rate. In contrast,
alternating fluxes yield an $\mathcal{O}(h^{k+1})$ rate for all
$k \geq 0$; while the case $k=1$ is slightly ambiguous, the
corresponding $k=3$ experiments support this conclusion. This agrees
with \cite{LDGErrorEstimates} for deterministic convection-diffusion
in one space dimension.

\begin{figure}[t]
	\includegraphics[width=.67\linewidth]
	{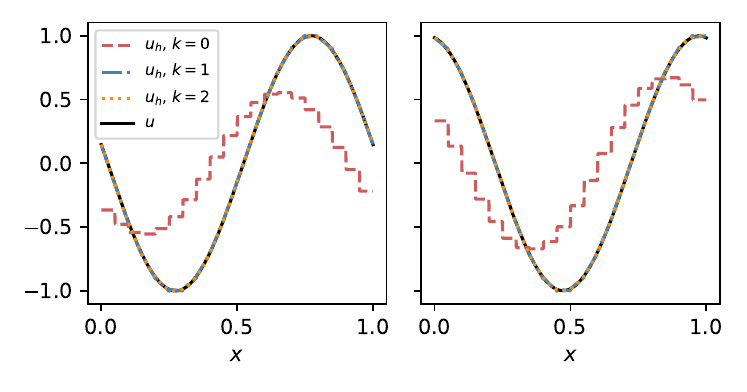}
	\captionsetup{width=.975\linewidth}
	\vspace{-0.6cm}
	\caption{Each plot depicts one realization 
	of the true solution for Example \ref{ex:accuracyTest} and the 
	corresponding numerical approximations computed 
	with $k=0$ (red dashed), $k=1$ (blue dashdotted), 
	$k=2$ (orange dotted) for the alternating flux pair 
	$(\widetilde{\F}_u, \widetilde{\F}_q)=(u_h^+, q_h^-)$. 
	Here $\sigma = \nicefrac{1}{2}=T$, $h=2^{-3}$, 
	and $\Dt = 6.25 \cdot 10^{-6}$.}
	\label{fig:Ex1polynomials}
\end{figure}

\begin{figure}[t]
	\includegraphics[width=.67\linewidth]
	{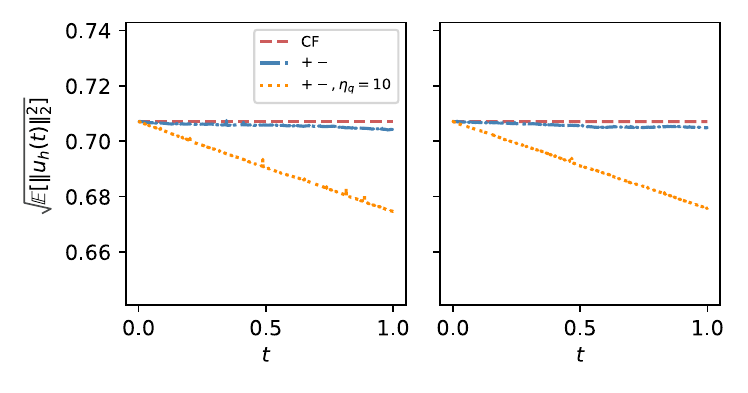}
	\captionsetup{width=.975\linewidth}
	\vspace{-0.6cm}
	\caption{(Evolution of the $L^2_{\omega, x}$ norm)
	The plots show the estimated $L^2_{w, x}$ norm for two batches of
	$M=30$ realizations of Example \ref{ex:accuracyTest} with $k=1$.
	We use $(\widetilde{\F}_u, \widetilde{\F}_q)
	=(\avg{u_h}, \avg{q_h})$, denoted CF,
	$(u_h^+, q_h^-)$, labeled $+-$, and
	$(u_h^+, q_h^-+\eta_q \mathrm{sgn}(\sigma)\llbracket u_h \rrbracket)$,
	denoted $+-, \eta_q =10$. Here $\bar{\sigma} =1$, $h=2^{-4}$,
	and $\Dt \approx 4.88\cdot 10^{-6}$. The behavior is in excellent
	agreement with Theorem \ref{thm:L2StabilityEst}.}
	\label{fig:meanSquareL2Ex1}
\end{figure}

\begin{figure}[t]
	\includegraphics[width=.67\linewidth]
	{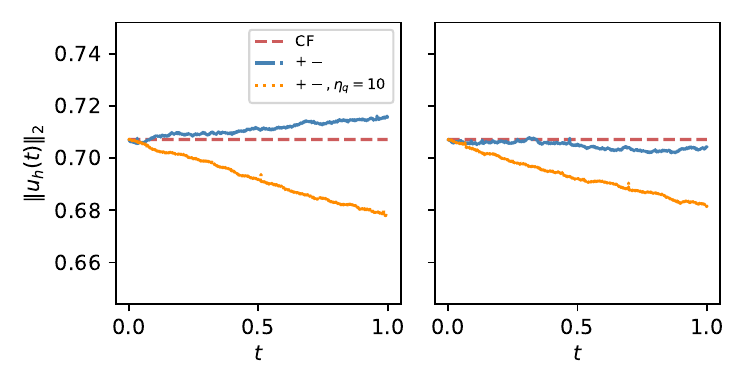}
	\captionsetup{width=.975\linewidth}
	\vspace{-0.6cm}
	\caption{(Evolution of the pathwise $L^2_{x}$ norm)
	The evolution of $\|u_h(t)\|_2$ over $[0, 1]$ for two realizations of
	Example \ref{ex:accuracyTest} and different
	choices of $(\widetilde{\F}_u, \widetilde{\F}_q)$, labeled as in
	Figure \ref{fig:meanSquareL2Ex1}. Here $k=1$,
	$h=2^{-4}$, $\bar{\sigma}=1$ and $\Dt\approx 4.88 \cdot 10^{-6}$.
	In agreement with Theorem \ref{thm:pathwiseEst}, this
	norm may increase for any choice of
	$(\widetilde{\F}_u, \widetilde{\F}_q)$
	which is not the central flux pair.}
	\label{fig:pathwiseEx1}
\end{figure}

\begin{table}[t]
	\centering
	\begin{tabular}{c c cc cc cc}
	\toprule
	& &
	\multicolumn{2}{c}{\textbf{$k=0$}} &
	\multicolumn{2}{c}{\textbf{$k=1$}} &
	\multicolumn{2}{c}{\textbf{$k=2$}}\\
	\cmidrule(lr){3-4}\cmidrule(lr){5-6}\cmidrule(lr){7-8}
	$(\widetilde{\F}_u, \widetilde{\F}_q)$ & $h$
	& $\mathrm{Err}^0\bigl(\tfrac{1}{10}\bigr)$ & $\mathrm{EOC}^0$
	& $\mathrm{Err}^1\bigl(\tfrac{1}{10}\bigr)$ & $\mathrm{EOC}$
	& $\mathrm{Err}^2\bigl(\tfrac{1}{10}\bigr)$ & $\mathrm{EOC}$\\
	\midrule
	\multirow{5}{*}{$\begin{array}{c} +-\end{array}$}
	  & $ 2^{-3}$   & $8.12\cdot 10^{-1}$ &  {---} & $6.69\cdot 10^{-2}$ 
	  & {---} &     $4.56 \cdot 10^{-3}$& {---}  \\
	  & $ 2^{-4}$   & $5.33\cdot 10^{-1}$ & 0.61 & $1.62\cdot 10^{-2}$ & 2.05  
	  &   $5.55 \cdot 10^{-4}$ & 3.04   \\
	  & $ 2^{-5}$   & $2.90 \cdot 10^{-1}$ & 0.88 & $6.43\cdot 10^{-3}$ & 1.33  
	  &  $6.53 \cdot 10^{-5}$  & 3.09   \\
	  & $ 2^{-6}$   & $1.49\cdot 10^{-1}$ & 0.96 & $1.72\cdot 10^{-3}$ & 1.90  
	  &   $1.00 \cdot 10^{-5}$ & 2.71  \\
	  & $ 2^{-7}$   & $7.51\cdot 10^{-2}$ & 0.99 & $3.27\cdot 10^{-4}$ & 2.40   
	  &  $1.02 \cdot 10^{-6}$ & 3.29   \\
	\midrule
	\multirow{5}{*}{$\begin{array}{c}\mathrm{CF}\end{array}$}
	  & $ 2^{-3}$ 	& $2.01 \cdot 10^{-1}$  & {---}    &   $6.57\cdot 10^{-2}$     
	  & \multicolumn{1}{c}{---}  & $1.30 \cdot 10^{-3}$ & {---}\\
	  & $ 2^{-4}$ & $8.61 \cdot 10^{-2}$  &  1.22	 &   $3.22\cdot 10^{-2}$     
	  & 1.03     & $1.41 \cdot 10^{-4}$ & 3.20\\
	  & $ 2^{-5}$ & $4.09 \cdot 10^{-2}$  & 1.07    &   $1.62\cdot 10^{-2}$     
	  & 0.99    & $1.71 \cdot 10^{-5}$ & 3.04 \\
	  & $ 2^{-6}$ & $2.01\cdot 10^{-2}$   & 1.02   	 &   $8.24 \cdot 10^{-3}$	 
	  & 0.98    & $2.12 \cdot 10^{-6}$ & 3.01 \\
	  & $ 2^{-7}$ & $1.00\cdot 10^{-2}$   & 1.01   	&    $4.12\cdot 10^{-3}$     
	  & 1.00    & $2.61 \cdot 10^{-7}$ & 3.02 \\
	\bottomrule
	\end{tabular}
	\captionsetup{width=.975\linewidth}
	\vspace{0.5cm}
	\caption{Estimated errors and $\mathrm{EOC}$'s for 
	Example \ref{ex:accuracyTest}, where $k=0,1,2$. Here
	$(\widetilde{\F}_u, \widetilde{\F}_q)$
	is selected as the alternating pair $(u_h^+, q_h^-)$
	and the central pair $(\avg{u_h}, \avg{q_h})$, respectively.
	Moreover, $\bar{\sigma}=1$ and $T=\nicefrac{1}{10}$ are kept
	fixed. The alternating pair converges with rate
	$\mathcal{O}(h^{k+1})$, while the central flux pair exhibits the
	error rate $\mathcal{O}(h^k)$ for $k$ odd and
	$\mathcal{O}(h^{k+1})$ for $k$ even; this is consistent with
	deterministic theory.}
	\label{tab:Ex1Sec3}
\end{table}
\end{example}

\begin{example}[Nonconstant $\sigma$]\label{ex:Ex2}
We next consider the compactly supported initial data
\begin{align*}
	\bar{u}(x) &
	=
	\begin{cases}
		\sin(2\pi x)e^{-\frac{1}{1-x^2}}, & |x| < 1, \\
		0, & 1 \leq |x|,
	\end{cases}
\end{align*}
with noise intensity $\sigma(x)=\bar{\sigma}x$, 
$\bar{\sigma}\in\R$. The exact solution 
to \eqref{eq:problem1} is
$u(t,x)=\bar{u}\bigl(xe^{-\bar{\sigma}W_t}\bigr)
e^{-\bar{\sigma}W_t}$. In contrast to 
Example \ref{ex:accuracyTest}, 
the flux $\F_u$ now contributes. 
The mean energy identity from
Theorem \ref{thm:L2StabilityEst} implies that 
\begin{align}\label{ex2:upper}
	\E \pigl[ \|u_h(t)\|_2^2 \pigl]
	\leq \|u_h(0)\|_2^2
	e^{\frac{1}{2}\bar{\sigma}^2 t},
\end{align}
while pathwise we have no theoretical 
control on the energy for any flux choice. 
Figure \ref{fig:energyEx2}
confirms this: even when all three numerical fluxes are central
(labeled $\mathrm{CF},\gamma=0.0$), the pathwise norm is not
preserved. Here $\mathrm{CF}$ and $+-$ refer to the choice of
$(\widetilde{\F}_u,\widetilde{\F}_q)$ as in
Figure \ref{fig:meanSquareL2Ex1}, modulo the penalty term involving
$\eta_q$ in \eqref{eq:firstFlux2}, while $\gamma \in [0,1]$
denotes the interpolation parameter in \eqref{eq:firstFlux1}. In
particular, $\gamma=0$ gives the central flux $\avg{u_h}$ in the
first two terms of $\F_u$, whereas for $\gamma=1$ 
they reduce to the downwind
flux $\widecheck{\F}\pigl(u_h;(\sigma^2)'\pigl)$. The choice of
$\gamma$ has little influence on the dissipation level, as seen by
comparing the red dashed and blue dashdotted curves in
Figures \ref{fig:energyEx2} and \ref{fig:meanEx2}. By contrast, the
penalty terms, namely $\eta_q>0$ in $\widetilde{\F}_q$ and
$\widetilde{\gamma}>0$ in $\F_u$, have a more pronounced effect.
Figure \ref{fig:meanEx2} suggests that $\eta_q>0$ is dominant; 
compare the blue dashdotted, orange dotted,
and dark solid curves, especially at later times. The solid smooth
curve is the upper bound
$\|u_h(0)\|_2e^{\frac{1}{4}\bar{\sigma}^2 t}$ from
\eqref{ex2:upper}. It is respected at later times, while the
apparent violation on $[0,\tfrac{1}{10}]$ is explained by the fact
that the root mean energy is estimated from only $30$ realizations.

Figure \ref{fig:solEx2} compares the exact solution, for two
realizations, with numerical approximations computed with
$k=0,1,2$ and $h=2^{-3}$. On this coarse mesh, the cases $k=0$ and
$k=1$ exhibit visible oscillations near local extrema, while the
$k=2$ approximation resolves the smooth profile accurately. The
right panel shows how the flux choice affects the $k=1$
approximation: increasing $\eta_q$ suppresses, though does not
eliminate, the over- and undershoots, and tuning
$\widetilde{\gamma}$ in \eqref{eq:firstFlux1} can further improve
the accuracy.

\begin{figure}[t]
	\includegraphics[width=.67\linewidth]
	{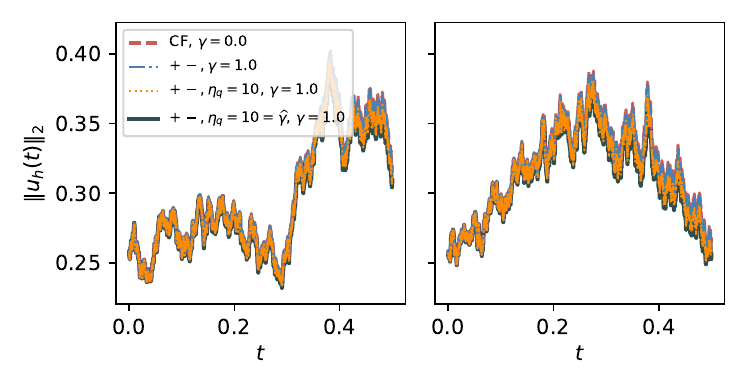}
	\captionsetup{width=.975\linewidth}
	\vspace{-0.6cm}
	\caption{(Evolution of the pathwise $L^2_x$ norm)
	The plots show $\|u_h(t)\|_2$ for two realizations of
	Example \ref{ex:Ex2} and various choices of
	$\F_u$ and $(\widetilde{\F}_u,\widetilde{\F}_q)$.
	Here $\gamma \in [0,1]$ is the interpolation parameter in
	\eqref{eq:firstFlux1}, $\widetilde{\gamma}\geq 0$ is the penalty
	parameter, and $(\widetilde{\F}_u,\widetilde{\F}_q)$ is labeled as
	in Figure \ref{fig:meanSquareL2Ex1}. Moreover,
	$h=2^{-4}$, $\bar{\sigma}=1$, and
	$\Dt \approx 9.77 \cdot 10^{-6}$. The differences between the
	approximations are almost invisible and, since no theoretical
	pathwise $L_x^2$ control is available, the norm may increase.}
	\label{fig:energyEx2}
\end{figure}

\begin{figure}[t]
	\includegraphics[width=.67\linewidth]
	{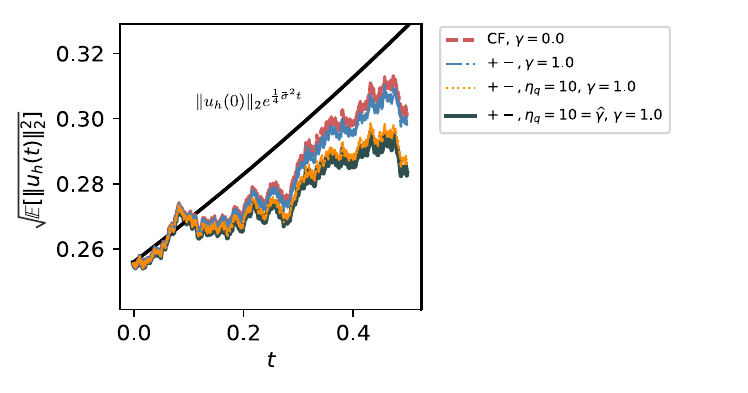}
	\captionsetup{width=.975\linewidth}
	\vspace{-0.6cm}
	\caption{(Evolution of the $L^2_{\omega, x}$ norm)
	The plot shows the estimated $L^2_{\omega,x}$ norm for
	Example \ref{ex:Ex2}, based on $30$ realizations and different
	choices of $\F_u$ and $(\widetilde{\F}_u,\widetilde{\F}_q)$.
	Here $\gamma$ and $\widetilde{\gamma}$ are the tunable parameters
	in \eqref{eq:firstFlux1}, while
	$(\widetilde{\F}_u,\widetilde{\F}_q)$ is labeled as in
	Figure \ref{fig:meanSquareL2Ex1}. Moreover, $h=2^{-4}$,
	$\bar{\sigma}=1$, and $\Dt\approx 9.77 \cdot 10^{-6}$. The
	estimated norm respects the theoretical result 
	\eqref{ex2:upper} except at early times,
	due to the small sample size, and the influence of the flux
	parameters is consistent with the mean energy 
	balance \eqref{eq:stabilityEst}.}
	\label{fig:meanEx2}
\end{figure}

\begin{figure}[t]
	\includegraphics[width=.67\linewidth]
	{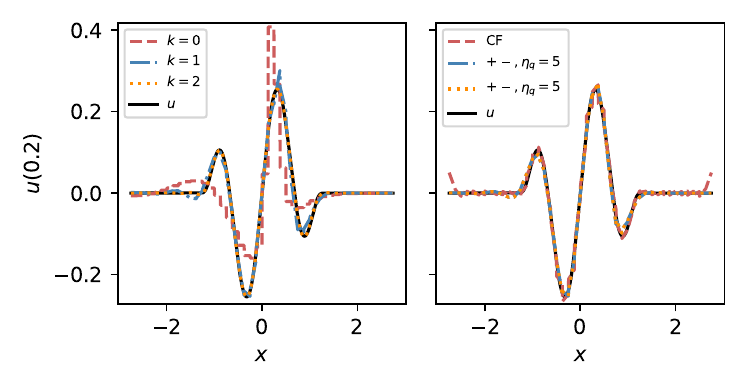}
	\captionsetup{width=.975\linewidth}
	\vspace{-0.6cm}
	\caption{The left plot compares the exact solution (solid black),
	for a single realization of Example \ref{ex:Ex2}, to 
	approximations computed with different polynomial degrees:
	$k=0$ (red dashed), $k=1$ (blue dashdotted), and
	$k=2$ (orange dotted), where $\F_u = \avg{u_h}$ and
	$(\widetilde{\F}_u,\widetilde{\F}_q) = (u_h^+, q_h^-)$.
	The right plot compares the exact solution (same realization) to
	different approximations computed with $k=1$ and different
	fluxes, denoted as in Figure \ref{fig:meanSquareL2Ex1}, where
	$\F_u = \avg{u_h}$ is fixed. Moreover, $h=2^{-3}$,
	$T=\nicefrac{1}{2}$, $\bar{\sigma}=1$, and
	$\Dt =6.25\cdot 10^{-6}$.}
	\label{fig:solEx2}
\end{figure}
\end{example}

\begin{example}[Irregular div-free $\sigma$]\label{ex:lowReg}
This example shows that the LDG scheme
\eqref{sec4:ldg1}--\eqref{sec4:ldg3} remains effective under
spatially rough noise. We consider a single noise mode whose
regularity is controlled by $\beta$, ranging from
$[H^2(\R^2)]^2$ down to merely $[L^p(\R^2)]^2$. The latter regime
lies outside the deterministic DiPerna--Lions theory for bounded
solutions and is closer in spirit to the regularization by noise
results of \cite{Flandoli:2010yq}. Still, the mean energy stability
in Theorem \ref{sec4:StabilityTheorem} only requires $\sigma$ to be
divergence-free with well-defined normal traces and in $[L^2(\R^2)]^2$.

We study
\begin{equation*}
	du + \div_x\bigl(\sigma(x, y) u \bigr) \circ dW_t = 0,
	\hspace{0.55cm}
	u\lvert_{t=0}
	= \sin(2\pi x)\sin(2 \pi y)\psi^{\tfrac{1}{8}}(|\bold{x}|),
\end{equation*}
on $[0, \frac{1}{2}] \times [-\tfrac{1}{2}, \tfrac{3}{2}]^2$,
where $\psi$ is supported on the disk centered at
$(\tfrac{1}{2}, \tfrac{1}{2})$ with radius $\tfrac{1}{2}$:
\begin{align*}
	\psi(r) &=
	\begin{cases}
		e^{1 - \frac{1}{1-4r^2}},
		& r \in [0, \tfrac{1}{2}), \\
		0, & r \geq \tfrac{1}{2},
	\end{cases}
	\qquad 
	|\bold{x}| =
	\sqrt{(x-\tfrac{1}{2})^2+(y-\tfrac{1}{2})^2}.
\end{align*}
For $\beta > -1$ we set
\begin{equation*}
	\sigma(x, y) =
	\pigl(-\partial_y |\bold{x}|^{\beta+1}, \,
	\partial_x |\bold{x}|^{\beta+1} \pigr)
	= J \nabla \mathcal{H}(|\bold{x}|),
\end{equation*}
where
\begin{align*}
	J = \begin{pmatrix}
	0 & -1 \\
	1 & 0
	\end{pmatrix},
	\qquad
	\mathcal{H}(|\bold{x}|) = |\bold{x}|^{\beta+1}.
\end{align*}
Thus $\sigma$ is a Hamiltonian vector field with radial
Hamiltonian and, by construction,
$\mathrm{div}_x \sigma(x, y) = 0$. Since $\sigma$ is singular at
$(\tfrac{1}{2}, \tfrac{1}{2})$, its integrability depends on
$\beta$; it is only locally integrable, though it can be made
integrable through multiplication with a cut-off. In particular,
$\beta > -1$ implies
$\sigma \in [L^2_{\mathrm{loc}}(\R^2)]^2$, while
\begin{align*}
	\beta > 0 \implies
	\sigma \in [H^1_{\mathrm{loc}}(\R^2)]^2,
	\qquad
	\beta > 1 \implies
	\sigma \in [H^2_{\mathrm{loc}}(\R^2)]^2.
\end{align*}

Figure \ref{fig:Ex5H1RegularSigma} shows the numerical
solution for $\beta=\frac34$, computed on a quadrilateral
mesh with $h=2^{-3}$, $k=2$, and
$\Dt \approx 3.13 \cdot 10^{-6}$, at the times
$t=\tfrac{m}{10}$ for $m=0,\ldots,5$. Even for this rough
divergence-free coefficient, the scheme transports the
initial four-lobed profile in a stable and coherent way.
The motion is mainly rotational, with mild shear, consistent
with the Hamiltonian structure, while localization and
amplitude are largely preserved.

Figure \ref{fig:Ex5LowRegularSigma} shows that the same
qualitative behavior persists for $\beta=-\frac12$.
Although the lower regularity produces stronger deformation
near $(\tfrac{1}{2},\tfrac{1}{2})$, the sign structure is
still transported in a stable manner. In this case,
$\sigma \in [L^p(\R^2)]^2$ only for $p<4$, and $\sigma$
blows up at $(\tfrac{1}{2},\tfrac{1}{2})$. Nevertheless,
the numerical solution remains bounded. This is possible
even though some mesh interfaces lie on
$x=\tfrac{1}{2}$ or $y=\tfrac{1}{2}$, since the surface
quadrature rule used here does not sample the singular point
$(\tfrac{1}{2},\tfrac{1}{2})$. Thus no singular value of
$\sigma$ is evaluated. More generally, if a quadrature rule
would include this point, one can avoid the singular
evaluation by a slight perturbation of the quadrature nodes.

Notice also that $\beta=-\tfrac12$ gives, near the singular
point, only $\sigma\in H^s(\R^2)$ for $s<\tfrac12$.
This lies outside Assumption \ref{sec4:regularityg}: the
trace theorem used in the analysis does not apply, and faces
passing through $(\tfrac12,\tfrac12)$ may fail to have an
$L^2$ normal trace. The experiment therefore suggests that,
with a suitably modified quadrature, the schemes may still
be useful for problems rougher than those covered by the analysis.

In both contour plots, and in all subsequent contour plots, we also
display the minimal and maximal values,
$u_{h,\min}$ and $u_{h,\max}$, attained by the LDG approximation at
the shown times, to indicate the over- and undershoots. For this
example, the unique solution lies in approximately the interval 
$[-\frac{9}{10}, \frac{9}{10}]$, whereas the LDG
approximation attains values outside this invariant set. This is
consistent with the non-monotone 
character of the LDG scheme.

Figure \ref{fig:Ex5H1RegularSigmaSurface} shows a surface plot
(left) of the corresponding solution at $t=\tfrac{1}{2}$,
computed with $k=2$, $h=2^{-3}$, and
$\Dt=6.25 \cdot 10^{-6}$. The right plot shows the pathwise energy
evolution for LDG approximations computed with $k=1$, $h=2^{-3}$,
and $\Dt \approx 5.21 \cdot 10^{-6}$, using the fluxes from
\eqref{sec4:numericalFluxes} with $\eta_q=0$ (denoted CF) and with
$\eta_q=10$. For $\beta=\frac34$, the numerical solution remains
well organized under the stochastic transport: the profile is
rotated and mildly distorted, while the pathwise $L^2$ norm is
conserved for $\eta_q=0$ and decreases for $\eta_q=10$. This agrees
with Remark~\ref{sec4:remarkLinear} and
Theorem \ref{sec4:pathwiseStability}, since
\eqref{sec4:numericalFluxes} and
\eqref{sec4:pathwiseFlux1_ref}--\eqref{sec4:pathwiseFlux2_ref}
coincide when $g_{\ell}=u$, except for the penalty terms in
$\F_{g_{\ell}'(u)q_{\ell}}^{e}$.

\begin{figure}[t]
	\includegraphics[width=.67\linewidth]
	{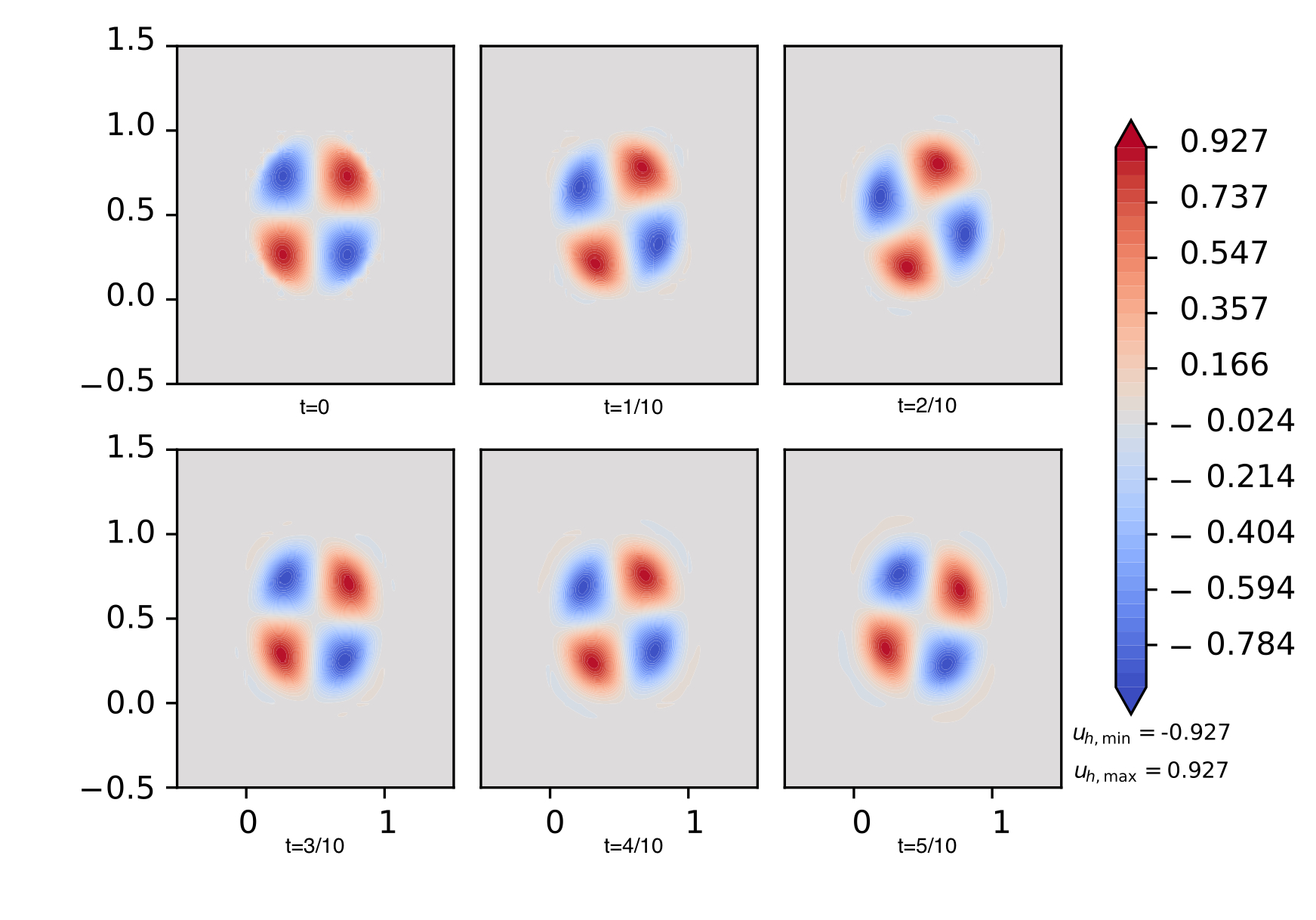}
	\vspace{-0.7cm}
	\captionsetup{width=.975\linewidth}
	\caption{($H^1$-regular $\sigma$) Evolution of one
	realization of the numerical approximation of
	Example \ref{ex:lowReg}, computed on a quadrilateral
	mesh with $k=2$, $h=2^{-3}$, and
	$\Dt =3.13 \cdot 10^{-6}$ for $\beta = \tfrac{3}{4}$.
	This corresponds to $\sigma \in [H^1(\R^2)]^2$.
	The approximation is displayed at
	$t=\tfrac{m}{10}$ for $m=0,\ldots,5$. Despite the
	roughness of the noise, the scheme transports the initial
	four-lobed profile in a stable and coherent way. 
	The evolution is predominantly rotational about 
	$(\tfrac12,\tfrac12)$, with mild radius-dependent twisting due to 
	the dependence of the angular velocity on $|\bold x|$.}
	\label{fig:Ex5H1RegularSigma}
\end{figure}

\begin{figure}[t]
	\includegraphics[width=.67\linewidth]
	{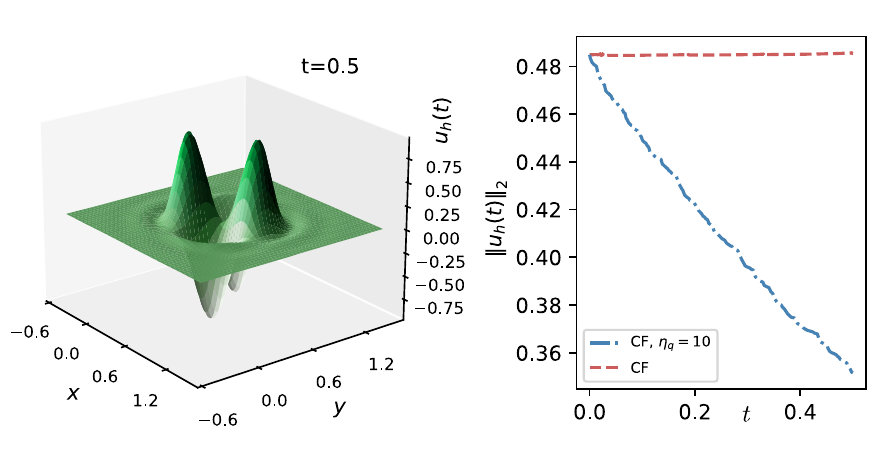}
	\vspace{-0.6cm}
	\captionsetup{width=.975\linewidth}
	\caption{($H^1$-regular $\sigma$) Surface plot (left)
	of the numerical approximation of
	Example \ref{ex:lowReg} computed on a quadrilateral mesh, with
	$k=2$, $h=2^{-3}$, and $\Dt =6.25 \cdot 10^{-6}$ for
	$\beta = \tfrac{3}{4}$, corresponding to
	$\sigma \in [H^1(\R^2)]^2$. The right plot displays the
	pathwise evolution of the $L^2_x$ norm over
	$[0, \tfrac{1}{2}]$ for the flux \eqref{sec4:numericalFluxes}
	and the two penalty values $\eta_q=0$ (red dashed) and
	$\eta_q=10$ (blue dashdotted), both computed with $k=1$,
	$h=2^{-3}$, and $\Dt \approx 5.21 \cdot 10^{-6}$. The
	numerical solution remains well organized under the stochastic
	transport: the initial four-lobed profile is rotated and mildly
	distorted, while the pathwise $L^2$ norm is conserved for
	$\eta_q=0$ and decreases for $\eta_q=10$, consistent with
	Remark~\ref{sec4:remarkLinear} and
	Theorem \ref{sec4:pathwiseStability}.}
	\label{fig:Ex5H1RegularSigmaSurface}
\end{figure}

\begin{figure}[t]
	\includegraphics[width=.67\linewidth]
	{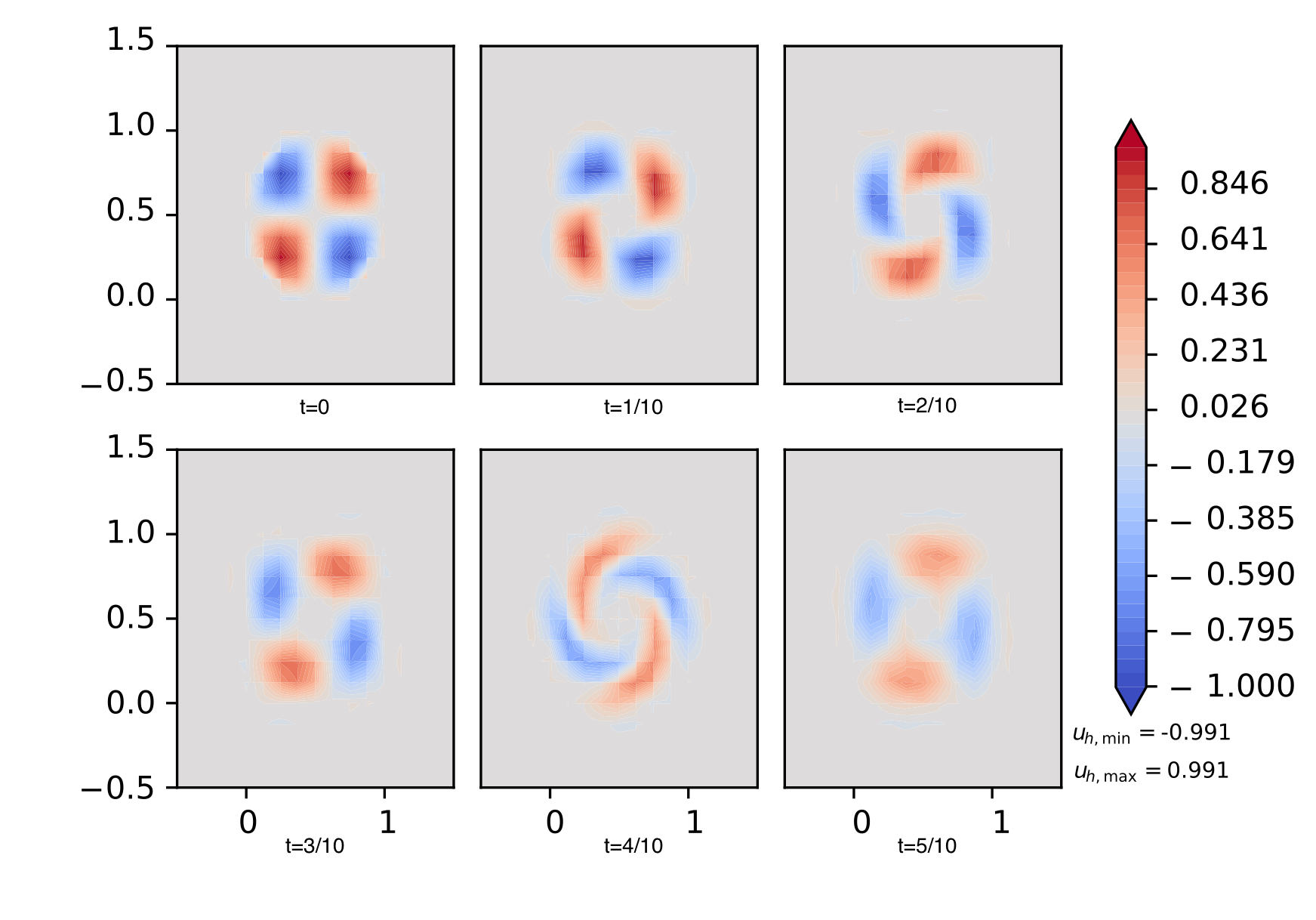}
	\vspace{-0.7cm}
	\captionsetup{width=.975\linewidth}
	\caption{($L^2$-regular $\sigma$) One realization of the
	numerical approximation of
	Example \ref{ex:lowReg} computed on a quadrilateral mesh, with
	$k=1$, $\eta_q = 5$, $h=2^{-3}$, and
	$\Dt= 8.68 \cdot 10^{-6}$ for $\beta = -\tfrac{1}{2}$,
	corresponding to $\sigma \in [L^p(\R^2)]^2$ for $p <4$.
	The approximation is shown at $t=\tfrac{m}{10}$ for
	$m=0,\ldots,5$. Despite the roughness of the noise, which causes
	some deformation near $(\tfrac{1}{2}, \tfrac{1}{2})$, the scheme
	still transports the flow coherently.}
	\label{fig:Ex5LowRegularSigma}
\end{figure}
\end{example}

\subsection{Nonlinear problems}
We next examine the LDG method
\eqref{sec4:ldg1}--\eqref{sec4:ldg3} on nonlinear problems.
We begin with two one-dimensional examples with $L=1$:
Burgers' equation with a stochastic convex flux, and
\eqref{sec4:stratonovichSPDE} with the nonconvex flux
$g(u)=u\,\sin(2\pi u)+u$.
The former was studied numerically in
\cite[Ex.~3.1]{Hoel:2018aa}, where the Wiener paths were
approximated by piecewise linear interpolants and the resulting
problem was discretized by a deterministic first-order finite
difference method.
Our results are consistent with those findings and indicate a
regularizing effect of the noise for convex fluxes.
The second example is motivated by \cite[Ex.~4.9]{Holden:2010fk},
where a shock forms, suggesting that this regularization may fail
for nonconvex stochastic fluxes.
Both one-dimensional examples also show that the penalty
parameters strongly affect the accuracy.
As a final nonlinear example, we consider the inviscid
two-dimensional Buckley--Leverett equation with gravity in the
$y$-direction and stochastic fluxes, corresponding to $L=2$ in
\eqref{sec4:stratonovichSPDE}.

For the one-dimensional problems, with fixed polynomial degree
$k$ and spatial resolution $h$, we require
\begin{align}\label{eq:CFL}
	\Dt \leq \min \left\{
	\frac{\kappa_{\mathrm{c}} h}{(2k+1)\lambda_{\max}},
	\frac{\kappa_{\mathrm{d}}h^2}{(2k+1)^2a_{\max}}
	\right\},
\end{align}
where $\kappa_{\mathrm{c}}$ and $\kappa_{\mathrm{d}}$ are
problem-dependent parameters associated with the convection and
diffusion parts of the It\^o SPDE \eqref{sec4:Ito}, and
$\lambda_{\max}$ and $a_{\max}$ denote the maximal wave speed and
diffusivity. This is motivated by the 
deterministic case, see for instance
\cite{Hesthaven:2008aa}.
Assuming each $\sigma_{\ell}$ is 
divergence-free, hence constant in
one dimension, we write this constant 
as $\bar{\sigma}_{\ell}$ and set
\begin{align*}
	\lambda_{\max}
	:=\sum_{\ell \in L}
	\bar{\sigma}_{\ell}\max_{u}|g_{\ell}'(u)|,
	\qquad
	a_{\max}
	:= \frac{1}{2}\sum_{\ell \in L}
	\bar{\sigma}_{\ell}^2\max_{u}| g_{\ell}'(u)|^2.
\end{align*}
The CFL condition \eqref{eq:CFL} extends to two dimensions in the
usual way, with $a_{\max}$ denoting the entrywise maximum of the
diffusion matrix associated with the correction term.

\begin{example}[Burgers' equation with stochastic flux]\label{ex:ex4}
Consider the nonlinear Cauchy problem
\begin{equation*}
	du + \partial_x \left(\frac{u^2}{2} \right)\circ dW_t =0,
	\qquad
	\bar{u} = 1_{[\frac{1}{4}, \frac{3}{4}]}.
\end{equation*}
Figures \ref{fig:ex4Time} and \ref{fig:ex4TimeCF} show single-realization
snapshots on $[0,\tfrac{1}{2}]$ computed with the LDG scheme
\eqref{sec4:ldg1}--\eqref{sec4:ldg2}. They illustrate the regularizing
effect of the noise: the initially discontinuous profile becomes
progressively smoother at times $t=\tfrac{m}{10}$ for $m=1,\ldots,5$.
At each time, the black solid curve is a reference solution computed
with $h=2.5 \cdot 10^{-3}$, $k=0$, and
$\Dt \approx 1.25 \cdot 10^{-7}$, while the other curves are
approximations with $h=2^{-5}$ and the same time step.
In Figure \ref{fig:ex4TimeCF}, the red dashed curve uses the flux
$(\F_{g'(u)}, \F_{g'(u)q})$ from \eqref{sec4:numericalFluxes} with
$\eta_q=0$. In Figure \ref{fig:ex4Time}, the blue dashdotted curve uses
the same flux with $\eta_q=2.5$, while the orange dotted curve uses the
fluxes \eqref{sec4:pathwiseFlux1_ref}--\eqref{sec4:pathwiseFlux2_ref},
denoted $\F_{g(u)} = \tfrac{[G(u_h)]}{[u_h]}$, again with
$\eta_q=2.5$. All approximations use polynomial degree $k=1$ and
$\eta_u=0$, so the auxiliary variables can be eliminated.
Comparing Figures \ref{fig:ex4Time} and \ref{fig:ex4TimeCF} shows that
the penalty term has a strong effect on the resolution of sharp
transitions and on the overall accuracy. Without penalization the
oscillations are pronounced, whereas for $\eta_q=2.5$ only mild
over- and undershoots remain. Such oscillations are well known for
high-order LDG-type discretizations of deterministic transport problems
with discontinuous or nearly discontinuous solutions
\cite{Cockburn:1999ud, Hesthaven:2008aa}, where they are typically
controlled by limiters or related post-processing, see e.g.
\cite{Chen:2017ab, Cockburn:1999ud, Hiltebrand:2014aa,
Liu:1996aa}. Developing analogous techniques in the present
stochastic setting lies beyond the scope of this work.

Figure \ref{fig:ex4Pathwise} shows the evolution of the pathwise
$L^2_x$ norm over $[0,\tfrac{1}{2}]$ for the same flux families,
with the penalty values indicated in the labels. The results are
fully consistent with our stability theory. For the fluxes
\eqref{sec4:numericalFluxes} with $\eta_q=0$, no pathwise
monotonicity result is available, and the red dashed curve indeed
shows pronounced jumps and oscillations. Adding a penalty term in
the same flux family improves the behavior substantially, as seen
from the blue dashdotted curve, but this case is still not covered
by Theorem \ref{sec4:pathwiseStability}, so monotonicity is neither
predicted nor observed. In contrast, if one chooses
$\F_{g(u)}=\frac{\llbracket G(u_h)\rrbracket}
{\llbracket u_h \rrbracket}$ together with
$\F_{g'(u)q}$ from \eqref{sec4:pathwiseFlux2_ref}, then
Theorem \ref{sec4:pathwiseStability} yields pathwise monotonicity
of $\|u_h(t)\|_2$, provided \eqref{sec4:lowerPenalty_ref} holds.
Since $g''=1$ here, this reduces to
$\eta_q \geq \frac{1}{12}$. Accordingly, the orange dotted curve
satisfies the bound and decreases monotonically, whereas the dark
solid curve with $\eta_q=0.07$ lies outside the theorem and
occasionally increases. Larger penalty values produce more
dissipation, and for a fixed penalty value the two flux families
\eqref{sec4:numericalFluxes} and
\eqref{sec4:pathwiseFlux1_ref}--\eqref{sec4:pathwiseFlux2_ref}
yield pathwise comparable dissipation, as seen by the blue
dashdotted and orange dotted curves. 
The mean energy computations in
Figure \ref{fig:ex4Mean} are likewise consistent with
Theorem \ref{sec4:StabilityTheorem}: the cases with $\eta_q>0$
have nonincreasing $L^2_{\omega,x}$ norm, while
$\eta_q=0$ for the fluxes \eqref{sec4:numericalFluxes}
preserves the root mean energy. In this plot the $L^2_{\omega, x}$ norms are estimated based on $200$ realizations with $k=0$, and, again, the size of $\eta_q$ has a pronounced effect on the total
amount of dissipation.

\begin{figure}[t]
	\includegraphics[width=.67\linewidth]
	{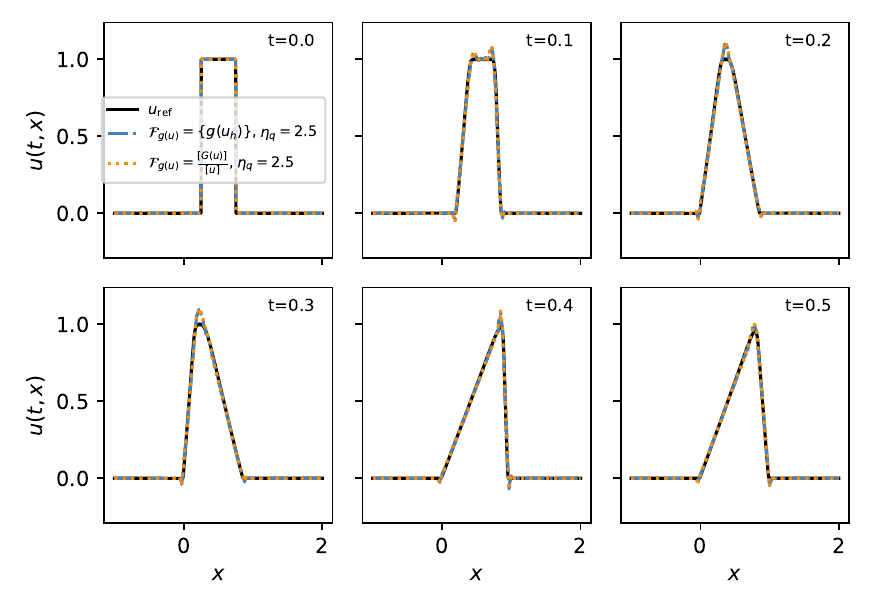}
	\vspace{-0.6cm}
	\captionsetup{width=.975\linewidth}
	\caption{(Burgers' with stochastic flux) 
	Snapshots of several numerical
	approximations for one realization of Example \ref{ex:ex4} at
	$t=\tfrac{m}{10}$ for $m=0,\ldots,5$. The black solid line is a
	reference solution with $h=2.5 \cdot 10^{-3}$, $k=0$, and
	$\Dt \approx 1.25 \cdot 10^{-7}$. The blue dashdotted and orange
	dotted curves correspond to the fluxes
	\eqref{sec4:numericalFluxes} and
	\eqref{sec4:pathwiseFlux1_ref}--\eqref{sec4:pathwiseFlux2_ref},
	respectively, both with $\eta_q=2.5$, computed with $k=1$,
	$h=2^{-5}$, and $\Dt \approx 1.25 \cdot 10^{-7}$. The two
	approximations perform similarly and exhibit only mild over- and
	undershoots near sharp transitions.}
	\label{fig:ex4Time}
\end{figure}

\begin{figure}[t]
	\includegraphics[width=.67\linewidth]
	{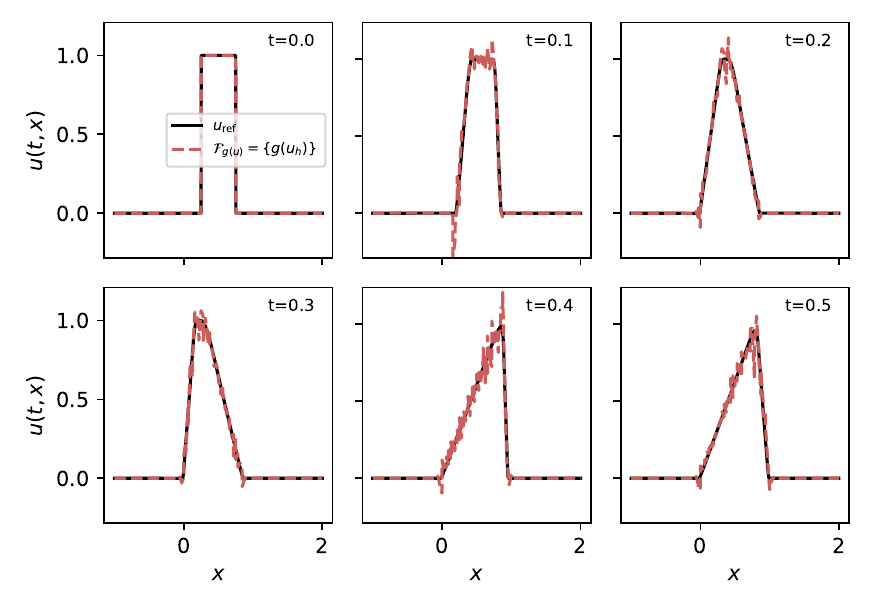}
	\vspace{-0.6cm}
	\captionsetup{width=.975\linewidth}
	\caption{(Burgers' with stochastic flux and central flux) 
	LDG approximation of
	Example \ref{ex:ex4} with central fluxes, corresponding to
	$\eta_q=0$ in \eqref{sec4:numericalFluxes}, computed with
	$h=2^{-5}$ and $k=1$ (red dashed), compared with the numerical
	reference solution obtained with $h=2.5 \cdot 10^{-3}$ and $k=0$
	(black solid). Snapshots are shown at $t=\tfrac{m}{10}$ for
	$m=0,\ldots,5$, with $\Dt \approx 1.25\cdot 10^{-7}$ in both
	computations. The central flux captures the main regularizing
	effect of the noise, but produces pronounced spurious oscillations
	near the sharp transition regions, as is typical for high-order
	DG/LDG discretizations without limiting.}
	\label{fig:ex4TimeCF}
\end{figure}

\begin{figure}[t]
	\includegraphics[width=.67\linewidth]
	{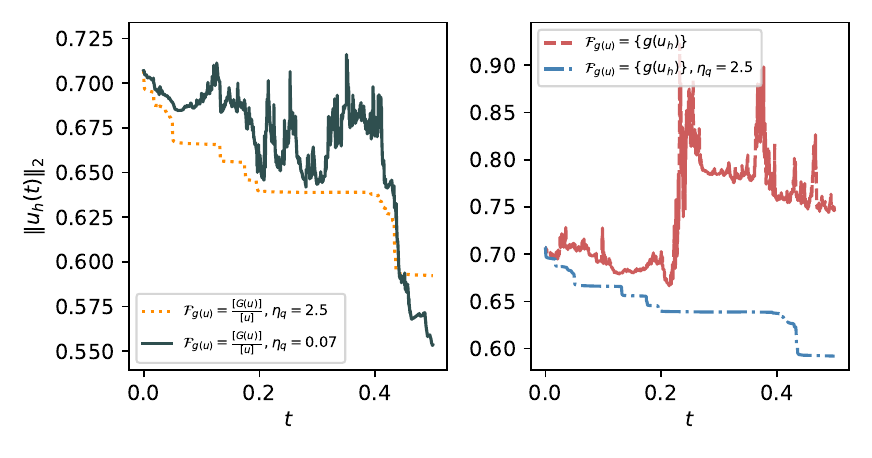}
	\vspace{-0.6cm}
	\captionsetup{width=.975\linewidth}
	\caption{(Pathwise evolution of the $L^2_x$ norm) Evolution of
	$\|u_h(t)\|_2$ for one realization of Example \ref{ex:ex4} over
	$[0,\tfrac{1}{2}]$ for various numerical fluxes. In the left plot,
	$(\F_{g'(u)q}, \F_{g(u)})$ is chosen according to
	\eqref{sec4:pathwiseFlux1_ref}--\eqref{sec4:pathwiseFlux2_ref} with
	$\eta_q=2.5$ (orange dotted) and $\eta_q=0.07$ (dark solid). In the
	right plot, the flux pair is based on \eqref{sec4:numericalFluxes}
	with $\eta_q=0$ (red dashed) and $\eta_q=2.5$ (blue dashdotted).
	Here $h=2^{-4}$, $\Dt=6.25 \cdot 10^{-6}$, and $k=1$. Consistent with
	Theorem \ref{sec4:pathwiseStability}, the orange dotted curve, which
	satisfies \eqref{sec4:lowerPenalty_ref}, decreases monotonically.
	For the red, blue, and dark curves no pathwise monotonicity result
	applies, and the norm may vary non-monotonically in time.}
	\label{fig:ex4Pathwise}
\end{figure}

\begin{figure}[t]
	\includegraphics[width=.67\linewidth]
	{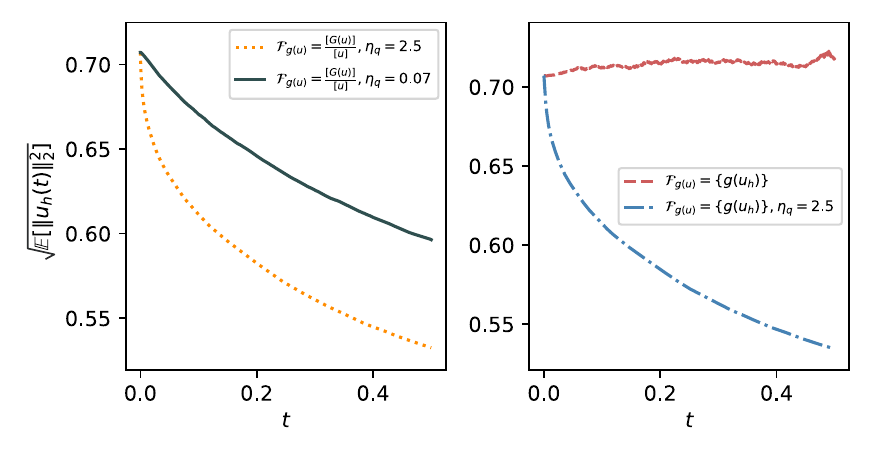}
	\vspace{-0.6cm}
	\captionsetup{width=.975\linewidth}
	\caption{(Evolution of the $L_{\omega, x}^2$ norm) Estimated root
	mean energy for Example \ref{ex:ex4}, based on $200$ realizations,
	over $[0,\tfrac{1}{2}]$ for different numerical fluxes. Here
	$k=0$, $h=2^{-5}$, and $\Dt \approx 1.22 \cdot 10^{-4}$ are fixed.
	In the left plot the fluxes are selected according to
	\eqref{sec4:pathwiseFlux1_ref}--\eqref{sec4:pathwiseFlux2_ref},
	with $\eta_q=2.5$ (orange dotted) and $0.07$ (dark solid), while
	in the right plot we use \eqref{sec4:numericalFluxes} with
	$\eta_q=0$ (red dashed) and $\eta_q=2.5$ (blue dashdotted).
	Consistent with Theorem \ref{sec4:StabilityTheorem} and its
	derivation, the red dashed curve is approximately constant,
	indicating preservation of the root mean energy, whereas
	$\eta_q>0$ introduces additional dissipation and forces the norm
	to decrease monotonically (provided \eqref{sec4:lowerPenalty_ref}
	holds).}
	\label{fig:ex4Mean}
\end{figure}
\end{example}

\begin{example}[Nonconvex stochastic flux]\label{ex:SinusEx}
Consider the Cauchy problem
\begin{equation*}
	du + \partial_x \bigl(u \sin(2\pi u)
	+u\bigr) \circ dW_t = 0,
	\qquad
	\bar{u} = -\sin(\pi x)1_{[-1, 2]}.
\end{equation*}

Figures \ref{fig:ex5Time} and \ref{fig:ex5TimeCF} 
compare several numerical 
approximations produced by the LDG scheme
\eqref{sec4:ldg1}--\eqref{sec4:ldg2} over the time interval 
$[0,\tfrac{1}{4}]$. Snapshots are shown at 
times $t=\tfrac{m}{20}$ for $m=0,\ldots,5$
and compared with a finer numerical solution (black solid line)
computed with $k=0$, $h=7.5 \cdot 10^{-3}$, and
$\Dt \approx 2.10 \cdot 10^{-8}$. All other approximations use
$k=1$, $h=2^{-5}$, and the same time step. More precisely,
Figure \ref{fig:ex5Time} compares the fluxes
\eqref{sec4:pathwiseFlux1_ref}--\eqref{sec4:pathwiseFlux2_ref}
(orange dotted) and \eqref{sec4:numericalFluxes} (red dashed),
both with $\eta_q=7.5$, while the blue dashdotted curve in
Figure \ref{fig:ex5TimeCF} corresponds to
\eqref{sec4:pathwiseFlux1_ref}--\eqref{sec4:pathwiseFlux2_ref}
with the smaller penalty $\eta_q=2.5$, which violates the lower
bound \eqref{sec4:lowerPenalty_ref}. As in
Example \ref{ex:ex4}, the resolution of steep transitions 
depends on the inclusion of a penalty term. The orange dotted
and red dashed curves behave similarly and show only mild
over- and undershoots near local extrema, as is typical for
high-order LDG approximations of nonsmooth profiles. The blue
dashdotted curve is more oscillatory; this is consistent with,
though not solely explained by, the failure of
\eqref{sec4:lowerPenalty_ref}. 

For the fluxes 
\eqref{sec4:pathwiseFlux1_ref}--\eqref{sec4:pathwiseFlux2_ref},
which are covered by Theorem \ref{sec4:pathwiseStability}, the
antiderivative $G(\cdot)$ from \eqref{sec4:G} is
$G(u)
= \int_0^{u}g(s)\, ds
= \frac{1}{4\pi^2}\sin(2\pi u)
+ \frac{u^2}{2}
- \frac{1}{2\pi}u\, \cos(2\pi u)$. 
Numerically, the resulting LDG scheme appears less sensitive to
the CFL constraint \eqref{eq:CFL} than the scheme based on
\eqref{sec4:numericalFluxes}. It is then important to choose
$\eta_q$ so that \eqref{sec4:lowerPenalty_ref} holds. For the
present example it is enough to take $\eta_q \geq 3$, since the
initial data is confined to $[-1,1]$, the exact solution
remains in $[-1,1]$, and 
$\|g''\|_{L^{\infty}([-1, 1])} \leq 35$, so that
$\frac{35}{12} < 3$---see \eqref{sec4:lowerPenalty_ref}.

Figure \ref{fig:ex5Pathwise} shows the evolution 
of the pathwise $L^2_x$ norm for 
the penalty choice $\eta_q=5$. 
The norm decreases rapidly on $[0,0.05]$ 
and then decays more slowly. 
The monotonous decay is exactly as
predicted by Theorem \ref{sec4:pathwiseStability}. 
Additional numerical experiments (not shown) 
also indicate that the estimated
root mean energy decays in agreement with
Theorem \ref{sec4:StabilityTheorem}.

\begin{figure}[t]
	\includegraphics[width=.67\linewidth]
	{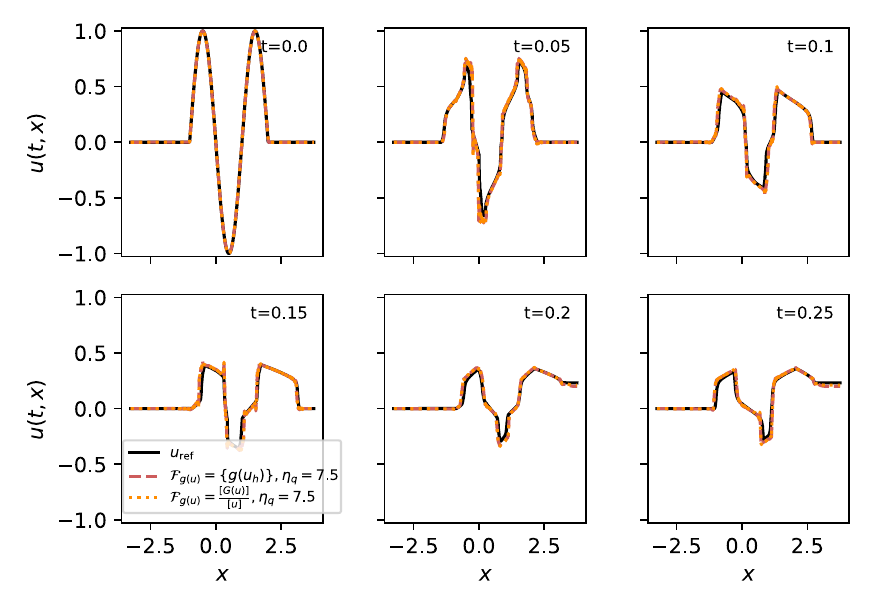}
	\vspace{-0.6cm}
	\captionsetup{width=.975\linewidth}
	\caption{(Evolution nonconvex flux) The subplots compare a finer
	numerical solution (black solid), obtained with $k=0$,
	$h=7.5\cdot 10^{-3}$, and $\Dt \approx 2.10 \cdot 10^{-8}$, to
	two coarser approximations computed with $k=1$, $h=2^{-5}$, and
	the same time step. More precisely, the red dashed curve is
	obtained with the fluxes \eqref{sec4:numericalFluxes} and the
	dotted orange with
	\eqref{sec4:pathwiseFlux1_ref}--\eqref{sec4:pathwiseFlux2_ref},
	both with penalty value $\eta_q = 7.5$. The approximations are
	compared at the times $t=\tfrac{m}{20}$ for $m=0,\ldots,5$.
	Notice that the dotted orange and dashed red curves behave
	similarly, exhibiting mild over- and undershoots near local
	extrema, which is expected for high-order LDG approximations of
	nonsmooth profiles.}
	\label{fig:ex5Time}
\end{figure}

\begin{figure}[t]
	\includegraphics[width=.67\linewidth]
	{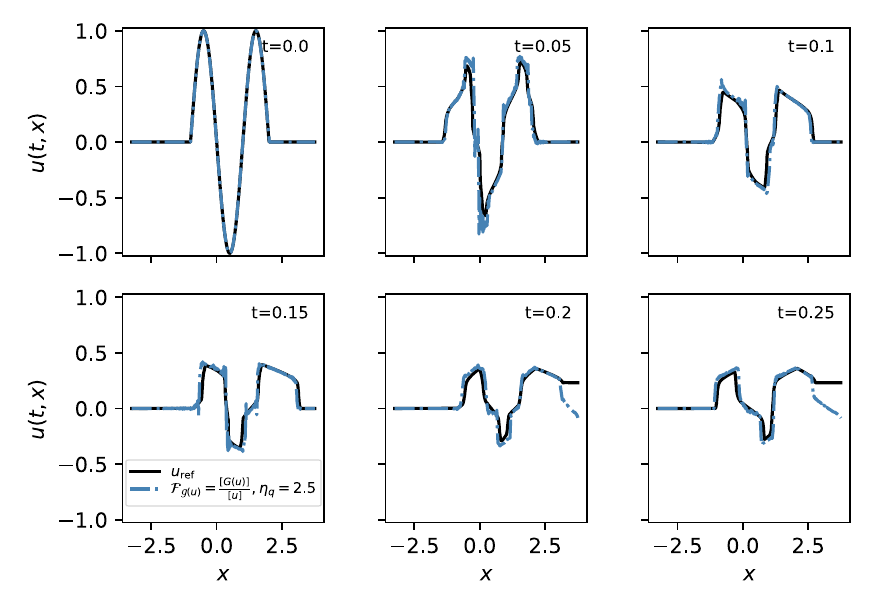}
	\vspace{-0.6cm}
	\captionsetup{width=.975\linewidth}
	\caption{(Evolution nonconvex for violated lower bound) A
	comparison of a finer numerical solution (black solid), obtained
	with $k=0$, $h=7.5\cdot 10^{-3}$, and
	$\Dt \approx 2.10 \cdot 10^{-8}$, to a coarser approximation
	(blue dashdotted) computed with $k=1$, $h=2^{-5}$, the same time
	step, and with the numerical fluxes selected according to
	\eqref{sec4:pathwiseFlux1_ref}--\eqref{sec4:pathwiseFlux2_ref}
	with $\eta_q=2.5$. This penalty value does not satisfy the lower
	bound \eqref{sec4:lowerPenalty_ref}, and one observes visible
	oscillatory behavior, which is consistent with, although not
	solely explained by, the failure of this bound. The solutions are
	compared at the times $t=\tfrac{m}{20}$ for $m=0,\ldots,5$.}
	\label{fig:ex5TimeCF}
\end{figure}

\begin{figure}[t]
	\includegraphics[width=.5\linewidth]
	{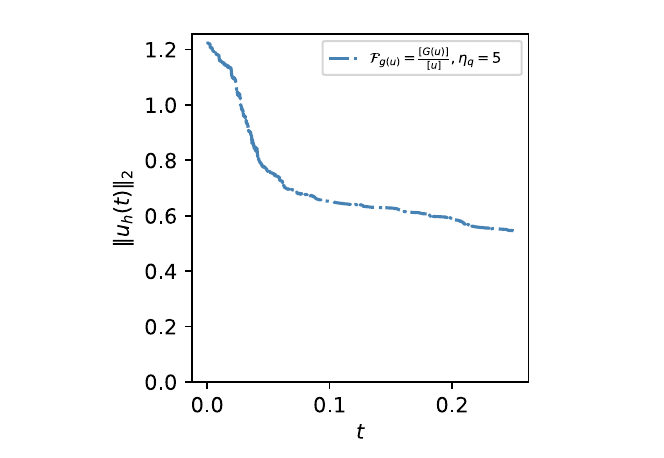}
	\vspace{-0.6cm}
	\captionsetup{width=.975\linewidth}
	\caption{(Pathwise evolution of the $L^2_x$ norm) The plots
	display the evolution of the pathwise $L^2_x$ norm
	$\|u_h(t)\|_2$ for one realization of
	Example \ref{ex:SinusEx} over the time interval
	$[0, \tfrac{1}{4}]$. The flux pair 
	$(\F_{g'(u)q}, \F_{g(u)})$ is chosen according
	to \eqref{sec4:pathwiseFlux1_ref}--\eqref{sec4:pathwiseFlux2_ref}
	with $\eta_q  = 5$. Here the resolutions $h=2^{-3}$ and
	$\Dt=2.17 \cdot 10^{-7}$ are used together with the polynomial degree
	$k=1$. Consistent with Theorem \ref{sec4:pathwiseStability}, the 
	curve decreases monotonically.}
	\label{fig:ex5Pathwise}
\end{figure}
\end{example}

\begin{remark}
The last two nonlinear examples suggest that the penalty
parameters can be tuned to reduce spurious oscillations.
One may also add the damping term from
\cite{Lu:2021aa},
\begin{equation*}
	- \sum_{l=0}^k \frac{\eta_j^l}{h_{K}}
	\int_{K}
	\pigl( u_h - \Pi_{l-1}u_h \pigr)\varphi
	\, dx\, dt,
\end{equation*}
to the primary equation \eqref{sec4:ldg1}, or to
\eqref{eq:weak1} with $K=I_j$. Numerical experiments in
\cite{Li:2025aa, Lu:2021aa} show that this term
efficiently suppresses oscillations, also for stochastic
convection--diffusion problems with nonlinear Itô source.
Here $\Pi_{-1}=\Pi_0$, and the coefficients
$\eta_j^l\geq 0$ are chosen small in smooth regions and
large near discontinuities. For large coefficients,
high-frequency modes are damped, so the method behaves
locally more like a first-order scheme. One possible choice is
\begin{align*}
	\eta_j^l
	:= \frac{2(2l+1)}{2k-1}
	\frac{h^l}{l!}
	\sqrt{
	\llbracket \partial_x^l u_h\rrbracket_{j+\frac{1}{2}}^2
	+ \llbracket \partial_x^l u_h\rrbracket_{j-\frac{1}{2}}^2
	},
	\qquad k \geq 1.
\end{align*}
This damping does not destroy the $L^2$-stability of the
scheme or compromise high-order accuracy 
\cite{Li:2025aa,Lu:2021aa}.
\end{remark}

\begin{example}[Inviscid Buckley--Leverett with stochastic fluxes]
\label{ex:ex7}
As a final nonlinear example, we consider a simplified model
for two-phase flow in a homogeneous medium: the two-dimensional
Buckley--Leverett equation with gravity in the $y$-direction
and stochastic fluxes,
\begin{equation*}
	du
	+ \div_x \bigl(\sigma_1(x,y)g_1(u)\bigr) \circ dW_t^1
	+ \div_x \bigl(\sigma_2(x,y)g_2(u)\bigr) \circ dW_t^2
	=0,
\end{equation*}
for $(t,x)\in [0,\tfrac{3}{10}]
\times [-\tfrac{3}{2},\tfrac{3}{2}]^2$.
Here $\sigma_1=(1,0)$, $\sigma_2=(0,1)$, and
\begin{align*}
	g_1(u) &= \frac{u^2}{u^2+(1-u)^2},
	\qquad 
	g_2(u) = g_1(u)\bigl(1-5(1-u)^2\bigr).
\end{align*}
The initial data is
$u|_{t=0}
=
\begin{cases}
	1, & x^2+y^2 \leq \frac{1}{5}, \\
	0, & \text{otherwise}.
\end{cases}$

This is a challenging benchmark due to the non-convex
fluxes and the sign change of $g_2'(u)$, leading to
nontrivial wave interactions. In the deterministic case,
gravity induces anisotropic deformation of the initial
circular patch, with a composite shock--rarefaction
structure \cite{Karlsen:1998ib}. The stochastic fluxes
enhance this anisotropy, while the Itô--Stratonovich
correction introduces directional smoothing and
accelerates mixing.

Figure \ref{fig:ex7:SBLShort} shows contour plots of one
realization computed with $k=1$, $h=2^{-5}$, and
$\Dt\approx 1.97\cdot 10^{-7}$ at the times
$t=\tfrac{3m}{250}$, $m=0,\dots,5$. The high-saturation
patch disperses more rapidly than in the deterministic
case. This is further illustrated in
Figure \ref{fig:ex7:SBLSurface}, while
Figure \ref{fig:ex7:Cut} shows another realization along
$x=0$, highlighting the composite structure.

The time step is chosen according to the parabolic
CFL constraint
\begin{align*}
	\Dt
	\leq
	\frac{\kappa_{\mathrm{d}}h^2}
	{2(2k+1)^2a_{\mathrm{max}}},
	\qquad
	a_{\max}
	\leq \tfrac{11}{2},
\end{align*}
with $\kappa_{\mathrm{d}}=\tfrac{1}{50}$. The associated
diffusion matrix is
\begin{align*}
	a(u)
	=
	\frac{1}{2}
	\begin{pmatrix}
	(g_1'(u))^2 & 0 \\
	0 &
	\bigl(
	g_1'(u)(1-5(1-u)^2)
	+ 10 g_1(u)(1-u)
	\bigr)^2
	\end{pmatrix},
	\quad 
	g_1'(u)
	=
	\frac{2u(1-u)}
	{(u^2+(1-u)^2)^2}.
\end{align*}

Figure \ref{fig:ex7:WienerPathwise} also shows the
Wiener paths $(W_t^1,W_t^2)$ and the evolution of the
pathwise $L^2_x$ norm for two flux choices. The standard
flux \eqref{sec4:numericalFluxes} yields non-monotone
behavior, while the flux pair from
\eqref{sec4:pathwiseFlux1_ref}--\eqref{sec4:pathwiseFlux2_ref},
consistent with Theorem \ref{sec4:pathwiseStability},
ensures monotone decay.

The latter flux involves 
the integrated quantities
\begin{align*}
	G_1(u)
	&:=
	\int_0^u g_1(\lambda)
	\, d\lambda
	=
	\frac{u}{2}
	+ \frac{1}{4}
	\ln\bigl(2u(u-1)+1\bigr), \\
	G_2(u)
	&:=
	\int_0^u g_2(\lambda)
	\, d\lambda
	=
	-\frac{5\pi}{16}
	+ \frac{7}{4}u
	+ \frac{5}{4}u^2
	- \frac{5}{6}u^3
	+ \frac{1}{4}
	\ln\bigl(2u(u-1)+1\bigr)
	- \frac{5}{4}\arctan(2u-1).
\end{align*}
Numerically, this choice is 
more robust with respect to
the time step. Moreover,
\eqref{sec4:lowerPenalty_ref} yields the bounds
\begin{equation}\label{eq:SBLLower}
	\eta_{q,1} \geq \frac{1}{2},
	\qquad
	\eta_{q,2} \geq \frac{33}{24}.
\end{equation}

\begin{figure}[t]
	\includegraphics[width=.67\linewidth]
	{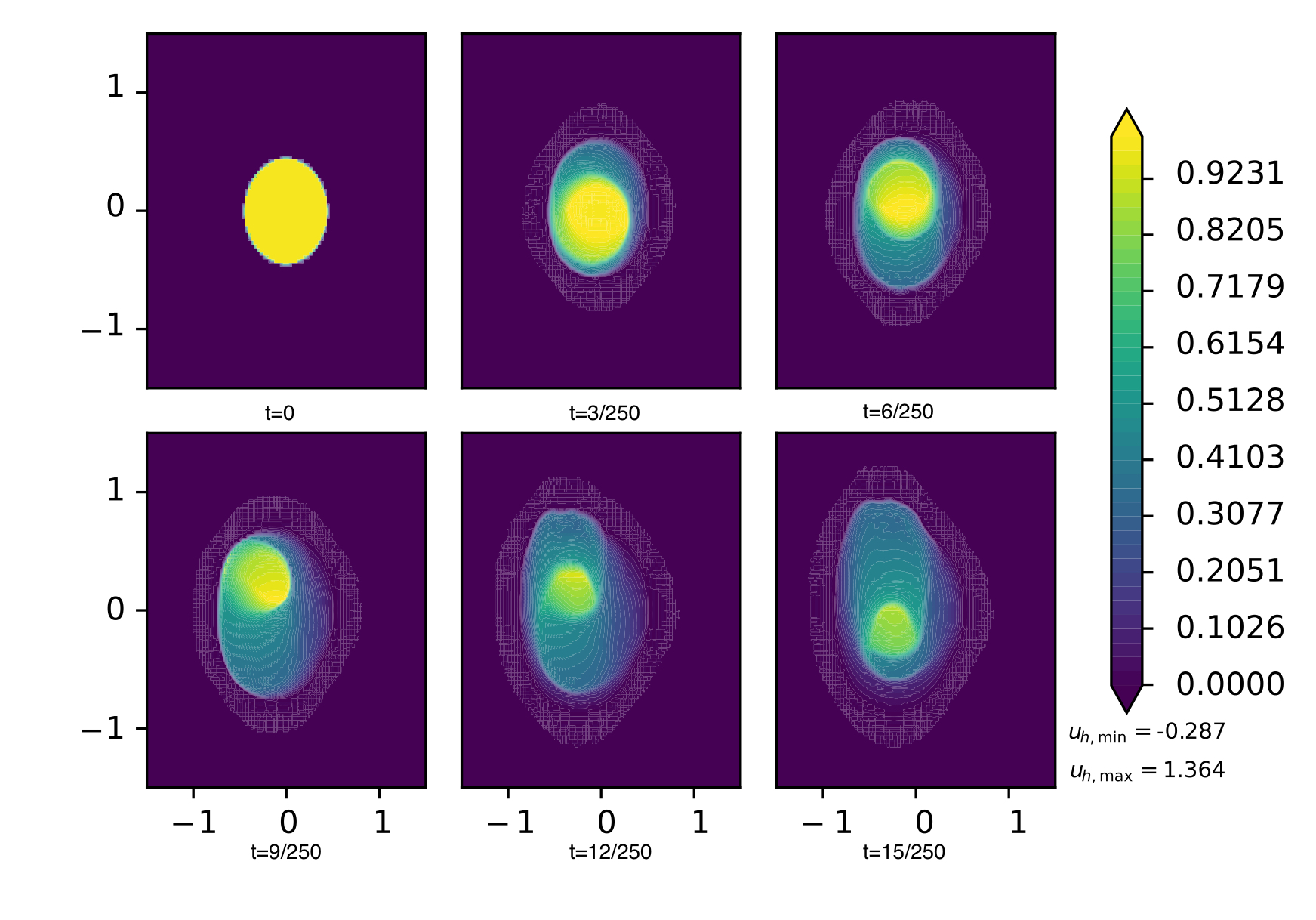}
	\vspace{-0.6cm}
	\captionsetup{width=.975\linewidth}
	\caption{Contour plots of a single realization of the
	approximation to Example~\ref{ex:ex7}, shown at the
	times $t=\tfrac{3m}{250}$ for $m=0,\ldots,5$.
	The solution is computed with $k=1$, penalty values
	$(\eta_{q_1},\eta_{q_2})=(10,10)$, and uniform
	resolutions $h=2^{-5}$ and
	$\Dt\approx 1.97\cdot 10^{-7}$ on a quadrilateral mesh.
	The initially saturated circular patch is rapidly dispersed
	over the time interval $[0,\tfrac{3}{250}]$.}
	\label{fig:ex7:SBLShort}
\end{figure}

\begin{figure}[t]
	\includegraphics[width=.67\linewidth]
	{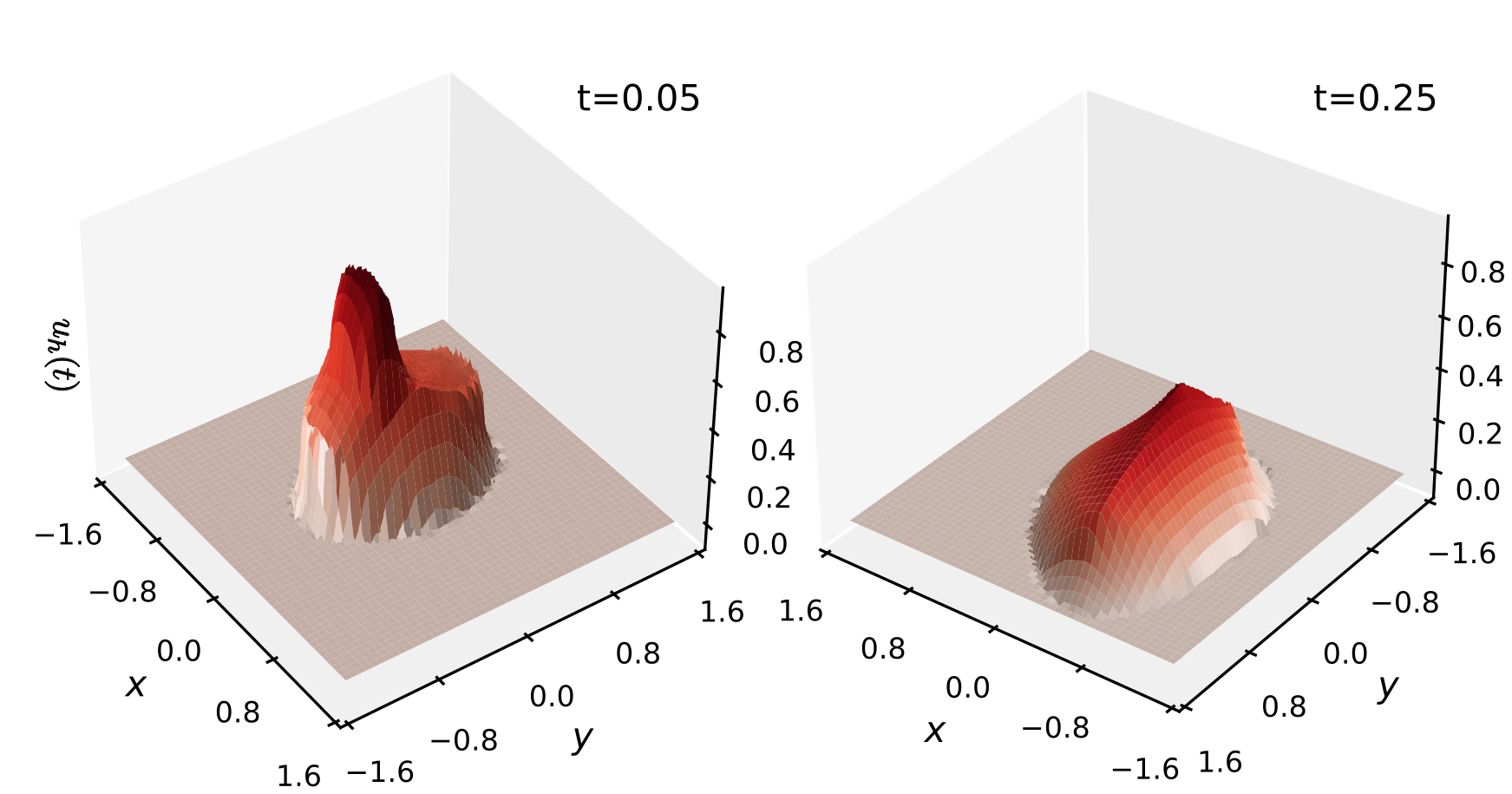}
	\vspace{-0.4cm}
	\captionsetup{width=.975\linewidth}
	\caption{Surface plots of the numerical solution of
	Example~\ref{ex:ex7} at the times $t=\tfrac{1}{20}$
	(left) and $t=\tfrac{1}{4}$ (right). The approximation
	is computed with $k=1$, penalty values
	$(\eta_{q_1},\eta_{q_2})=(10,10)$, and uniform
	resolutions $h=2^{-5}$ and
	$\Dt\approx 1.97\cdot 10^{-7}$ on a quadrilateral mesh.
	The plots show the rapid dispersion and anisotropic
	deformation of the initially saturated circular patch.}
	\label{fig:ex7:SBLSurface}
\end{figure}

\begin{figure}[t]
	\includegraphics[width=.67\linewidth]
	{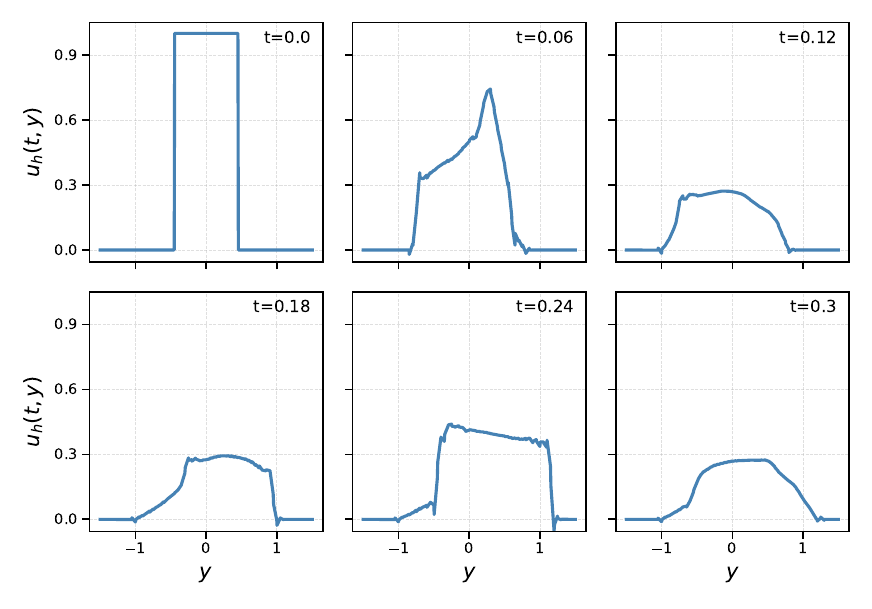}
	\vspace{-0.6cm}
	\captionsetup{width=.975\linewidth}
	\caption{Numerical solution of Example~\ref{ex:ex7}
	along the cut $x=0$, shown at the times
	$t=\tfrac{3m}{50}$ for $m=0,\ldots,5$. The approximation
	is computed with $k=1$, penalty values
	$(\eta_{q_1},\eta_{q_2})=(10,10)$, and uniform
	resolutions $h=5\cdot 10^{-2}$ and
	$\Dt\approx 5.05\cdot 10^{-7}$ on a quadrilateral mesh.
	The profiles display the composite shock--rarefaction
	structure of the solution. Mild over- and undershoots occur
	near steep fronts; these are typical for LDG
	discretizations of nonsmooth data, and here no additional
	limiting is used.}
	\label{fig:ex7:Cut}
\end{figure}

\begin{figure}[t]
	\centering
	\includegraphics[width=.67\linewidth]
	{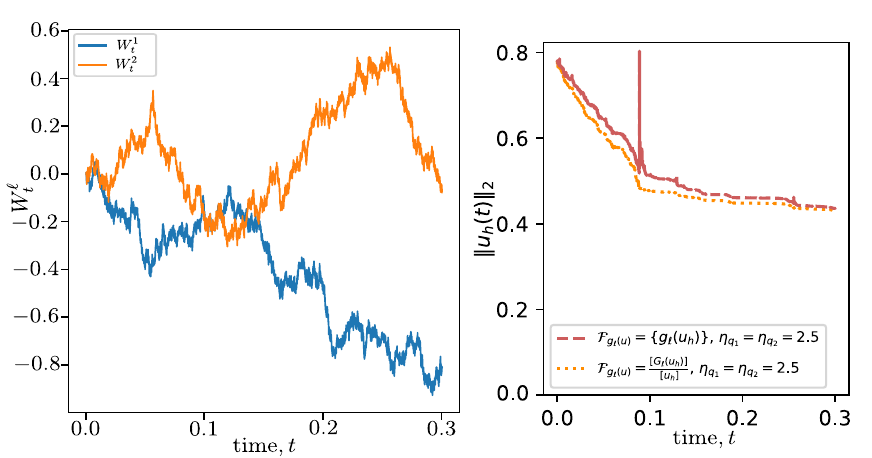}
	\vspace{-0.4cm}
	\captionsetup{width=.975\linewidth}
	\caption{Wiener paths and pathwise $L^2_x$ norm for
	Example~\ref{ex:ex7}. The left plot shows the realizations
	of $(W_t^1,W_t^2)$ used for the numerical solution in
	Figure \ref{fig:ex7:Cut}. The right plot shows the
	pathwise evolution of $\|u_h(t)\|_2$ over
	$[0,\tfrac{3}{10}]$ for two choices of numerical fluxes:
	\eqref{sec4:numericalFluxes} (red dashed) and
	\eqref{sec4:pathwiseFlux1_ref}--\eqref{sec4:pathwiseFlux2_ref}
	(orange dotted). Both approximations are computed with
	$\eta_{q_1}=\eta_{q_2}=2.5$, $\eta_u=0$, $k=1$,
	$h=2^{-4}$, and $\Dt\approx 3.95\cdot 10^{-7}$.
	The orange dotted curve satisfies \eqref{eq:SBLLower}
	and is therefore monotonically decreasing by
	Theorem \ref{sec4:pathwiseStability}. No such pathwise
	monotonicity result applies to the red dashed curve,
	which exhibits a sudden increase.}
	\label{fig:ex7:WienerPathwise}
\end{figure}
\end{example}

\subsection{Kraichnan-like turbulence model}
\label{subsec:Kraichnan}

We next approximate a Kraichnan-type model for the advection of
a passive scalar by a Gaussian velocity field that is white
in time and colored in space 
\cite{Kraichnan1967,Kraichnan1994}. 
For background material, see 
\cite{Eyink:2000aa,Le-Jan:2002aa,Gawedzki:2008aa,
FlandoliLuongo2023} and \cite[Sec.~2]{Coghi:2023aa}.

We consider the stochastic transport equation on the
two-dimensional torus $\mathbb{T}=[0,2\pi]^2$,
\begin{equation}\label{sec6:kraichnan}
	du
	+ \sum_{\ell=1}^L
	\sigma_{\ell}(x)\cdot \nabla u
	\circ dW_t^{\ell}
	=0,
\end{equation}
where the deterministic amplitudes
$\sigma_{\ell}:\mathbb{T}\to \R^2$ are divergence-free.
We write
$\sigma_{\ell}
=(\sigma_{\ell,1},\sigma_{\ell,2})$
for their components. Formally, \eqref{sec6:kraichnan}
corresponds to advection by the incompressible Gaussian
velocity field
$$
v(t,x)
=\sum_{\ell=1}^L
\sigma_{\ell}(x)\dot W_t^{\ell}.
$$
Its covariance is
\begin{equation*}
	\E \bigl[v_i(t,x)v_j(s,y)\bigr]
	=
	\delta(t-s)C_{ij}(x,y),
	\qquad
	C_{ij}(x,y)
	:=
	\sum_{\ell=1}^L
	\sigma_{\ell,i}(x)\sigma_{\ell,j}(y).
\end{equation*}

For the numerical experiments, we use a finite Fourier
approximation of this covariance. The amplitudes are chosen
so that, in the active range of scales, the velocity increments
satisfy the Kraichnan-type scaling
\begin{equation}\label{eq:holderExp}
	\E\bigl[
	|v(t,x+\delta)-v(t,x)|^2
	\bigr]
	\sim |\delta|^{\xi},
	\qquad
	\xi\in(0,2).
\end{equation}
This corresponds heuristically to spatial 
H\"older regularity
$C^\beta$ for every $\beta<\tfrac{\xi}{2}$.
The rough regime is the case $\beta<\tfrac{1}{2}$,
where the velocity is not differentiable in space

We fix a cutoff $k^*$, which determines the smallest active
length scale $\ell_{\min}=1/k^*$. For spatial increments
$\delta\in\mathbb{R}^2$ with $|\delta|\lesssim \ell_{\min}$,
the truncated field is smooth, whereas for
$|\delta|\gg \ell_{\min}$ it displays the prescribed rough
scaling. On an $N\times N$ quadrilateral mesh, we choose
$k^*\lesssim N/3$ so that the smallest active scales remain
resolved. Equivalently, we take
\begin{equation}\label{sec6:explicitCutoff}
	\frac{1}{k^*}
	=
	3h.
\end{equation}

Define the active set of wave vectors by
\begin{align*}
	\mathcal{K}
	:=
	\{k\in\mathbb{Z}^2\setminus\{0\}
	\,:\,
	1\leq |k|\leq k^*\},
\end{align*}
and let $M=|\mathcal{K}|$. For each
$k^{(m)}=(k_1^{(m)},k_2^{(m)})\in\mathcal{K}$,
we introduce two modes, one cosine and one sine, so that
$L=2M$. More precisely, for $m=1,\ldots,M$, we set
\begin{equation*}
	k_{\ell}
	:=
	k^{(m)},
	\qquad
	\ell\in\{2m-1,2m\}.
\end{equation*}
For $k=(k_1,k_2)$, write $k^\perp:=(-k_2,k_1)$.
The divergence-free amplitudes are then defined by
\begin{align}\label{sec6:kraichnanSigma}
	\sigma_{\ell}(x)
	=
	a_{\ell}
	\frac{k_{\ell}^{\perp}}{|k_{\ell}|}
	\phi_{\ell}(k_{\ell}\cdot x),
	\qquad
	\ell=1,\ldots,2M,
\end{align}
where
\begin{align*}
	\phi_{\ell}(s)
	&:=
	\begin{cases}
	\cos(s), & \ell \text{ is odd}, \\
	\sin(s), & \ell \text{ is even},
	\end{cases}
	\qquad
	a_{\ell}:=\sqrt{D}\,
	|k_{\ell}|^{-(1+\frac{\xi}{2})}.
\end{align*}
The factor $k_{\ell}^{\perp}/|k_{\ell}|$ 
ensures that $\nabla\cdot\sigma_{\ell}=0$.

We normalize the spatial covariance density 
by prescribing the formal mean-square velocity amplitude
$\E[|v(t,x)|^2]=V_0^2$. More precisely, 
this is understood as $\sum_{i=1}^2 C_{ii}(x,x)
=\sum_{\ell=1}^L |\sigma_\ell(x)|^2
=V_0^2$.
Here the second equality follows from the definition of
$C_{ij}$. With the sine--cosine pairing above, the two modes
associated with each $k\in\mathcal{K}$ combine to give the
$x$-independent contribution $D|k|^{-(2+\xi)}$. 
Summing over the active modes 
$k\in \mathcal{K}$ therefore gives 
$D=\frac{V_0^2}{\sum_{k\in\mathcal{K}}
|k|^{-(2+\xi)}}$. 
In the computations, we take $V_0=1$.

\begin{example}[A single initial Fourier mode in rough Kraichnan noise]
\label{ex:Kraichnan}
We now study \eqref{sec6:kraichnan} with the noise amplitudes 
$\sigma_{\ell}$ given by \eqref{sec6:kraichnanSigma}. 
The initial data is the single low-frequency Fourier mode
\begin{equation*}
	\bar{u}(x,y)=\sin(x)\sin(y).
\end{equation*}
In what follows, we take $\xi=1$ in \eqref{eq:holderExp}. For
$h=2^{-3}$, the cutoff in \eqref{sec6:explicitCutoff} gives
$k^*=3$. Hence the active set $\mathcal{K}$ has cardinality
$M=28$, and the equation is driven by $L=2M=56$ noise modes.

Figure \ref{fig:ex8:KraichnanContour} shows the numerical
solution on $[0,1]$, computed with $k=2$ and
$\Dt\approx 1.04\cdot 10^{-5}$, at the times
$t=\tfrac{m}{5}$ for $m=0,\ldots,5$. Compared with the
more regular setting in Figure \ref{fig:Ex5H1RegularSigma},
the Kraichnan velocity field produces a much richer
interaction structure and highly irregular flow patterns. 
Figure \ref{fig:ex8:KraichnanCut} shows the same realization 
along the cut $x=0$. Although no closed-form solution is 
available, the approximation remains stable and physically 
plausible. The cut also shows the development of increasingly 
fine-scale spatial structure as time evolves. 
Finally, Figure \ref{fig:ex8:KraichnanBMs} shows six of the
$56$ driving Wiener processes. Four of these correspond to
wave vectors $k\in\mathcal{K}$ with $|k|=1$. Together, the
$56$ independent processes generate an intricate velocity
field and produce the irregular transport patterns observed
above.

Although not shown, the numerical experiments also confirm
the stability result from Theorem \ref{sec4:StabilityTheorem}.
When all flux pairs are central, the pathwise $L^2_x$ norm is
preserved. For the other flux choices in
\eqref{sec4:numericalFluxes}, the root mean energy
$L^2_{\omega,x}$ is nonincreasing, as predicted by the theorem.

\begin{figure}[t]
	\includegraphics[width=.67\linewidth]
	{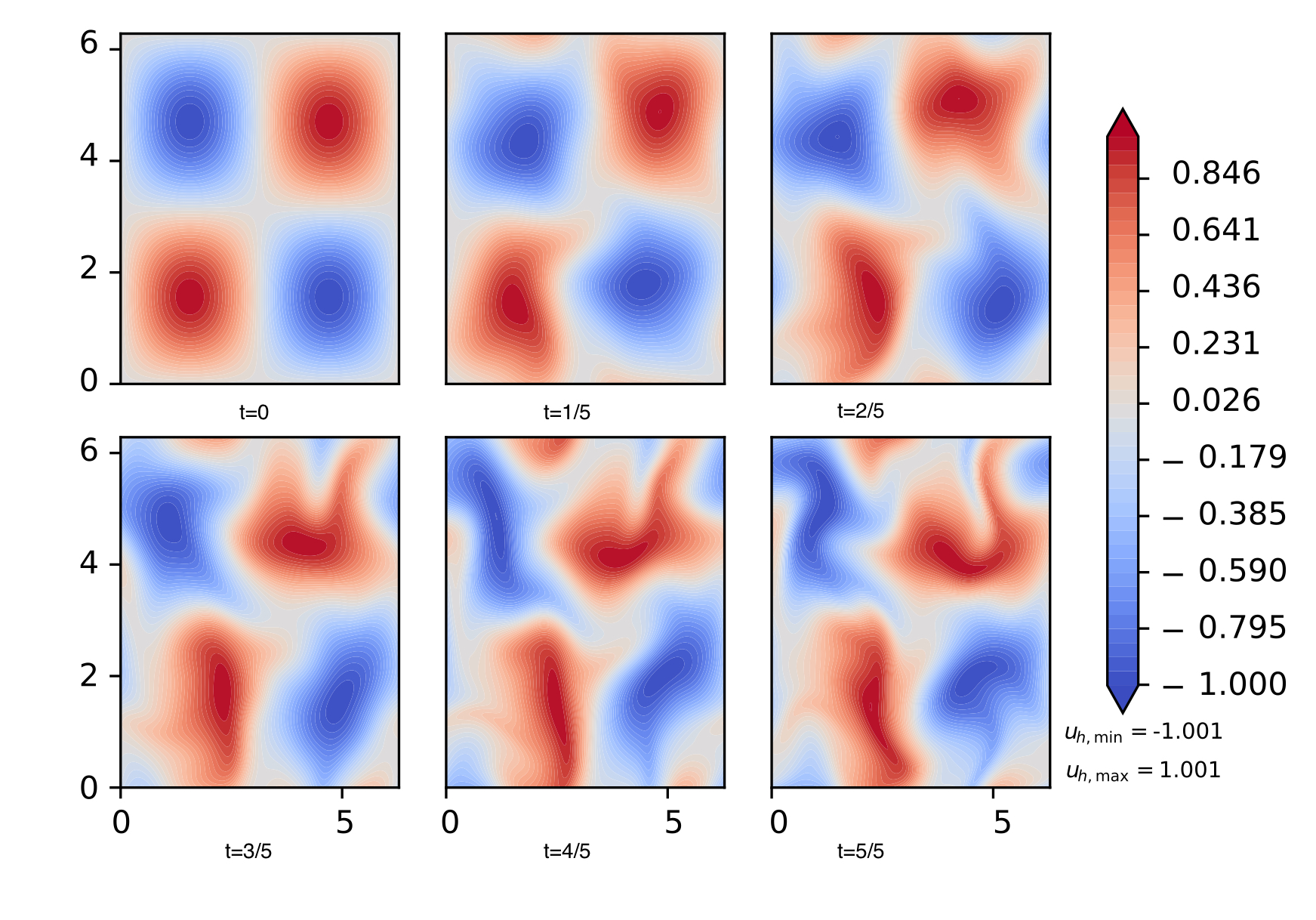}
	\vspace{-0.6cm}
	\captionsetup{width=.975\linewidth}
	\caption{(Kraichnan-like model) 
	Snapshots of the numerical solution of
	Example~\ref{ex:Kraichnan} at the times
	$t=\tfrac{m}{5}$ for $m=0,\ldots,5$. The approximation
	is computed with polynomial degree $k=2$, uniform
	resolutions $h=2^{-3}$ and
	$\Dt\approx 1.04\cdot 10^{-5}$, and penalty values
	$\eta_{q_{\ell}}=2$ for all $\ell$. The problem is driven
	by $L=56$ Wiener processes. The Kraichnan velocity field
	produces irregular flow patterns from the initial
	low-frequency Fourier mode.}
	\label{fig:ex8:KraichnanContour}
\end{figure}

\begin{figure}[t]
	\includegraphics[width=.67\linewidth]
	{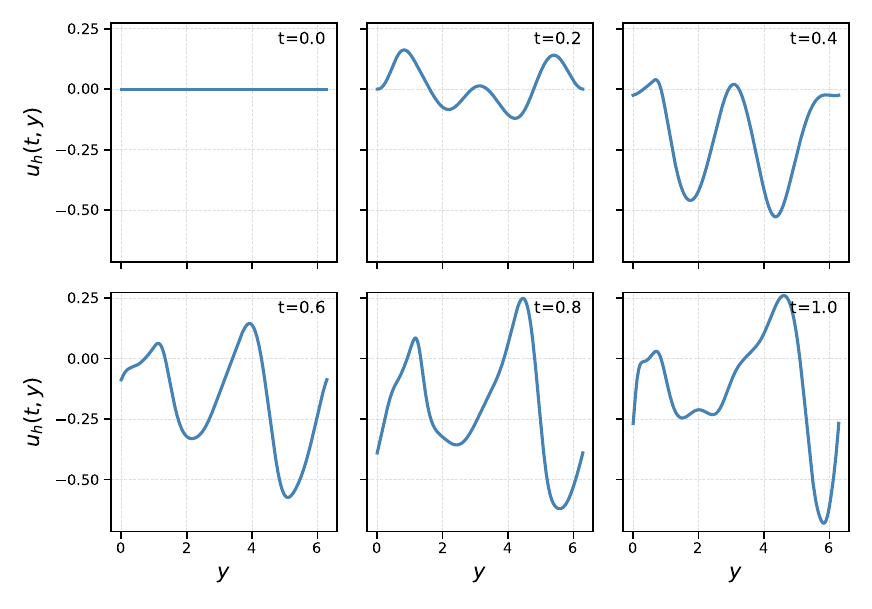}
	\vspace{-0.6cm}
	\captionsetup{width=.975\linewidth}
	\caption{(Kraichnan-like model) 
	Snapshots of the numerical solution from
	Example~\ref{ex:Kraichnan} along the cut $x=0$, shown at
	the times $t=\tfrac{m}{5}$ for $m=0,\ldots,5$. The
	approximation is computed on a quadrilateral mesh with
	$k=2$, $h=2^{-3}$,
	$\Dt\approx 1.04\cdot 10^{-5}$, and
	$\eta_{q_{\ell}}=2$ for all $\ell$. 
	The solution along the cut develops 
	fine-scale spatial 
	structure as time evolves.}
	\label{fig:ex8:KraichnanCut}
\end{figure}

\begin{figure}[t]
	\includegraphics[width=.5\linewidth]
	{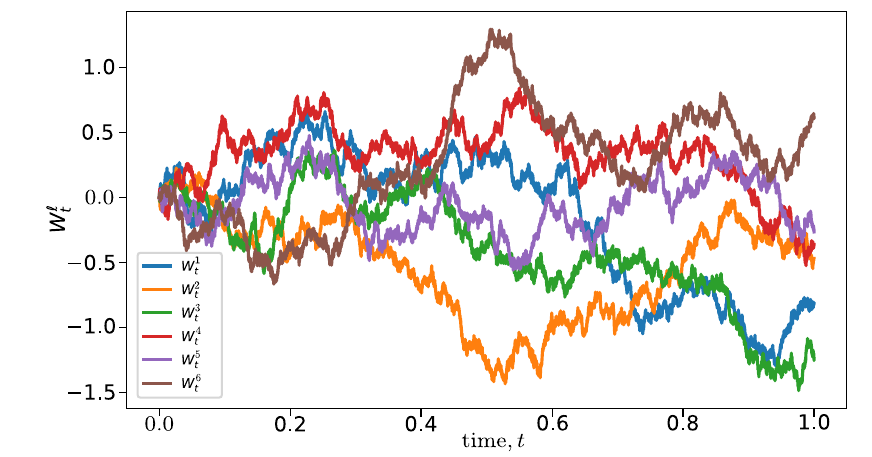}
	\vspace{-0.4cm}
	\captionsetup{width=.975\linewidth}
	\caption{(Kraichnan-like model) 
	Six of the $56$ Wiener processes driving
	Example~\ref{ex:Kraichnan}. Four of the displayed
	processes correspond to wave vectors $k\in\mathcal{K}$
	with $|k|=1$. Together, these processes drive the solution
	in different spatial directions and generate the irregular
	velocity field.}
	\label{fig:ex8:KraichnanBMs}
\end{figure}
\end{example}



\appendix

\section{Local well-posedness of the SDE system}
\label{sec:appWell}

We prove that the coefficients in \eqref{sec4:SDE-system}
are locally Lipschitz and locally of linear growth on
$\R^{(N_k+1)\times|\T_h|}$. Throughout, $\|\cdot\|$
denotes the Euclidean norm on finite-dimensional vectors,
while $\|\cdot\|_F$ denotes the Frobenius norm on coefficient
matrices. Since the state space is finite-dimensional, it is
enough to work on bounded sets. Fix $R>0$ and set
$B_R:=\{\bold{u}\in\R^{(N_k+1)\times|\T_h|}:
\|\bold{u}\|_F\le R\}$. For $K\in\T_h$, write
$U^K=\bold{u}^K\cdot\boldsymbol{\phi}^K$ and
$V^K=\bold{v}^K\cdot\boldsymbol{\phi}^K$. By
finite-dimensional norm equivalence on each element, there is
a mesh-dependent constant $C_h>0$ such that, for all
$\bold{u},\bold{v}\in B_R$,
$$
\|U^K\|_{L^\infty(K)}
+\|V^K\|_{L^\infty(K)}
\le C_hR,
\qquad
\|U^K-V^K\|_{L^\infty(K)}
\le C_h\|\bold{u}^K-\bold{v}^K\|.
$$
The same estimates hold for traces on each face
$e\subset\partial K$. Hence all arguments of $g_{\ell}$,
$g_{\ell}'$, $\F_{\ell}$, and $\widetilde{\F}_{\ell}$
remain in a compact interval depending only on $R$ and $h$.

We first consider $\bold{Q}_{\ell}^K$. Since
$g_{\ell}\in C^1_{\mathrm{loc}}(\R)$ and
$\widetilde{\F}_{\ell}$ is locally Lipschitz, the definition
\eqref{sec4:Q} gives
$$
\|\bold{Q}_{\ell}^K(\bold{u})
-\bold{Q}_{\ell}^K(\bold{v})\|
\le C_{R,h,\ell}
\left(
\|\bold{u}^K-\bold{v}^K\|
+\sum_{e\subset\partial K}
\|\bold{u}^{K_e}-\bold{v}^{K_e}\|
\right).
$$
Here and below $C_{R,h,\ell}$ denotes a finite constant
depending on $R$, the mesh, the basis functions, the traces of
$\sigma_{\ell}\cdot n$, and the local Lipschitz constants of
the nonlinearities and numerical fluxes. Boundary faces are
handled using the convention in \eqref{sec4:numericalFluxesBD}.

Similarly, the polynomial growth of $g_{\ell}$ and the local
linear growth of $\widetilde{\F}_{\ell}$ yield
$$
\|\bold{Q}_{\ell}^K(\bold{u})\|
\le C_{R,h,\ell}
\left(
1+\|\bold{u}^K\|
+\sum_{e\subset\partial K}\|\bold{u}^{K_e}\|
\right),
\qquad \bold{u}\in B_R.
$$
Indeed, on $B_R$, every polynomial growth term is bounded by
a constant times $1+\|\bold{u}^K\|$, where the constant may
depend on $R$.

The estimates for $G_{\ell}$ follow directly from
\eqref{sec4:GL}. Each entry of $G_{\ell}$ is a fixed linear
functional of $\bold{Q}_{\ell}^K(\bold{u})$, and the mesh is
finite. Hence the two estimates above imply
$$
\|G_{\ell}(\bold{u})-G_{\ell}(\bold{v})\|_F
\le C_{R,h,\ell}\|\bold{u}-\bold{v}\|_F,
\qquad
\|G_{\ell}(\bold{u})\|_F
\le C_{R,h,\ell}\bigl(1+\|\bold{u}\|_F\bigr),
\qquad \qquad \bold{u},\bold{v}\in B_R.
$$

It remains to treat the drift. By \eqref{sec4:FL}, each entry
of $F_{\ell}$ is a finite sum of terms involving
$g_{\ell}'(U^K)$, $\bold{Q}_{\ell}^K(\bold{u})$, and
$\F_{\ell}$ evaluated at
$U^{K_e}$, $U^K$,
$\bold{Q}_{\ell}^{K_e}(\bold{u})\cdot
\boldsymbol{\phi}^{K_e}$, and
$\bold{Q}_{\ell}^{K}(\bold{u})\cdot
\boldsymbol{\phi}^{K}$. On $B_R$, the maps
$\bold{u}\mapsto U^K$ and $\bold{u}\mapsto U^{K_e}$ are
linear and bounded, while the estimate for
$\bold{Q}_{\ell}^K$ above shows that
$\bold{u}\mapsto \bold{Q}_{\ell}^K(\bold{u})$ is locally
Lipschitz. Since $g_{\ell}'$ and $\F_{\ell}$ are locally
Lipschitz, each entry of $F_{\ell}$ is locally Lipschitz on
$B_R$. Thus
$$
\|F_{\ell}(\bold{u})-F_{\ell}(\bold{v})\|_F
\le C_{R,h,\ell}\|\bold{u}-\bold{v}\|_F,
\qquad \bold{u},\bold{v}\in B_R.
$$
The same representation, together with the growth estimate
for $\bold{Q}_{\ell}^K$, the polynomial growth of
$g_{\ell}'$, and the local linear growth of $\F_{\ell}$,
gives
$$
\|F_{\ell}(\bold{u})\|_F
\le C_{R,h,\ell}\bigl(1+\|\bold{u}\|_F\bigr),
\qquad \bold{u}\in B_R.
$$
Since $L$ is finite, the same local Lipschitz and local growth
bounds hold for $F=\sum_{\ell\in L}F_{\ell}$ and also 
for the family $\{G_{\ell}\}_{\ell\in L}$.

We now conclude the local well-posedness of 
the SDE \eqref{sec4:SDE-system} by localization. 
Let $\chi_R$ be a smooth
cutoff on $\R^{(N_k+1)\times|\T_h|}$ such that
$\chi_R=1$ on $B_R$ and $\chi_R=0$ outside $B_{R+1}$.
Set $F_R:=\chi_RF$ and $G_{\ell,R}:=\chi_RG_{\ell}$.
The estimates above imply that these truncated coefficients
are globally Lipschitz and of linear growth. Therefore
Theorem \ref{thm:SDEStandard} gives a unique global strong
solution $\bold{u}_R$ of the truncated SDE system
$d\bold{u}_R=F_R(\bold{u}_R)\,dt
+\sum_{\ell\in L}G_{\ell,R}(\bold{u}_R)\,dW_t^\ell$. 
Define the stopping time $\tau_R$ by 
\eqref{eq:SDE-stopping-time}. 
If $R<R'$, then $\bold{u}_R$ and $\bold{u}_{R'}$ agree on
$[0,\tau_R]$ by pathwise uniqueness. Hence the stopped
solutions can be pasted together to obtain a unique maximal
local strong solution $\bold{u}$ of \eqref{sec4:SDE-system}
on $[0,\tau_{\max})$, where
$\tau_{\max}:=\lim_{R\to\infty}\tau_R$. Moreover, the usual
explosion alternative holds:
$\limsup_{t\uparrow\tau_{\max}}
\|\bold{u}(t)\|_F=\infty$ 
on $\{\tau_{\max}<T\}$. 
Finally, the standard moment estimate 
for the truncated SDEs yields \eqref{eq:mean-stopped}. 
This proves local well-posedness 
of \eqref{sec4:SDE-system},
up to indistinguishability.

\medskip

\end{document}